\documentclass{amsart}

\usepackage{amsmath}
\usepackage{amsthm}
\usepackage{amssymb}
\usepackage{paralist}

\usepackage{enumitem}  
\usepackage{stmaryrd}

\usepackage{stmaryrd}
\usepackage{textcomp}

\setenumerate{label={\normalfont(\textit{\roman*})},ref={\textrm{\roman*}}}

\usepackage[usenames,dvipsnames]{color}
\usepackage{ifthen}
\usepackage{scalerel}
\usepackage{hyperref}
\usepackage{mathrsfs}
\usepackage[mathscr]{euscript}

\newtheorem{theorem}{Theorem}[section]
\newtheorem{proposition}[theorem]{Proposition}
\newtheorem{lemma}[theorem]{Lemma}
\newtheorem{corollary}[theorem]{Corollary}

\newtheorem{alphatheorem}{Theorem}

\theoremstyle{definition}

\newtheorem{example}[theorem]{Example}
\newtheorem{remark}[theorem]{Remark}

\newtheorem{question}[theorem]{Question}

\newcommand{\algcl}{\operatorname{alg.cl}}
\DeclareMathOperator*{\alglim}{alg.cl}
\newcommand{\topcl}{\operatorname{top.cl}}
\newcommand{\cl}{\operatorname{cl}}
\newcommand{\inter}{\operatorname{int}}

\newcommand{\linspan}{\operatorname{span}}
\newcommand{\SL}{\operatorname{SL}}
\newcommand{\AGL}{\operatorname{AGL}}
\newcommand{\GL}{\operatorname{GL}}
\newcommand{\Aff}{\operatorname{Aff}}
\newcommand{\Aut}{\operatorname{Aut}}
\newcommand{\End}{\operatorname{End}}
\newcommand{\IP}{\mathrm{IP}}

\newcommand{\KN}{\mathrm{K}(\beta \NN)}
\newcommand{\EN}{\mathrm{E}(\beta \NN)}

\newcommand{\Alpha}{\mathrm{A}}
\newcommand{\Rho}{\mathrm{P}}
\newcommand{\Zeta}{\mathrm{Z}}

\renewcommand{\Xi}{\Delta}

\newcommand{\norm}[1]{\left\lVert #1 \right\rVert}

\newcommand{\fp}[1]{\left\{ #1 \right\} }
\newcommand{\ip}[1]{\left[ #1 \right] }

\newcommand{\bra}[1]{\left( #1 \right)}
\newcommand{\braif}[1]{\left\llbracket #1 \right\rrbracket}

\renewcommand{\tilde}{\widetilde}
\renewcommand{\bar}{\overline}
\newcommand{\bsp}[1]{\left< #1 \right>}

\newcommand{\abs}[1]{\left|#1\right|}
\newcommand{\set}[2]{\left\{ #1 \ \middle| \ #2 \right\} }
\newcommand{\grp}[2]{\left< #1 \ \middle| \ #2 \right> }

\newcommand{\ceil}[1]{\left\lceil #1 \right\rceil}

\newcommand{\e}{\varepsilon}
\renewcommand{\a}{\alpha}

\newcommand{\NN}{\mathbb{N}}
\newcommand{\KK}{\mathbb{K}}
\newcommand{\QQ}{\mathbb{Q}}

\newcommand{\ZZ}{\mathbb{Z}}
\newcommand{\RR}{\mathbb{R}}

\newcommand{\cB}{\mathscr{B}}
\newcommand{\cC}{\mathcal{C}}
\newcommand{\cG}{\mathcal{G}}
\newcommand{\cH}{\mathcal{H}}
\newcommand{\cD}{\mathcal{D}}
\newcommand{\cE}{\mathcal{E}}

\newcommand{\cS}{\mathcal{S}}
\newcommand{\cR}{\mathcal{R}}
\newcommand{\cT}{\mathcal{T}}
\newcommand{\cI}{\mathcal{I}}
\newcommand{\cN}{\mathcal{N}}

\newcommand{\Ball}{\mathrm{B}}

\newcommand{\fI}{\mathsf{I}}
\newcommand{\cP}{\mathcal{P}}

\newcommand{\cU}{\mathcal{U}}
\newcommand{\cV}{\mathcal{V}}
\newcommand{\fV}{\mathsf{V}}
\newcommand{\fS}{\mathsf{S}}

\newcommand{\cW}{\mathcal{W}}
\newcommand{\fg}{\mathfrak{g}}
\newcommand{\fh}{\mathfrak{h}}

\definecolor{White}{gray}{.35}
\definecolor{LightGray}{gray}{0.8}

\definecolor{BurntOrange}{RGB}{179, 134, 0}

\newcommand{\bb}{\mathbf}

\newcommand{\complexity}{\operatorname{cmp}}
\newcommand{\eqab}{\underset{\mathrm{ab}}{\equiv}}
\newcommand{\s}{D}
\newcommand{\HH}{\mathscr{H}}
\newcommand{\GG}{\mathscr{G}}

\newcommand{\tcH}{\Rho}
\newcommand{\tcG}{\tilde\Sigma{}}
\renewcommand{\cH}{\Pi}
\renewcommand{\cG}{\Sigma}
\newcommand{\bA}{\bar{A}{}}

\newcommand{\fpp}[1]{\text{\text{\textlquill}} #1 \text{\text{\textrquill}}}

\newcommand{\ST}{\mathrm{M}}
\newcommand{\st}{\mathfrak{m}}

\newcommand{\STp}{\ST'}
\newcommand{\bSTp}{\bar\ST{}'}
\newcommand{\bST}{\bar{\ST}}

\newcommand{\stp}{\st'}
\newcommand{\bstp}{\bar\st{}'}

\newcommand{\Vp}{V^{\ast}}

\renewcommand{\subset}{\subseteq}

\newtheorem{innercustomthm}{Theorem}
\newenvironment{customthm}[1]
  {\renewcommand\theinnercustomthm{#1}\innercustomthm}
  {\endinnercustomthm}

\renewcommand{\color}[1]{}
\renewcommand{\checkmark}[1]{}

\begin{document}

\author[J.\ Konieczny]{Jakub Konieczny}
\address[J.\ Konieczny]{Camille Jordan Institute, 
Claude Bernard University Lyon 1,
43 Boulevard du 11 novembre 1918,
69622 Villeurbanne Cedex, France}
\address{Faculty of Mathematics and Computer Science, Jagiellonian University in Krak\'{o}w, \L{}ojasiewicza 6, 30-348 Krak\'{o}w, Poland\newline}
\email{jakub.konieczny@gmail.com}

\title{Generalised polynomials and integer powers}

\begin{abstract}
	We show that there does not exist a generalised polynomial which vanishes precisely on the set of powers of two. In fact, if $k \geq 2$ is an integer and $g \colon \mathbb{N} \to \mathbb{R}$ is a generalised polynomial such that $g(k^n) = 0$ for all $n \geq 0$ then there exists infinitely many $m \in \mathbb{N}$, not divisible by $k$, such that $g(mk^n) = 0$ for some $n \geq 0$. As a consequence, we obtain a complete characterisation of sequences which are simultaneously automatic and generalised polynomial.
\end{abstract}


\maketitle

\section{\checkmark\ Introduction}\label{sec:intro}

\subsection{\checkmark\ Background}\label{ssec:intro-background}\mbox{}
{
Combinatorial properties of generalised polynomials --- that is, expressions build up from ordinary polynomials with the use of addition, multiplication and the floor function --- have long been studied. The following fundamental result due to Bergelson and Leibman connects generalised polynomials to nilpotent dynamics. (For the terminology used, see Section \ref{sec:prelims}.)
}
\begin{theorem}[\cite{BergelsonLeibman-2007}]\label{thm:BL-intro}
	Let $g \colon \ZZ \to \RR^d$ be a bounded generalised polynomial. Then there exists a minimal nilsystem $(X,T)$, a point $z \in X$ and a piecewise polynomial map $F \colon X \to \RR^d$ such that $g(n) = F(T^n(z))$ for all $n \in \ZZ$.
\end{theorem}

{
Since the dynamical properties of nilsystems are fairly well-understood, this theorem has numerous corollaries. In particular, it leads to a complete description of possible limiting distributions of bounded generalised polynomial sequences. We also have the following recurrence result: If $g \colon \ZZ \to \RR^d$ is a generalised polynomial, then for every $\e > 0$ and for almost every $n \in \NN$ (with respect to Banach density), the set of $m \in \NN$ such that $\norm{g(n+m) - g(n)} < \e$ is $\IP^*$ (and hence syndetic).
}

{
The corollaries mentioned above do not convey much information for generalised polynomials which are constant away from a set with Banach density $0$. However, there are interesting examples of such ``almost constant'' generalised polynomials. In this context it is more convenient to speak of \emph{generalised polynomial sets}, i.e., sets of zeros of generalised polynomials or, equivalently, sets whose characteristic sequences are generalised polynomials. 
}
\begin{theorem}[{\cite[Thm.\ B and C]{ByszewskiKonieczny-2018}}]\label{thm:BK-gp-set-exples}
	The following sets are generalised polynomial:
	\begin{enumerate}[wide]
	\item The set $\set{F_n}{n \geq 1}$ of Fibonacci numbers, given by $F_0=0,\ F_1 =1$ and $F_{n+2} = F_{n+1}+F_n$;
	\item The set $\set{T_n}{n \geq 1}$ of Tribonacci numbers, given by $T_0 = 0,\ T_1 = 1,\ T_2 =1$ and $T_{n+3} = T_{n+2} + T_{n+1} + T_n$;
	\item Any set $\set{a_{n}}{n \geq 1}$ of positive integers with $\displaystyle \liminf_{n \to \infty} \frac{\log a_{n+1}}{\log a_n} > 1$.
	\end{enumerate}
\end{theorem}

{
In joint work with Byszewski \cite{ByszewskiKonieczny-2018} we undertook the study of generalised polynomial sets with Banach density $0$. Theorem \ref{thm:BL-intro} implies that a generalised polynomial set with positive Banach density is an $\IP^*_+$ set. We obtained the following complementary result.
}
\begin{theorem}[{\cite[Thm.\ A]{ByszewskiKonieczny-2018}}]
	Let $E \subset \NN$ be a generalised polynomial set with Banach density $0$. Then $E$ is not an $\IP_+$ set.
\end{theorem}

{
Until now, this theorem provided essentially the only available method of proving that a given set with Banach density $0$ is not generalised polynomial. In fact, to the best of our knowledge, even the following basic question was open.
}

\begin{question}\label{quest:sparse=>gp}
	Does there exist a set $\set{a_n}{n \geq 1} \subset \NN$ with $a_{n+1} - a_n \to \infty$ as $n \to \infty$ which is \emph{not} generalised polynomial?
\end{question}

\subsection{\checkmark\ New results}\label{ssec:intro-new}\mbox{}
{
In this paper, we answer the above question in the affirmative and produce a very explicit example of a non-generalised polynomial set with logarithmic growth, namely the set of powers of $10$ (or of any other integer $k \geq 2$).
}

\begin{alphatheorem}\label{thm:A}
	Let $k \geq 2$ and let $E \subset \NN$ be a generalised polynomial set such that $k^n \in E$ for all $n \geq 0$. Then the set
\[ \set{ m \in \NN}{ m k^n \in E \text{ for infinitely many } n \geq 0} \]
is $\IP^*_+$.
In particular, the set $\set{k^n}{n \geq 0}$ is not generalised polynomial.
\end{alphatheorem}

{\color{NavyBlue}
One of the main sources of motivation behind Theorem \ref{thm:A} is the application in \cite{ByszewskiKonieczny-2019}, where we partially classified generalised polynomials which are also $k$-automatic sequences, i.e., sequences whose $n$-th term can be computed by a finite device given  on input the digits of $n$ in base $k$ (for more background on this problem, see \cite{ByszewskiKonieczny-2019}; for background on automatic sequences see \cite{AlloucheShallit-book}). In fact, we obtained a complete classification conditional on the hypothesis that the set $\set{k^n}{n \geq 0}$ is not generalised polynomial. The following result is an immediate consequence of Theorem D in \cite{ByszewskiKonieczny-2019} and Theorem \ref{thm:A} here. Recall that a sequence $a \colon \NN \to \RR$ is \emph{ultimately periodic} if there exist $n_0 \geq 0$ and $d > 0$ such that $a(n+d) = a(n)$ for all $n \geq n_0$.
}

\begin{alphatheorem}\label{thm:B}\color{Mahogany}
	Let $g \colon \NN_0 \to \RR$ be a sequence which is both generalised polynomial and automatic. Then $g$ is ultimately periodic.
\end{alphatheorem}

\subsection{\checkmark\ Future directions}
{\color{NavyBlue}
While Theorem \ref{thm:A} gives a new criterion for detecting non-generalised polynomial sets, many questions remain open. We mention two of them.
}

{\color{NavyBlue}
Firstly, we note that there is a dearth of examples of generalised polynomial sets containing infinite geometric progressions. This prompts the following question.
}
\begin{question}\color{OliveGreen}\label{quest:sparse-and-k-rich}
	Does there exist an integer $k \geq 2$ and a generalised polynomial set $E \subset \NN$ with natural density $0$ such that $k^n \in E$ for all $n \geq 0$?
\end{question}

{\color{NavyBlue}
Secondly, we recall that the simplest example of a nilsystem is a rotation on a torus. In Section \ref{sec:torus} we deal with a variant of Theorem \ref{thm:A} corresponding to this spacial case. It is natural to ask if sets with Banach density $0$ can already appear in this context.
}

\begin{question}\color{OliveGreen}\label{quest:sparse-on-torus}
	Does there exist $\alpha \in \RR^d$ and an algebraic variety $V \subsetneq \RR^d$ such that $\fp{n \alpha} \in V$ for infinitely many $n \in \NN$?
\end{question}

\subsection*{\checkmark\ Acknowledgements}
{\color{NavyBlue}
While writing this paper, the author was supported by the ERC grant ErgComNum 682150 at the Hebrew University of Jerusalem. During the review process, the author was working within the framework of the LABEX MILYON (ANR-10-LABX-0070) of Universit\'{e} de Lyon, within the program "Investissements d'Avenir" (ANR-11-IDEX-0007) operated by the French National Research Agency (ANR). The author also acknowledges support from the Foundation for Polish Science (FNP).

The author is grateful to Jakub Byszewski, David Kazhdan and Tamar Ziegler for fruitful discussions and insightful comments.
}

\subsection{Organisation}

The paper is organised as follows.

In Section \ref{sec:prelims} we introduce the relevant background. All results are standard in their respective fields, but since we cover a relatively wide array of subjects, we include all of the prerequisite results and terminology. Readers familiar with each of the subjects can safely skip over the corresponding subsection.

In Section \ref{sec:lemmas} we discuss some basic constructions and prove a number of fundamental facts about them. In particular, we define a notion of an algebraic closure along an ultrafilter, which plays a crucial role in our reasoning and to the best of our knowledge does not appear in the literature.

In Section \ref{sec:torus} we discuss the special case of Theorem \ref{thm:A}, concerning generalised polynomial sets which can be realised on the torus using the construction from Theorem \ref{thm:BL-intro}, or equivalently, whose representation in Theorem \ref{thm:Leibman} involves only basic generalised monomials of degree $1$. Arguably, this is the first non-trivial case of our main theorem, and it allows us to present the main ideas of the proof in a simpler setting. 

In Section \ref{sec:setup} we discuss the way in which an arbitrary generalised polynomial can be represented in terms of generalised monomials (cf.{} Theorem \ref{thm:Leibman}) and we introduce a considerable amount of notation that will be used in the remainder of the paper. We also construct the maps $T_k \colon [0,1)^{\cD} \to [0,1)^{\cD}$, where $k \in \NN$ and $\cD$ is a set of indices. These maps play a role that is closely analogous to the $\times k$ maps $[0,1)^d \ni x \mapsto \{kx\} \in [0,1)^d$, where $\{\cdot\}$ denotes the coordinate-wise fractional part. The fundamental property of these maps is that $\{v(km)\} = T_k(\{v(m)\})$, where $v \colon \NN \to \RR^{\cD}$ is a generalised monomial of a certain special form and $m \in \NN$ (cf.{} Prop. \ref{prop:Tk-basic}(\ref{it:Tk-basic:times-k})). It is also important that $T_k$ are piecewise linear and take a particularly simple form, which is discussed in Proposition \ref{prop:Tk-basic}.

In Section \ref{sec:recurrent}, we prove a type of a recurrence theorem for the maps $T_k$ defined above, which additionally takes into account the behaviour of certain algebraic varieties under these maps (Theorem \ref{thm:recurrence}). This is the technical heart of the paper, and a fundamental component of the proof of Theorem \ref{thm:A}. The argument proceeds by induction with respect to the complexity of the index set $\cD$, and relies in a non-trivial way on the algebra of the Stone--\v{C}ech compactification $\beta \NN$.
 
In Section \ref{sec:main} we complete the proof of Theorem \ref{thm:A}. With the help of  Leibman's representation theorem, Theorem \ref{thm:Leibman}, we can translate Theorem \ref{thm:A} into a statement concerning a generalised monomial sequence whose values ``often'' belong to a certain semialgebraic set. This puts us in the context where Theorem \ref{thm:recurrence} can be applied. 

In Appendices we discuss technical results which would hinder the flow of the paper if they were included in the main body. Appendix \ref{ap:sg-lim} concerns limits of sequences of semialgebraic sets, and Appendix \ref{ssec:lemmas-gp-lim} concerns pointwise limits of sequences of generalised polynomials. We show that, with some natural constraints on complexity, generalised polynomial sequences on $\NN$ are closed under pointwise limits, which is a new and unexpected result. \section{Preliminaries}\label{sec:prelims}

\subsection{Notation}\label{ssec:prelims-notation}\mbox{}
{\color{NavyBlue}
We write $\NN = \{1,2,\dots\}$ and $\NN_0 = \NN \cup \{0\}$.  
As usual, the fractional part and the integer part of a real number $x \in \RR$ are denoted by $\fp{x} \in [0,1)$ and $\ip{x} \in \ZZ$. We also use the same notation for real vectors $x = (x_i)_{i=1}^d \in \RR^d$ ($d \in \NN$), where the operations are defined coordinatewise: $\fp{x} = (\fp{x_i})_{i=1}^d$ and $\ip{x} = (\ip{x_i})_{i=1}^d$. Following the Iverson bracket convention, for a statement $\varphi$ we let $\braif{\varphi}$ denote $1$ if $\varphi$ is true and $0$ if $\varphi$ is false.

For $x \in \RR^d$ ($d \in \NN)$, we let $\norm{x} = \norm{x}_2$ denote the Euclidean norm of $x$ and $\Ball(x,r)$ denote the open ball with radius $r > 0$ centred at $x$. For a linear map $T \in \End(\RR^d)$, we let $\norm{T}$ denote the operator norm, 
\[
	\norm{T} = \sup\set{ \norm{T(x)}}{x \in \RR^d,\ \norm{x} = 1}. 
\]
For a set $E \subset \NN$ we define its \emph{upper and lower densities} as
\begin{align*}
	\overline{d}(E) &= \limsup_{N \to \infty} \frac{\abs{A \cap [1,N]}}{N}, &&&
	\underline{d}(E) &= \liminf_{N \to \infty} \frac{\abs{A \cap [1,N]}}{N},
\end{align*}
where by a slight abuse of notation we use $[1,N]$ to denote the set $\{1,2,\dots,N\}$. 
If $\overline{d}(E) = \underline{d}(E)$ we call the common value \emph{natural density} of $E$, and let it be denoted by $d(E)$. We also define the \emph{(upper) Banach density} 
\begin{align*}
	d^*(E) &= \limsup_{N \to \infty} \sup_{M \geq 0}  \frac{\abs{A \cap [M,M+N)}}{N},
\end{align*}
where $[M,M+N) = \{M,M+1,\dots,M+N-1\}$.
}
 
\subsection{\checkmark\ Ultrafilters}\label{ssec:prelims-ultrafilters}
Throughout this paper, it will be helpful to use the notions of largeness and limit provided by ultrafilters. For an extensive introduction to ultrafilters on semigroups and their applications we refer to \cite{HindmanStrauss-book}; for a more succinct treatment see \cite{Bergelson-2003, Bergelson-2010} and references therein. 

An \emph{ultrafilter} on $\NN$ is an element of the Stone--\v{C}ech compactification of $\NN$ with discrete topology, denoted by $\beta \NN$. We also frequently work with $\beta \NN_0 = \beta(\NN_0) = (\beta \NN) \cup \{0\}$. The Stone--\v{C}ech compactification $\beta \NN_0$ comes equipped with a compact topology and an inclusion map $\NN_0 \hookrightarrow \beta \NN_0$ which allows us to identify $\NN_0$ with an open dense subset of $\beta \NN_0$. Ultrafilters in the image of $\NN_0$ are \emph{principal} and those in $\beta \NN_0 \setminus \NN_0$ are \emph{non-principal}. The defining feature on $\beta \NN_0$ is that any map $f$ from $\NN_0$ to a compact space $X$ can be uniquely extended to a continuous map $\beta f$ from $\beta \NN_0$ to $X$. We will use $\displaystyle\lim_{n \to p} f(n)$ to denote the value of $\beta f$ at the point $p \in \beta \NN_0$.

As a consequence of the extension property, $\beta \NN_0$ inherits from $\NN_0$ the additive semigroup operation given by
\[ \beta \NN_0 \times \beta \NN_0 \ni (p,q) \mapsto p+q := \lim_{n \to p} \lim_{m \to q} (n+m) \in \beta \NN_0.\]
This operation is neither commutative nor jointly continuous; it is continuous in the first argument and the (topological) centre of $\beta \NN_0$ is exactly $\NN_0$. By the same token, $\beta \NN_0$ also carries the multiplicative semigroup structure given by
\[ \beta \NN_0 \times \beta \NN_0 \ni (p,q) \mapsto p \cdot q := \lim_{n \to p} \lim_{m \to q} (n \cdot m) \in \beta \NN_0.\]
By a slight abuse of notation, we also define exponentiation
\[
	\beta \NN_0 \times \NN_0 \ni (p,k) \mapsto k^p := \lim_{n\to p} k^n \in \beta \NN_0.
\]

Ultrafilters can also be identified with families of subsets of $\NN_0$. Under this identification, an ultrafilter $p \in \beta \NN_0$ contains all sets $E \subset \NN_0$ such that $\lim_{n \to p} 1_E(n) = 1$. The property of belonging to a given ultrafilter is monotone,  partition regular and preserved under finite intersections, meaning that: 
\begin{inparaenum}[(1)]
\item\label{it:40:1} if $E \supset F \in p$ then $E \in p$; 
\item\label{it:40:2} if $E \cup F \in p$ then $E \in p$ or $F \in p$; 
\item\label{it:40:3} if $E,F \in p$ then $E \cap F \in p$.
\end{inparaenum}
Together with the non-triviality requirement:\begin{inparaenum}[(1)]\setcounter{enumi}{-1} \item\label{it:40:0}$\emptyset \not \in p$ and $\NN_0 \in p$, \end{inparaenum}
 these properties can be taken as an alternative definition of the set of ultrafilters. Any infinite set $E \subset \NN_0$ is a member of a non-principal ultrafilter. The sets $\overline{E} = \set{ p \in \beta \NN_0}{E \in p}$ are closed and open, and form a basis for the topology of $\beta \NN_0$.

A formula $\varphi(n)$ is true \emph{for $p$-almost all $n$}, denoted $\forall^p_n \ \varphi(n)$, if and only if
\[ \set{ n \in \NN_0}{ \varphi(n) \text{ is true}} \in p.\] 
The properties \eqref{it:40:0}---\eqref{it:40:3} imply that the quantifier $\forall^p$ behaves in a very convenient way with respect to logical connectives:
\begin{align}\label{eq:quantifier-betaN}
	\forall^p_n \ \neg \varphi(n) 
	 \Longleftrightarrow
	\neg \forall^p_n \ \varphi(n), 
	&& \forall^p_n \ \bra{\varphi(n) \wedge \psi(n)}
	 \Longleftrightarrow 
	\bra{ \forall^p_n \ \varphi(n)} \wedge \bra{ \forall^p_n \ \psi(n)},
	\end{align} 
	and likewise with $\vee$ or any other connective in place of $\wedge$. We briefly digress to remark that this behaviour is closely related to \L{}o\'s theorem and the ultraproduct construction.

An ultrafilter $p \in \beta \NN$ is \emph{idempotent} if $p+p = p$ and \emph{minimal} if $\beta \NN_0 + p$ is a minimal (left) ideal, meaning that for any $q \in \beta \NN_0$ there exists $r \in \beta \NN_0$ such that $r+q+p = p$. The sets of idempotent and minimal ultrafilters are denoted by $\EN$ and $\KN$ respectively.  Existence of minimal idempotents follows from the Ellis--Numakura lemma and an application of the Kuratowski--Zorn lemma. In fact, any minimal left ideal in $\beta \NN_0$ takes the form $\beta \NN_0 + q$ where $q$ is a minimal idempotent. For $p,q \in \KN$ we write $p \sim q$ if $\beta \NN_0 + p = \beta \NN_0 + q$. We also note that $\EN \cap \NN = \emptyset$ and more generally $p+n \neq p$ for all $p \in \beta \NN_0$ and $n \in \NN$. If $p \in \beta \NN$ is idempotent then $m\NN \in p$ for all $m \in \NN$.

A set $E \subset \NN$ is $\IP$ if it contains the set of finite sums of a sequence $(n_i)_{i=1}^\infty \subset \NN$,
\[
	\mathrm{FS}\bra{(n_i)_{i=1}^\infty} = \set{ \textstyle \sum_{i \in I} n_i}{ I \subset \NN, \text{ finite}}.
\]
Accordingly, $E \subset \NN$ is $\IP^*$ if it has a non-empty intersection with any set of finite sums (or, equivalently, with any $\IP$ set). More generally, a set is $\IP_+$ (resp.\ $\IP^*_+$) if it is a translation of an $\IP$ (resp.\ $\IP^*$) set. 
One can show that a set $E \subset \NN_0$ is $\IP$ (resp.\ $\IP^*$) if and only if it is an element of an idempotent ultrafilter (resp.\ all idempotent ultrafilters). 
(This is a key step in the ultrafilter proof of the Hindman theorem by Glazer and Galvin.) In the next section we introduce the notion of a central set. For now, we mention that a set $E \subset \NN_0$ is central if and only if it is an element of a minimal idempotent ultrafilter.

More generally, if $\cP$ is a partition regular and monotone property, then there exists a closed family of ultrafilters $\Pi$ such that the sets $E \subset \NN_0$ satisfying $\cP$ are exactly the elements of ultrafilters in $\Pi$. If additionally $\cP$ is preserved under translations and dilations then $\Pi$ is an additive and multiplicative left ideal \cite[Thm.\ 6.79]{HindmanStrauss-book}. In particular, let us put
\begin{equation}\label{eq:def-of-Xi}
	\Xi = \set{ p \in \beta \NN_0}{ {d}^*(E) > 0 \text{ for all } E \in p}.
\end{equation}
Then $\Xi$ is a closed, non-empty left additive and multiplicative ideal. If $E \subset \NN_0$ and ${d}^*(E) > 0$ then there exists $p \in \Xi$ such that $E \in p$. 

We will also be interested in convergence of sets along ultrafilters. To make the notion of the convergence precise, recall that for any set $X$ the space $\{0,1\}^X$ of all subsets $X$ can be endowed with the product topology. For a sequence of sets $S_n \subset X$ ($n \in \NN_0$) and an ultrafilter $p \in \beta \NN_0$ by ``the limit of $S_n$ along $p$'' we mean the limit with respect to this topology, that is, 
\begin{align*}
	\lim_{n \to p} S_n = \set{ x \in X}{ \forall^p_n \ x \in S_n}.
\end{align*}

\begin{lemma}\label{lem:set-limit-basic}
	Let $X_n, Y_n$ ($n \in \NN_0$) be two sequences of sets and let $p \in \beta \NN_0$. Then 
	\begin{align*}
		\lim_{n \to p}(X_n \cap Y_n) = \lim_{n \to p}X_n \cap \lim_{n \to p} Y_n,  
	\end{align*}
	and the same holds with other set-theoretic operations 	$\cup, \setminus, \triangle$ in place of $\cap$.
\end{lemma}
\begin{proof}
	Follows immediately from \eqref{eq:quantifier-betaN}.
\end{proof} 
\subsection{\checkmark\ Dynamical systems}\label{ssec:dynamical-systems}\label{ssec:prelims-dynsys}\mbox{}
{
A \emph{dynamical system} is a pair $(X,T)$ where $X$ is a compact topological space and $T \colon X \to X$ is a continuous map. The dynamical system $(X,T)$ is \emph{minimal} if for any point $x \in X$ the corresponding orbit $\set{T^n(x)}{n \in \NN}$ is dense in $X$. 

A point $x \in X$ is \emph{recurrent} if for any open neighbourhood $x \in U \subset X$, there exists $n \in \NN$ with $T^n(x) \in U$. 
Accordingly, a point $x \in X$ is \emph{uniformly recurrent} if for any open neighbourhood $x \in U \subset X$, the set $\set{n \in \NN}{T^n(x) \in U}$ is syndetic, that is, it has bounded gaps. It is well-known (e.g.\ \cite[{Thm.\ 1.15 and 1.17}]{Furstenberg-book}) that $x \in X$ is uniformly recurrent if and only if it is an element of a minimal subsystem of $(X,T)$. These notions of recurrence can also be characterised in terms of ultrafilters. A point $x \in X$ is recurrent if $\lim_{n \to p} T^n(x) = x$ for some $p \in \beta \NN$, and uniformly recurrent if $x = \lim_{n \to q} T^n(x)$ for some $q \in \mathrm{K}(\beta \NN)$.

If $X$ is additionally a metric space then a pair of points $x, y \in X$ is \emph{proximal} if $\liminf_{n \to \infty} d_X(T^n(x),T^n(y)) = 0$, or equivalently if there exists $p \in \beta \NN$ with $\lim_{n \to p} T^n(x) = \lim_{n \to p} T^n(y)$. The system $(X,T)$ is \emph{distal} if no pair of distinct points $x,y \in X$ is proximal. 
}

{
A \emph{nilsystem} is a dynamical system of the form $(G/\Gamma,T_g)$ where $G$ is a nilpotent Lie group, $\Gamma < G$ is a discrete subgroup such that $G/\Gamma$ is compact and $T_g$ is the action of $g \in G$ on $G/\Gamma$ by (left) multiplication. There exists a natural choice of coordinates on $G/\Gamma$ via Mal'cev basis, identifying $G/\Gamma$ with the cube $[0,1)^d$ (with some of the sides glued), where $d$ is the dimension of $G$. All nilsystems are distal. For more details, see e.g.\ \cite{BergelsonLeibman-2007}.
}

{
A set $E \subset \NN$ is \emph{central} if there exist a dynamical system $(X,T)$, a pair of proximal points $x,y \in X$ with $y$ uniformly recurrent and an open neighbourhood $y \in U \subset X$ such that $E = \set{ n \in \NN}{ T^n(x) \in U}$. 
}

{
We record two recurrence results which use the notions of largeness introduced in the previous section.
}

\begin{theorem}[{\cite[Thm.\ 9.11.]{Furstenberg-book}}]\label{thm:recurrence-IP*}
 Let $(X,T)$ be a minimal distal topological dynamical system. Then for every $x\in X$ and every open set $\emptyset \neq U \subset X$, the set $\set{n\in \NN}{T^n(x)\in U}$ is $\IP^*_+$.
\end{theorem}

\begin{theorem}[{\cite[Thm.\ 3.4]{Bergelson-2010}}]\label{thm:recurrence-central}
	Let $(X,T)$ be a topological dynamical system. Then a point $x \in X$ is uniformly recurrent if and only if there exists a minimal idempotent $q \in \beta \NN$ such that $\lim_{n \to q} T^q(x) = x$.
\end{theorem}
 
\subsection{\checkmark\ Algebraic geometry}\label{ssec:prelims-algebraic-geometry}
{\color{NavyBlue}
An \emph{algebraic variety} (or an \emph{algebraic set}) in $\RR^d$ is the zero locus of a system of polynomial equations,
\[
\fV(F) = \set{ x \in \RR^d }{ f(x) = 0 \text{ for all } f \in F},
\]
where $F \subset \RR[\mathbf{x}_1,\dots,\mathbf{x}_d]$. Conversely, for a set $X \subset \RR^d$ we define the corresponding ideal
\[
	\fI(X) = \set{ f \in \RR[\mathbf{x}_1,\dots,\mathbf{x}_d]}{ f(x) = 0 \text{ for all } x \in X}.
\]
A variety is \emph{irreducible} if it cannot be expressed as the union of proper subvarieties. Any variety is a finite union of its irreducible components.
Note that throughout we work over $\RR$, which is not an algebraically closed field, rendering many basic results in algebraic geometry inapplicable.
}

{\color{Blue}
The \emph{Zariski closure} of a set $X \subset \RR^d$ is the smallest algebraic variety containing $X$; it is denoted by $\algcl(X) = \fV(\fI(X))$ (the topological closure is denoted by $\topcl(X)$). By Hilbert's basis theorem, any ideal in $\RR[\mathbf{x}_1,\dots,\mathbf{x}_d]$ is finitely generated, so any variety can be represented as $\fV(F)$ with $F$ finite. In fact, one can always assume that $\abs{F} = 1$, replacing the conditions $f(x) = 0$ for all $f \in F$ with the single condition $\sum_{f \in F} f^2(x) = 0$. As a consequence, any descending sequence of varieties \[ V_1 \supseteq V_2 \supseteq V_3 \supseteq \dots \supseteq V_{n} \supseteq \dots \qquad (n \in \NN),\] is eventually constant. A variety is \emph{defined over} a field $\KK < \RR$ if it takes the form $\fV(F)$ with $F \subset \KK[\mathbf{x}_1,\dots,\mathbf{x}_d]$; without loss of generality, $\abs{F} = 1$. 
}

{
The following lemmas are standard. We let $\Aff(d) = \RR^d \rtimes \End(\RR^d)$ denote the semigroup of affine maps $T(x) = S(x)+c$ where $S \in \End(\RR^d)$ and $c \in \RR^d$, and $\AGL(d) = \RR^d \rtimes \Aut(\RR^d)$ denote the group of affine invertible maps. Note that $\Aff(d)$ is a vector space over $\RR$, so it makes sense to speak of algebraic varieties in $\Aff(d)$. The following lemma is reminiscent of \cite[Lem.\ 3.6]{ByszewskiKonieczny-2018}.
}

\begin{lemma}[]\label{lem:stab-is-alg}
	Let $U,V \subset \RR^d$ be algebraic varieties. 
	\begin{enumerate}
	\item\label{it:stab-is-alg:A} If $T \in \AGL(d)$ is an invertible affine map and $T(V) \subset V$, then $T(V) = V$.

	\item\label{it:stab-is-alg:B} The set of affine maps $T \in \Aff(d)$ such that $T(V) \subset U$ is algebraic.
	
	\item\label{it:stab-is-alg:C} The set of vectors $v \in \RR^d$ such that $V+v = V$ is a vector space. 
	\end{enumerate}
\end{lemma}
\begin{proof}\color{White}
\begin{enumerate}[wide]
\item It follows from the Hilbert basis theorem that the descending sequence of algebraic varieties $T^n(V)$, $n \in \NN_0$, stabilises. Hence, there exists $n$ such that $T^n(V) = T^{n+1}(V)$. Since $T^n$ is injective, it follows that $V = T(V)$.

\item If $T \in \Aff(d)$, then $T(V) \subset U$ if and only if $T(x) \in U$ for all $x \in V$. For each $x \in V$, the set of $T \in \Aff(d)$ such that $T(x) \in U$ is algebraic. It remains to recall that the intersection of any family of algebraic sets is algebraic. 

\item The set of such $v$ is a subgroup of $\RR^d$ which is also an algebraic variety. Hence, it is a vector space.
\qedhere
\end{enumerate}
\end{proof}

\begin{lemma}\label{lem:alg-cl-Q}\color{BurntOrange}
Let $\KK < \RR$ and let $X \subset \KK^d$. Then $\algcl(X)$ is defined over $\KK$.
\end{lemma}
\begin{proof}\color{White}
	Let $f \in \fI(X)$. We may expand $f$ as
	\[\textstyle
		f(x) = \sum_{i=1}^s \alpha_i g_i(x),
	\]
	where $\alpha_i \in \RR$ are linearly independent over $\KK$ and $g_i \in \KK[\mathbf{x}_1,\dots,\mathbf{x}_d]$. If $x \in \KK^d$ and $f(x) = 0$, then $g_i(x) = 0$ for all $1 \leq i \leq s$. It follows that $g_i \in \fI(X)$ for all $1 \leq i \leq s$ and consequently that $\fI(X)$ is spanned by polynomials with coefficients in $\KK$.
\end{proof}

\subsection{\checkmark\ Semialgebraic geometry}\label{ssec:prelims-semialgebraic}\mbox{}
{
Even though the fact that $\RR$ is not algebraically closed is often seen as a drawback in algebraic geometry, this very fact gives rise to the rich theory of semialgebraic geometry. For background, see e.g.\ \cite{BochnakCosteRoy-book}.
}

{
A \emph{basic semialgebraic set} in $\RR^d$ is the set of solutions to a system of polynomial equations and inequalities:
\begin{equation}\label{eq:def-of-semialg}
	\fS(F,G) = \set{ x \in \RR^d }{ f(x) = 0 \text{ for all } f \in F \text{ and } g(x) > 0 \text{ for all } g \in G},
\end{equation}
where $F,G \subset \RR[\mathbf{x}_1,\dots,\mathbf{x}_d]$ are finite. We may additionally assume that $\fV(F)$ is the Zariski closure of $\fS(F,G)$. A \emph{semialgebraic set} is a finite union of basic semialgebraic sets. 
}
	
{
Recall that projections of algebraic sets need not be algebraic, as shown already by the hyperbola $\set{(x,y) \in \RR^2}{xy=1}$. However, a foundational theorem of Tarski and Seidenberg show that projections of semialgebraic sets are again semialgebraic. We cite a slightly stronger variant of this result.
}
\begin{theorem}[{\cite[Sec.\ 5]{BochnakCosteRoy-book}}]\label{thm:Tarski-Seidenberg}
	Let $S \subset \RR^d$ be a semialgebraic set and let $f \colon \RR^d \to \RR^e$ be a polynomial map. Then $f(S)$ is semialgebraic.
\end{theorem}

{
A map $f \colon \RR^d \to \RR^e$ is \emph{piecewise polynomial} if there exists a partition $\RR^d = S_1 \cup S_2 \cup \dots \cup S_r$ into semialgebraic pieces such for each $1 \leq i \leq r$ the restriction of $f$ to $S_i$ is a polynomial. A piecewise polynomial map on a domain $\Omega \subset \RR^d$ is the restriction of a piecewise polynomial map. It follows from Theorem \ref{thm:Tarski-Seidenberg} that the image of a semialgebraic set through a piecewise polynomial map is again semialgebraic. It is elementary to show that the preimage of a semialgebraic set through a piecewise polynomial map is semialgebraic as well.
}

{
As hinted before, we will be interested in the sets of times that a given bounded sequence hits a semialgebraic subset of $[0,1)^d$. The following lemma allows us to alter the choice of basis. While not strictly speaking necessary, it significantly simplifies notation in several places.
}

\begin{lemma}\label{lem:semialg-wlog-SL}
	Let $S \subset [0,1)^d$ be semialgebraic and let $T \in \SL(d,\ZZ)$. Put 
	\[ S' := \fp{T(S)} = \set{ \fp{T(x)}}{x \in S}.\]
	Then $S'$ is semialgebraic and
	\[ S = \set{ x \in [0,1)^d}{ \fp{T(x)} \in S' }.\]
\end{lemma}
\begin{proof}
	The claim follows directly from the observation that the map $x \mapsto \fp{T^{-1}(x)}$ is piecewise linear with a piecewise linear inverse $y \mapsto \fp{T(y)}$. 
\end{proof}

{
We close this section with a discussion of limits of semialgebraic sets, which are closely related to limits of generalised polynomials discussed in Section \ref{ssec:prelims-gen-poly}. In general, such limits are not guaranteed to be well-behaved: for instance, any open subset of $\RR^d$ can be represented as the limit of a sequence of semialgebraic sets, see Example \ref{ex:set-limit-alg-bad}. Thus, in order to ensure that the limit of a sequence of semialgebraic sets is again semialgebraic we need to impose some additional restrictions.
In a somewhat ad hoc manner, we define the \emph{complexity} of a basic semialgebraic set $S \subset \RR^d$ as the sum of degrees of all polynomials in its  representation \eqref{eq:def-of-semialg},
\[
	\complexity(S) := \min \set{ \sum_{ f \in F} \deg(f) + \sum_{g \in G} \deg(g) }{F,G \subset \RR[\bb x_1,\dots,\bb x_d],\ S = \fS(F,G) }. 
\]
The complexity of a semialgebraic set $S \subset \RR^d$ is the sum of complexities of its basic components,
\[
	\complexity(S) := \min\set{ \sum_{i=1}^r \complexity(S_i)}{
		S_i \subset \RR^d, \text{ basic semialgebraic, } S = \bigcup_{i=1}^r S_i
	}.
\]
It will be convenient to also define $\complexity(S) := \infty$ if $S \subset \RR^d$ is not semialgebraic. 
We will never be concerned with the exact value of the complexity of a given semialgebraic set; rather, we are interested in sequences of semialgebraic sets whose complexity is uniformly bounded. The following example gives the class of such sequences which is the most important for our purposes.
}

\begin{example}\label{ex:semialg-param}
	Let $F,G \subset \RR[\bb x_1,\dots,\bb x_d, \bb y_1, \dots, \bb y_e]$ be finite families of polynomials in $d+e$ variables. The corresponding \emph{parametrised family of basic semialgebraic sets} is given by
\begin{equation*}\fS(F,G;y) = \set{ x \in \RR^d }{ f(x,y) = 0 \text{ for all } f \in F \text{, } g(x,y) > 0 \text{ for all } g \in G},
\end{equation*}
where $y \in \RR^e$. The complexity of $\fS(F,G;y)$ is bounded uniformly in $y$; in fact,
\begin{align*}
	\complexity\bra{\fS(F,G;y)} &\leq  \sum_{f \in F} \deg(f) + \sum_{g \in G} \deg(g) \text{ for all } y \in \RR^e.
\end{align*}
\end{example}

The proof of the following result is fairly routine. For the sake of keeping the preparatory sections reasonably short, we delegate it to Appendix \ref{ap:sg-lim}.
\begin{proposition}\label{prop:semialg-limit}
	Let $S_n \subset \RR^d$ ($n \in \NN$) be a sequence of semialgebraic sets with uniformly bounded complexity and let $p \in \beta \NN$. Then $\displaystyle \lim_{n \to p} S_n$ is also semialgebraic. 
\end{proposition}
\begin{corollary}\label{cor:semialg-limit}
	Let $S \subset \RR^{d+e}$ be a semialgebraic set, let $y_n \in \RR^e$ ($n \in \NN$), and let $p \in \beta \NN$. Then the set \(
		\set{x \in \RR^d }{ \forall^p_n \ (x,y_n) \in S }
	\)
	is semialgebraic.
\end{corollary}
 
\subsection{\checkmark\ Generalised polynomials}\label{ssec:prelims-gen-poly}\mbox{}
{
Recall that generalised polynomials were informally discussed in the introduction. To be more precise, we define generalised polynomial maps from $\RR^d$ to $\RR$ as the smallest family such that the coordinate maps $x \mapsto x_i$ ($1 \leq i \leq d$) and the constant maps $x \mapsto \alpha$ ($\alpha \in \RR)$ are generalised polynomials and if $g,h$ are generalised polynomials then so are $g+h$ and $g \cdot h$ and $\ip{g}$ (given by $\ip{g}(x) = \ip{ g(x)}$ for all $x \in \RR^d$). If $g$ is a generalised polynomial then so is the fractional part $\fp{g} = g -\ip{g}$, and we can similarly characterise generalised polynomials as the smallest family containing polynomials and closed under addition, multiplication and the fractional part. 
}

{
A map $\RR^d \to \RR^e$ is generalised polynomial if its projection on any $1$-dimensional subspace is a generalised polynomial. A generalised polynomial map on a domain $\Omega \subset \RR^d$ is the restriction of a generalised polynomial to $\Omega$ (we usually take $\Omega = \NN^d$ or $\ZZ^d$ and $d = 1$). Any generalised polynomial on $\Omega$ can be extended to $\RR^d$ but the extension is not unique unless $\Omega = \RR^d$.
A generalised polynomial subset $E$ of a domain $\Omega \subset \RR^d$ is the zero locus of a generalised polynomial, i.e., a set of the form $\set{n \in \Omega}{g(n) = 0}$ where $g \colon \RR^d \to \RR$ is a generalised polynomial. For instance, $\ZZ^d$ is a generalised polynomial subset of $\RR^d$ and so is any algebraic variety in $\RR^d$. Note that this notion depends on $\Omega$; in particular $\NN$ is trivially a generalised polynomial subset of $\NN$ but not of $\ZZ$ (\footnote{It follows from Theorem \ref{thm:Bergelson-Leibman} that if $g \colon \ZZ \to \RR$ is a bounded generalised polynomial and $g(n) = 0$ for all $n \in \NN$ then the set of $n \in \ZZ$ such that $g(n) \neq 0$ has Banach density $0$.}). It follows from the Lemma \ref{lem:gp-a<g<b} below that in the above definition we may always further assume that $g$ takes only values $0$ and $1$. As a consequence, generalised polynomial sets form an algebra.
}

\begin{lemma}[{\cite[Lem.\ 1.2]{ByszewskiKonieczny-2018}}]\label{lem:gp-a<g<b}
	Let $g \colon \ZZ \to \RR$ be a generalised polynomial on $\ZZ$. Then the map $n \mapsto \braif{g(n)=0}$ is a generalised polynomial. Moreover, for any $a, b \in \RR$, the map $n \mapsto \braif{a \leq g(n) < b}$ is a generalised polynomial on $\ZZ$.
\end{lemma}

{
We next discuss various ways in which generalised polynomials can be represented. We begin with a lemma relating generalised polynomials to piecewise polynomial maps on a bounded domain.
}

\begin{lemma}[{\cite[Lem.\ 1.6]{BergelsonLeibman-2007}}]\label{lem:gp=pw-poly}
	Let $\Omega \subset \RR^d$ be a bounded set and let $f \colon \Omega \to \RR$ be a map. Then the following conditions are equivalent:
	\begin{enumerate}
	\item $f$ is a generalised polynomial;
	\item $f$ is a piecewise polynomial.
	\end{enumerate}	 
\end{lemma}

{
The representation theorem of Bergelson and Leibman was already mentioned in the introduction. We recall it here and present a more complete statement. 
}

\begin{theorem}[{\cite{BergelsonLeibman-2007}}]\label{thm:Bergelson-Leibman}
Let $g \colon \ZZ \to \RR^d$ be a bounded generalised polynomial.
\begin{enumerate}[wide]
\item\label{it:BL:1} There exists a minimal nilsystem $(X,T)$, a point $z \in X$ and a piecewise polynomial map $F \colon X \to \RR^d$ such that $g(n) = F(T^n(z))$ for all $n \in \ZZ$.
\item\label{it:BL:2}  Conversely, for any nilsystem $(X,T)$, any point $z \in X$ and any piecewise polynomial map $F \colon X \to \RR^d$, the map $n \mapsto F(T^n(z))$ is a bounded generalised polynomial on $\ZZ$.
\item\label{it:BL:3} There exists a semialgebraic set $S \subset \RR^d$ parametrized by a piecewise polynomial map $f \colon [0,1]^e \to S$ and a set $Z \subset \ZZ$ with Banach density $0$ such that $g(n) \in S$ for all $n \in \ZZ \setminus Z$ and $g(n)$ is equidistributed in $S$ with respect to the measure induced by the parametrization.
\item\label{it:BL:4} If $Z$ is the set from \eqref{it:BL:3} above, then for any $n \in \ZZ \setminus Z$ and any $\e > 0$ the set 
$\set{m \in \NN}{ \abs{g(n+m) - g(n)} < \e}$ is $\IP^*$.
\end{enumerate}
\end{theorem}

This theorem also ensures that the densities of generalised polynomial sets exist in a rather strong sense. 
\begin{lemma}[{e.g.\ \cite[Cor.\ 1.4]{ByszewskiKonieczny-2018}}]\label{lem:gp-set-density-exists}
	Let $E \subset \NN$ be a generalised polynomial set. Then $E$ has a natural density. In fact, $\abs{E \cap [M,M+N)}/N \to d(E)$ uniformly in $M$ as $N \to \infty$. In particular, $d^*(E) = d(E)$.
\end{lemma}

{\color{red}
Another closely related representation theorem is due to Leibman. In \cite{Leibman-2012}, he constructed a family of basic generalised polynomials 
such that any bounded generalised polynomial can be expressed in terms of a finite number of basic generalised polynomials which are jointly equidistributed. Here, we will only need a weak version of this result, stating that any generalised polynomial can be represented in terms of generalised polynomials of a particularly simple form, but not requiring any equidistribution. This allows us to use a simpler family of basic generalised polynomials and also simplifies the statement of the result. 
}

{\color{Blue}
We define \emph{generalised monomial maps} from $\RR^d$ to $\RR$ to be the smallest family containing all monomials $x \mapsto \alpha x_i^k$ ($\alpha \in \RR,\ 1 \leq i \leq d,\ k \in \NN$) and such that if $g$ and $h$ are generalised monomials then so is $g \fp{h}$. Accordingly, a map from $\RR^d$ to $\RR^e$ is a generalised monomial if and only if each of its coordinates is a generalised monomial. 
}

\begin{theorem}[{\cite{Leibman-2012}}]\label{thm:Leibman}\color{RoyalPurple}
	Let $g \colon \ZZ^d \to \RR^e$ be a generalised polynomial. Then there exists a generalised monomial $v \colon \ZZ^d \to \ZZ^m$ as well as a piecewise polynomial map $F \colon [0,1)^{m} \to \RR^e$ such that $g(n) = F(\fp{v(n)})$.
\end{theorem}

Let us now consider pointwise limits of generalised polynomials. Of course, any sequence on $\ZZ$ is the pointwise limit of a sequence of (ordinary) polynomials, which motivates us to additionally impose a constraint on complexity. The most convenient way to do this is to consider sequences $g \colon \Omega \to \RR$ ($\Omega \subset \RR^e$ for some $e \geq 1$) of the form $g(n) = \lim_{i} h(x_i,n)$, where $h \colon \RR^d \times \Omega \to \RR$ is a generalised polynomial and $x_i \in \RR^d$ $(i \in \NN)$ is a bounded sequence. (We may either consider the limit as $i \to \infty$ under the additional assumption that the limit exists, or alternatively use the notion of convergence along an ultrafilter.) Example \ref{ex:gp-limit-not-gp} shows that when $\Omega = \ZZ$, it still can happen that the limit defining $g(n)$ converges as $i \to \infty$ for each $n \in \Omega$ but the sequence $g$ is not a generalised polynomial. The same holds when $\Omega = \NN^2$.
Surprisingly, the situation is different for generalised polynomials on $\NN$.

\begin{proposition}\label{prop:gp-lim-is-gp}
	Let $p \in \beta \NN$, let $x_i \in \RR^d$ $(i \in \NN)$ be a bounded sequence and let $h \colon \RR^d \times \NN \to \RR$ be a generalised polynomial. Define $g \colon \NN \to \RR$ by
\begin{align*}
g(n) &:= \lim_{i \to p} h(x_i,n).
\end{align*}
	 Then $g$ is a generalised polynomial.
\end{proposition}

Since the proof of this result is quite lengthy and independent of the main focus of the paper, we delegate it to Appendix \ref{ssec:lemmas-gp-lim}.  
  \section{Preparatory results}\label{sec:lemmas}

In this section we discuss some basic facts and constructions on which the remainder of the paper relies. While none of the arguments used in this section are particularly novel, to the best of our knowledge they do not appear elsewhere in the literature.

\subsection{\checkmark\ Closures along ultrafilters}\label{ssec:lemmas-alg-cl}\mbox{}
{\color{Blue}
In analogy to the operation of taking the Zariski closure of a set, for a sequence of points $x_n \in \RR^d$ ($n \in \NN$) and $p \in \beta \NN$ we define the ``Zariski closure of $x_n$ along $p$'': 
\begin{equation}\label{eq:def-of-alglim}
	\alglim_{n \to p}\bra{x_n} := \bigcap_{I \in p} \algcl\set{ x_n }{ n \in I}.
\end{equation}
The set defined in \eqref{eq:def-of-alglim} is clearly algebraic since it is the intersection of a family of algebraic varieties. The following lemma characterises it as the smallest algebraic variety containing $x_n$ for $p$-almost all $n$, and lists some other basic properties.}

\begin{lemma}\label{lem:alg-lim-basic}\color{BurntOrange}
	Let $x_n \in \RR^d$ for $n \in \NN$ and let $p \in \beta \NN$. Put $\displaystyle V := \alglim_{n \to p}\bra{x_n}$. 
\begin{enumerate}
\item\label{it:alg-lim-basic:A}
For $p$-almost all $n$, $x_n \in V$. 
\item\label{it:alg-lim-basic:B}
If $U \subset \RR^d$ is algebraic and $x_n \in U$ for $p$-almost all $n$, then $V \subset U$.
\item\label{it:alg-lim-basic:C}
The algebraic variety $V$ is irreducible.
\item\label{it:alg-lim-basic:E}
If $q \in \beta \NN$ and $(x_n)_{n=1}^\infty$ is bounded, then $\displaystyle  \alglim_{n\to p} \bra{\lim_{m \to q} x_{n+m}} \subset \alglim_{n\to p+q}(x_n)$. 
\end{enumerate}
\end{lemma}
\begin{proof}\color{White}
\begin{enumerate}[wide]
\item It follows from the Hilbert basis theorem that there exists a finite sequence of sets $I_j \in p$, $1 \leq j \leq s$, such that 
\[
	V = \bigcap_{1 \leq j \leq s} \algcl\set{ x_n }{ n \in I_j}.
\]
In particular, $x_n \in V$ for all $n \in I_1 \cap I_2 \cap \dots \cap I_s \in p$.
\item Let $J = \set{n}{x_n \in U}$. Then $J \in p$, whence 
\[
V \subset \algcl\set{ x_n }{ n \in J} \subset U.
\]
\item 
Suppose that $V$ is the union of two subvarieties, $V = U_1 \cup U_2$. Let $J_i = \set{n \in \NN}{x_n \in U_i}$ for $i \in \{1,2\}$. Then $J_1 \cup J_2 \in p$, so $J_1 \in p$ or $J_2 \in p$. Hence, $U_1 = V$ or $U_2 = V$, as needed. 
\item Let $\displaystyle U := \alglim_{n\to p+q}(x_n)$. It follows directly from the definition \eqref{eq:def-of-alglim} that
\[
	\forall^p_n \ \forall^q_m \ x_{n+m} \in U.
\]
Since $U$ is (topologically) closed, it follows that
\[
	\forall^p_n \ \lim_{m\to q} x_{n+m} \in U.
\]
It follows from item \eqref{it:alg-lim-basic:B} that $\displaystyle \alglim_{n\to p} \bra{\lim_{m \to q} x_{n+m}} \subset U$. \qedhere
\end{enumerate}
\end{proof}
\begin{remark}\color{NavyBlue}
	The assumption of boundedness in item \eqref{it:alg-lim-basic:E} can be removed by working in the projective plane and altering the definitions accordingly. Because we do not need this generalisation, we omit further details.
\end{remark}

\subsection{\checkmark\ Separating variables}\label{ssec:prelims-sep}\mbox{}
{ 
Our main result, Theorem \ref{thm:A}, can be rephrased as a statement concerning a certain bounded sequence of points $x_n \in \RR^d$ ($n \in \NN$) which all belong to a semialgebraic subset of $\RR^d$ (cf.{} Theorem \ref{thm:main-strong}). In the course of the argument, which proceeds by structural induction, we encounter a more general situation where $(x_n,t_n) \in S$ for a bounded sequence $x_n \in \RR^d$, a divergent sequence $t_n \in \RR$ and a semialgebraic set $S \subset \RR^{d+1}$. In this section we develop tools that allow us to treat the two components $x_n$ and $t_n$ independently. 
Consider the following motivating example. 
}

\begin{lemma}\label{lem:sep-exple}
	Let $(x_n,t_n) \in \RR^2$ ($n \in \NN$) be a sequence of points such that 
\begin{inparaenum}[(1)]
\item\label{it:30:A} $x_n$ is a limit point of the sequence $(x_m)_{m=1}^\infty$ for each $n \in \NN$ and 
\item\label{it:30:B} $t_n \to \infty$ as $n \to \infty$.
\end{inparaenum} 
Put $V := \algcl\set{(x_n,t_n)}{n \in \NN}$ and $X := \set{x_n}{n \in \NN}$. Then either $V = \RR^2$ or $X$ is finite and $V = X \times \RR$.
\end{lemma}
\begin{proof}	If $V = \RR^2$ we are done, so suppose that this is not the case.
	Then the line $\{x\} \times \RR$ is tangent to $V$ at infinity for each $x \in X$. Since $V$ has only finitely many points at infinity, it follows that $X$ is finite and in particular each point in $X$ is isolated. Hence, for each $x \in X$ there exist infinitely many $n \in \NN$ such that $x_n = x$. As a consequence, the line $\{x\} \times \RR$ contains infinitely many points $(x_n,t_n)$, $n \in \NN$, and thus must be contained in $V$. It follows that $V = X \times \RR$.
\end{proof}

{
The following analogue of Lemma \ref{lem:sep-exple} will be useful for our purposes. A related result was obtained in \cite[Prop.\ 3.8]{ByszewskiKonieczny-2018}. 
}

\begin{proposition}\label{prop:sep-poly-IP}
	Let $q \in \beta \NN$ be idempotent and $p \in \beta \NN + q$. Let $(x_n,t_n) \in \RR^{d} \times \RR$  ($n \in \NN$) be a sequence such that 
\begin{inparaenum}[(1)]
\item\label{it:31:A} $\displaystyle\lim_{m \to q} x_{n+m} = x_n$ for each $n \in \NN$ and 
\item\label{it:31:B} $t_n \to \infty$ as $n \to \infty$.
\end{inparaenum} 
Let $V$ be an algebraic variety such that
	\begin{equation}\label{eq:34:00-2}
		\forall^p_n \ (x_n,t_n) \in V.	
	\end{equation}
Then we have also the ostensibly stronger condition
	\begin{equation}\label{eq:34:01-2}
		\forall^p_n \ \{x_n\} \times \RR \subset V.
	\end{equation}
\end{proposition}
\begin{remark}\label{rmk:separation}
\begin{enumerate}[wide]
\item Using the terminology of algebraic closures of sequences along ultrafilters introduced in Section \ref{ssec:lemmas-alg-cl}, the conclusion of Proposition \ref{prop:sep-poly-IP} can be stated more succinctly as
	\[
		\alglim_{n \to p} (x_n,t_n) = \alglim_{n \to p} (x_n) \times \RR.
	\]
\item Consider the situation when $q$ is idempotent and $p \in \beta \NN+q$. Assumption \eqref{it:31:A} is satisfied for any sequence $x_n$ given by $x_n = \lim_{m \to q} x'_{n+m}$ for a bounded sequence $x_n'$. Conversely, if $x_n$ satisfies \eqref{it:31:A} then it is given by the above formula with $x_n' = x_n$. 
\end{enumerate}
\end{remark}
\begin{proof}
	Take any $f \in \fI(V)$. We will show that for $p$-almost all $n$, $f(x_n,t) = 0$ as a polynomial in $t$. Expand $f$ as 
	\[\textstyle
		f(x,t) = \sum_{j=0}^s t^j f_j(x).
	\]
	Proceeding by induction on $s$, we will show that $f_j(x_n) = 0$ for $p$-almost all $n$ and for all $0 \leq j \leq s$. The case $s = 0$ is trivial, so assume that $s \geq 1$. Since $p+q = p$, the ultrafilter $q$ is not principal and it follows from \eqref{eq:34:00-2} that
	\begin{equation}\label{eq:34:02-2}
		\forall^p_n \ \forall^q_m \ \sum_{j=0}^s t_{n+m}^{j-s} f_j(x_{n+m}) = 0.	
	\end{equation}
	Keeping $n$ fixed and passing to the limit with respect to $m$ we conclude that
	\begin{equation}\label{eq:34:03-2}
		\forall^p_n \ f_s(x_n) = f_s \bra{\lim_{m \to q} x_{n+m} } = 0.	
	\end{equation}
	Let $f'(x,t)$ denote the truncated polynomial
	\[\textstyle
		f'(x,t) = \sum_{j=0}^{s-1} t^j f_j(x).	
	\]
	Then \eqref{eq:34:03-2} and Lemma \ref{lem:alg-lim-basic} imply that $f' \in \fI(V)$:
	\[
		\forall^p_n \ f'(x_n,t_n) = f(x_n,t_n) = 0.
	\]
	It remains to apply the inductive assumption to $f'$.
\end{proof}

{
We record a special case of Proposition \ref{prop:sep-poly-IP} relevant to rotations on the torus. 
}

\begin{corollary}\label{cor:sep-poly-IP}
	Let $q \in \beta \NN$ be idempotent, let $p \in \beta \NN + q$, $k \in \NN_{\geq 2}$. Let $x \in [0,1)^d$ be such that $\displaystyle \lim_{m \to q} \fp{k^m x} = x$. Let $V \subset \RR^d \times \RR$ be an algebraic variety such that 
	\begin{equation}\label{eq:34:00-3}
		\forall^p_n \ \bra{\fp{k^n x},k^n} \in V.	
	\end{equation}
	Then we also have the ostensibly stronger condition
	\begin{equation}\label{eq:34:01-3}
		\forall^p_n \ \fp{k^n x} \times \RR \subset V.
	\end{equation}
\end{corollary}
\begin{proof}
	This will follow from Proposition \ref{prop:sep-poly-IP} as soon as we show that
\begin{equation}\label{eq:129:1}
\fp{k^n x} = \lim_{m \to q} \fp{ k^{n+m} x} \text{ for each } n \in \NN.
\end{equation}
Reasoning separately for each coordinate, we may assume without loss of generality that $d = 1$. If $x$ is rational then the denominator of $x$ is coprime to $k$ and \eqref{eq:129:1} follows readily. If $x$ is irrational then \eqref{eq:129:1} follows from the fact that the map $t \mapsto \fp{t}$ is continuous at $x$.\end{proof}

{
Our next result is an analogue of Proposition \ref{prop:sep-poly-IP} for semialgebraic sets. 
}

\begin{proposition}\label{prop:sep-semialg-IP}
	Let $q \in \beta \NN$ be idempotent and $p \in \beta \NN + q$. Let $(x_n,t_n) \in \RR^{d} \times \RR$ ($n \in \NN$) be a sequence such that 
\begin{inparaenum}[(1)]
\item\label{it:32:A} $\displaystyle \lim_{m \to q} x_{n+m} = x_n$ for each $n \in \NN$ and 
\item\label{it:32:B} $t_n \to \infty$ as $n \to \infty$.
\end{inparaenum} 
Put $V := \displaystyle \alglim_{n \to p}(x_n)$ and let $S \subset \RR^{d+1}$ be a semialgebraic set such that 
\begin{equation}\label{eq:35:01}
\forall^p_n\ (x_n,t_n) \in S.
\end{equation}
Then there exists a relatively open set $U \subset V$ such that $ x_n \in U$ for $p$-almost all $n$ and a continuous map $f \colon U \to \RR$ such that
\begin{align}\label{eq:35:30}
 \set{ (x,y) \in \RR^{d+1} }{ x \in U,\ y \geq f(x) } \subset S. 
 \end{align}
In particular, for $p$-almost all $n$ there exists a semialgebraic set $Q \subset \RR^d$ and a threshold $l \in \RR$ such that 
\begin{align}\label{eq:35:02}
\forall^{n+q}_m \ x_{m} &\in Q &\text{ and }&& Q \times [l, \infty) &\subset S.
\end{align}

\end{proposition}
\begin{proof}
By Proposition \ref{prop:sep-poly-IP}, $\displaystyle \alglim_{n \to p} (x_n,t_n) = V \times \RR$. Replacing $S$ with $S \cap V \times \RR$ if necessary, we may assume that $\algcl(S) = V \times \RR$ (cf.\ Lemma \ref{lem:alg-lim-basic}\eqref{it:alg-lim-basic:A}). Decomposing $S$ and using partition regularity, we may assume that $S$ is a basic semialgebraic set. Consequently, $S$ takes the form 
\[ 
S = \fS(\fI(V \times \RR),G) = \set{(x,t) \in V}{ g_j(x,t) > 0 \text{ for all } 1 \leq j \leq r}
\]
for a finite set $G = \{g_1,\dots, g_{r}\}$ of polynomial maps $\RR^d \times \RR \to \RR$. Let $h_j(x)$ be the leading coefficient of $g_j(x,t)$ as a polynomial in $t$ ($1 \leq j \leq r$). We may assume without loss of generality that none of $h_j$ ($1 \leq j \leq r$) vanishes identically on $V$, since otherwise we could replace $g_j(x,t)$ with a polynomial of lower degree in $t$.
Let 
\begin{align*}
U  &= \set{x \in V}{h_j(x) > 0 \text{ for all } 1 \leq j \leq r}, \\
Y &= \set{x \in V}{h_j(x) = 0 \text{ for at least one } 1 \leq j \leq r}.
\end{align*}
It is also straightforward to construct a continuous map $f \colon U \to \RR$ such that if $x \in U$ and $t \geq f(x)$ then $g_j(x,t) > 0$ for all $1 \leq j \leq r$. To see this, expand $g_j(x,t)$ as $g_j(x,t) = h_j(x) t^{i_j} + g_{j,1}(x) t^{i_j -1} + \dots + g_{j,i_j}(x)$, and put
\[
	f(x) = 1+ \max_{1 \leq j \leq r} \frac{\abs{g_{j,1}(x)} + \abs{g_{j,2}(x)} + \dots + \abs{g_{j,i_j}(x)} }{h_j(x)}.
\]

The definitions above ensure that \eqref{eq:35:30} holds: if $(x,y) \in \RR^{d+1}$, $x \in U$ and $y \geq f(x)$ then $g_j(x,y) > 0$ for all $1 \leq j \leq r$ and consequently $(x,y) \in S$. It remains to show that $x_n \in U$ for $p$-almost all $n$. Since $p+q = q$, we have
\[
	\forall^p_n \ \forall^q_m \ (x_{n+m}, t_{n+m}) \in S. 
\]
We may assume that $t_n > 0$ for all $n$. Letting $g_j'(x,t) := g_j(x,t) - t^{i_j} h_j(x)$ denote the truncated version of $g_j$, for each $1 \leq j \leq r$ we conclude that
\[
	\forall^p_n \ \forall^q_m \ h_j(x_{n+m}) +g_j'(x_{n+m},t_{n+m})t_{n+m}^{-i_j} > 0. 
\]
Passing to the limit $m \to q$ we obtain
\[
	\forall^p_n \ h_j(x_{n}) \geq 0, 
\]
meaning that $x_n \in U \cup Y$ for $p$-almost all $n$. Since $Y$ is a proper subvariety of $V$, we have $x_n \not \in Y$ for $p$-almost all $n$, which finishes the argument.

For the additional part, it is enough to take $Q_n$ to be the intersection of $V$ with any open ball centred at $x_n$ whose closure is disjoint from $Y$.
\end{proof}

  \section{\checkmark\ Torus}\label{sec:torus}

\subsection{\checkmark\ Setup}\label{ssec:torus-setup}\mbox{}
{
Before we approach the proof of Theorem \ref{thm:A} in full generality, we first consider the special case of generalised polynomials which can be represented using Bergelson--Leibman machinery on a torus (or, equivalently, using Leibman's Theorem \ref{thm:Leibman} with degree $1$ classical monomials). This is the simplest non-trivial case of Theorem \ref{thm:A} and it allows us to present some of the main ideas of the proof in a less complicated context. }

{
\textit{Throughout this section, the dimension $d \geq 1$ and the basis $k \geq 2$ are fixed.}
}
\begin{theorem}\label{thm:main-torus}
	Let $x \in \RR^d$ and let $S \subset [0,1)^d$ be a semialgebraic set. Suppose that the set of $n \in \NN$ such that $\fp{k^n x} \in S$ is central. Then the set of $l \in \NN$ such that $\fp{lk^n x} \in S$ for infinitely many $n \in \NN$ is $\IP^*_+$.
\end{theorem}

{\label{ssec:setup-torus}
Let $x^0 \in [0,1)^d$ be an arbitrary point, fixed throughout the section. For $l \in \NN$, let  $T_l \colon [0,1)^d \to [0,1)^d$ be the $\times l$ map given by 
\begin{equation}\label{eq:def-of-T-ab}
T_l(x) := \fp{lx} = (\fp{lx_i})_{i=1}^d \qquad (x \in [0,1)^d).
\end{equation}
Immediately from the definition, we see that $T_l \circ T_m = T_{lm}$ for all $l,m \in \NN$. 
Put $x^n := T^n_k(x^0) = \fp{k^n x^0}$ for $n \in \NN$ and more generally $x^p = \displaystyle\lim_{n \to p} x^n \in [0,1]^d$ for $p \in \beta \NN$. The maps $T_l$ are piecewise affine. For $l \in \NN$ and $p \in \beta \NN_0$, define
\begin{equation}\label{eq:def-of-T|p-ab}
T_l|_p(x) := l x - \lim_{n \to p} \ip{l x^n} = l(x-x^p) + \lim_{n \to p} \fp{l x^n} \qquad (x \in \RR^d),
\end{equation}
so that $T_l|_p$ is an affine map such that $T_l(x^n) = T_l|_p(x^n)$ for $p$-almost all $n$.
Note that $\lim_{n \to p} \fp{l x^n} = \fp{lx^p}$ as long as $\fp{lx^p} \in (0,1)^d$ and that $T^{m}_k|_p(x^p) = x^{m+p}$ ($m \in \NN$).
We are interested in the situation when the points $x^n$ belong to a certain semialgebraic set for many $n$, which motivates us to further denote (cf.\ \eqref{eq:def-of-alglim})
\begin{equation}\label{eq:def-of-Vp-ab}
	V_p := \alglim_{n \to p}\bra{x^n} = \bigcap_{I \in p} \algcl\set{ x^n }{ n \in I} \qquad(p \in \beta \NN_0).
\end{equation}
Recall from Lemma \ref{lem:alg-lim-basic} that the varieties $V_p$ are irreducible and are minimal with respect to the property that $x_n \in V_p$ for $p$-almost all $n$. For $n \in \NN_0$ the corresponding varieties consist of a single point, $V_n = \{x^n\}$. We next investigate the behaviour of $V_p$ under the $\times k$ maps $T_k|_p$. 
}
\begin{lemma}\label{lem:translation}
	Let $p,q \in \beta \NN_0$. Then
\begin{equation}\label{eq:72:00}
	 \forall^p_n \ T_k^n|_q(V_q) \subset V_{p+q}.
\end{equation}	
	Moreover, if $q$ is minimal then
\begin{equation}\label{eq:72:01}
\forall^p_n \ V_{n+q} = T_k^n|_q(V_q) = V_{p+q}.
\end{equation}	
and the set $\set{ V_{u}}{ u \in \beta \NN +q}$ is finite.
\end{lemma}
\begin{proof}
It follows directly from the relevant definitions that
\[
	\forall^p_n \ \forall^q_m \ x^{n+m} = T_k^n|_q(x^m) \in V_{p+q}.
\]
Since the affine map $ T_k^n|_q$ is invertible, we may rewrite this as
\[
	\forall^p_n \ \forall^q_m \ x^m  \in T_k^n|_q^{-1}( V_{p+q}).
\]
Recalling the definition of $V_q$, we conclude that 
\[
	\forall^p_n \ V_q \subset T_k^n|_q^{-1}( V_{p+q} ),
\]
and \eqref{eq:72:00} follows by applying $T_k^n|_q$ to both sides.

If $q$ is additionally minimal then there exists $r \in \beta \NN$ such that $r+p+q = q$. Applying \eqref{eq:72:00} twice we obtain
\begin{align*}
		\forall^r_m & \ \forall^p_n \ T^{m}_k|_{p+q} \circ T^n_k|_{q} (V_{q}) \subset T^{m}_k|_{p+q}( V_{p+q} ), \\
\forall^r_m & \ T^{m}_k|_{p+q}( V_{p+q} ) \subset V_{r+p+q} = V_q. 
\end{align*}
	This is only possible if all inclusions are in fact equalities (cf.\ Lemma \ref{lem:stab-is-alg}), which implies \eqref{eq:72:01}.
	
	To prove the final part of the statement, it will suffice to show that the number of distinct varieties among $V_{n+q}$ ($n \in \NN$) is finite. Suppose otherwise and pick an infinite set $I \subset \NN$ such that the varieties $V_{n+q}$ ($n \in I$) are pairwise distinct. Then there exists $p \in \beta \NN \setminus \NN$ such that $I \in p$, which contradicts \eqref{eq:72:01}.
\end{proof}	
\begin{remark}\label{rmk:translation}
For future reference, note that the above argument uses the fact that $T^n_k|_q \in \AGL(d)$ for any $q \in \beta \NN_0$ and $n \in \NN_0$, and $T^n_k|_q(x^m) = x^{n+m}$ for $q$-almost all $m$, but no other properties of the maps $T^n_k|_q$.
\end{remark}

\subsection{\checkmark\ Uniform recurrence}\mbox{}\label{ssec:torus-ur}
{
In this section we consider the case when $x^q = x^0$ for a minimal idempotent $q$. If $[0,1)^d$ is identified with $\RR^d/\ZZ^d$, thus making the map $T_k$ continuous, then this is equivalent to $x^0$ being uniformly recurrent (Theorem \ref{thm:recurrence-central}). Without such identification, the discontinuity of $T_k$ slightly complicates the picture. The main result in this section is the following variant of Theorem \ref{thm:main-torus}.
}

\begin{theorem}\label{thm:main-ur-torus}
	Let $q \in \beta \NN$ be a minimal idempotent, let $p \in \beta \NN+q$, and let $S \subset [0,1)^d$ be a semialgebraic set. Suppose that $x^q = x^0$ and that $x^n \in S$ for $p$-almost all $n$. Then the set of $l \in \NN$ such that $T_l(x^0) \in S$ is $\IP^*_+$.  
\end{theorem}

Before we proceed further, we address a minor technical issue related to the discontinuity of the maps $T_k$ (cf.\ Corollary \ref{cor:sep-poly-IP}).
\begin{lemma}\label{lem:recurrent_if_x^q=x^0}
	If $q \in \beta \NN$ and $x^q = x^0$ then $x^{p+q} = x^p$ for all $p \in \beta \NN$.
\end{lemma}
\begin{proof}
	It suffices to check that for each $1 \leq i \leq d$ and each $n \in \NN$ we have $x^{n+q}_i = x^n_i$. If $x_i^0 \not \in \QQ$ then this follows from continuity of the map $x \mapsto \fp{k^n x}$ at $x_i^q = x_i^0$. If $x_i^0 \in \QQ$ then it is enough to notice that the denominator of $x_i^0$ is coprime to $k$ and hence the sequence $x^{m}_i$ is periodic.
\end{proof}

{
Under the assumptions of Theorem \ref{thm:main-ur-torus} we have a very satisfactory description of $V_q$, from which the statement of said theorem easily follows.
}
\begin{proposition}\label{lem:idempotent}
	Let $q \in \beta \NN$ be a minimal idempotent and suppose that $x^q = x^0$. Then $V_q$ is an affine space defined over $\QQ$.
\end{proposition}
\begin{proof}
Let $\GG < \AGL(d)$ be the group generated by those among the maps $T_k^n|_q$ ($n \in \NN$) which preserve $V_q$. By Lemma \ref{lem:translation}, $T_k^n|_q \in \GG$ for $q$-almost all $n$. All maps in $\GG$ take the form $T_{a,n}(v) = k^n v - a$ for some $a \in \QQ^d$ and $n \in \ZZ$. Hence, we may identify $\GG$ with a subgroup $\tilde \GG$ of the semidirect product $\QQ^d \rtimes \ZZ$ with the group operation given by $(a,n).(b,m) = (a+k^nb,n+m)$ via the map $T_{a,n} \mapsto (a,n)$. 
	
	The projection onto the second coordinate $(a,n) \mapsto n$ gives rise to a group homomorphism $\varphi \colon \GG \to \ZZ$. The image of $\varphi$ takes the form $\varphi(\GG) = m \ZZ$ for some $m \in \NN$ and we let $\GG' := \varphi^{-1}(0)$ denote the kernel of $\varphi$. Fix $T \in \GG$ with $\varphi(T) = m$, so that any $S \in \GG$ can be uniquely decomposed as $S = S' T^n$ where $S' \in \GG'$ and $n \in \ZZ$. Let $\cW \subset \RR^d$ denote the vector space (over $\RR$) spanned by all $a \in \QQ^d$ such that $T_{a,0} \in \GG'$, and let $\pi \colon \RR^d \to \cW^{\perp}$ denote the orthogonal projection. If $x \in \cW$ then $V_q$ is preserved under the map $v \mapsto v - x$ (cf.\ Lemma \ref{lem:stab-is-alg}), so $\cW \subset V_q-z$ for any point $z \in V_q$. It is also clear that $\cW$ is defined over $\QQ$. It remains to show $V_q \subset z+\cW$ for some point $z \in \QQ^d$.
	
	Consider the sequence of points $\pi(x^n)$ ($n \in \NN$). On one hand, this sequence is $q$-almost everywhere bounded in the sense that
	\begin{equation}\label{eq:73:50}
		\forall^q_n \ \pi(x^n) \in \pi\bra{ [0,1)^d } \subset \Ball\bra{0,\sqrt{d}}.
	\end{equation}
	On the other hand, for $q$-almost all $n$ we have the decomposition $T_k^n|_q = S'_n T^{n/m}$ for some $S_n' \in \GG'$, whence
	\begin{equation}\label{eq:73:51}
		\forall^q_n \ \pi\bra{x^n} = \pi\bra{x^{n+q}} = \pi\bra{T_k^n|_q(x^0)} = \pi\bra{S_n' T^{n/m}(x^0)} = \pi\bra{T^{n/m}(x^0)}.
	\end{equation}
	(Note that $m \mid n$ for $q$-almost all $n$ since $q$ is idempotent.) 
	Let $z \in \QQ^d$ be the unique fixed point of $T$, so that $T^{n/m}(v) = k^n(v-z) + z$. Then
	\begin{equation}\label{eq:73:52}
		\forall^q_n \ \pi(x^n) = k^n \pi(x^0-z)+\pi(z).
	\end{equation}
	Combining \eqref{eq:73:50} and \eqref{eq:73:52}, we conclude that $\pi(x^0-z)= 0$, meaning that $x^0 \in z+\cW$. Hence, $x^n \in z+\cW$ for $q$-almost all $n$, and consequently $V_q \subset z+\cW$. Together with earlier remarks, this finishes the argument.
\end{proof}

\begin{remark}
	We sketch an alternative proof of Proposition \ref{lem:idempotent}, which is perhaps more natural but also less amenable to generalisations.
	The argument splits into two separate steps: first we show that $V_q$ is an affine space, and then we conclude that it is defined over $\QQ$.
	
	For the first part, one can show in general (assuming that $q$ is minimal, but not that $x^q = x^0$ or that $q$ is idempotent) that $x^q$ is a centre of scaling symmetry of $V_q$. Since $q$ is idempotent, by the same token $x^{n+q}$ is a centre of $V_{n+q} = V_q$ for $q$-almost all $n$. Under the assumption that $x^q = x^0$ we conclude that $\alglim_{m \to q} (x^{n+q}) = V_q$, whence centres of $V_q$ are Zariski dense in $V_q$. Since the set of all centres of scaling symmetries of any set is an affine space contained in that set, we conclude that $V_q$ is an affine space.
	
	Secondly, in order to show that $V_q$ is defined over $\QQ$ (as an affine space), it will suffice to show the following more general fact: If $x \in [0,1)^d$, $I \subset \NN$, and the Zariski closure $V \subset \RR^d$ of the set of points $\set{ \fp{l x}}{l \in I}$ is an affine space, then $V$ is defined over $\QQ$. This can be shown by an inductive argument with respect to $d$. The key observation is that if all of the points $ \fp{l x}$ ($l \in I$) satisfy a non-trivial affine relation, then these points also satisfy a non-trivial affine relation with integer coefficients. This observation is noted in \cite[Prop.\ 3.10]{ByszewskiKonieczny-2018}.
\end{remark}

\begin{proof}[Proof of Theorem \ref{thm:main-ur-torus}]
	By Proposition \ref{lem:idempotent}, $V_q$ is an affine space defined over $\QQ$, and by Lemma \ref{lem:translation} so is $V_p$ (in fact, $V_p$ is a translate of $V_q$). Changing the basis (cf.\ Lemma \ref{lem:semialg-wlog-SL}) we may further assume that $V_p = z +\RR^e \times \{0\}^{d-e}$ for some $0 \leq e \leq d$ and $z \in \{0\}^e \times \QQ^{d-e}$. In particular $x^0 \in \RR^e \times \QQ^{d-e}$, so all of the points $T_l(x^0)$, $l \in \NN$, lie in a finite union of translates of $\RR^e \times \{0\}^{d-e}$. We may assume without loss of generality that $S \subset V_p$ and that $S$ is a basic semialgebraic set, in which case it is an open subset of $V_p$. Identifying the orbit closure of $0 \in \RR^d/\ZZ^d$ under the rotation by $x^0$ with a finite union of translates of $V_q$ in the natural way, we have thus represented the set of $l \in \NN$ such that $T_l(x^0) \in S$ as the set of those $l \in \NN$ for which the orbit of a point under a rotation on a torus hits a given open set. All non-empty sets of this form are $\IP^*_+$ by Theorem \ref{thm:recurrence-IP*}.
\end{proof}

\begin{remark}
	The proof of Theorem \ref{thm:recurrence-IP*} for rotations on compact abelian groups is somewhat simpler than the general case. Indeed, it is enough to show that if $p \in \beta \NN$ is idempotent and $\alpha \in \RR/\ZZ$ then $\displaystyle\lim_{n \to p} \alpha n = 0$. This follows directly from the fact that the only idempotent element of $\RR/\ZZ$ is $0$. 
\end{remark}
 
\subsection{\checkmark\ General case}\mbox{}\label{ssec:torus-general}
{
Theorem \ref{thm:main-ur-torus} from the previous section has the following useful corollary, applicable with no assumption on $x^0$. Discontinuity of the fractional part function once again leads to slight technical difficulties, which can be overcome in several ways (cf.\ Corollary \ref{cor:sep-poly-IP}, Lemma \ref{lem:recurrent_if_x^q=x^0}).
}

\begin{corollary}\label{cor:main-ur-1}
	Let $q \in \beta \NN$ be a minimal idempotent, let $p \in \beta \NN + q$, and let $S \subset [0,1]^d$ be a semialgebraic set. Suppose that
	\[
		\forall_n^p \ x^{n+q} \in S.
	\]
	Then the set $\set{ l \in \NN}{ T_l|_q(x^q) \in S}$ is $\mathrm{IP}_{+}^*$.
\end{corollary}
\begin{proof}
	Suppose first that all coordinates of $x^q$ are irrational. Then $T_l|_q(x^q) = \fp{l x^q}$ for all $l \in \NN$. Hence, the claim follows directly from Theorem \ref{thm:main-ur-torus} applied with $\tilde x^0 = x^q$.

	In general, let $\cE \subset \{1,2,\dots,d\}$ be the set of $i$ with $x^q_i \in \QQ$. Since the orbit $T_l|_q(x^q)_\cE$ ($l \in \NN$) is finite, may assume without loss of generality that there exists $z \in \QQ^{\cE}$ such that $x_\cE = z$ for all $x \in S$. (Here and elsewhere, we use the convention where $x_\cE = (x_i)_{i\in \cE}$ for $x \in \RR^d$.) Let also $S' := \fp{S} = \set{ \fp{x}}{x \in S}$. For any $i \in \cE$ one of the following holds: either $x^{n}_i \geq x^q_i$ for $q$-almost all $n$, in which case $T_l|_q(x^q)_i = \fp{lx^q_i}$ for all $l \in \NN$; or $x^{n}_i < x^q_i$ for $q$-almost all $n$, in which case $T_l|_q(x^q)_i = 1-\fp{-lx^q_i}$ for all $l \in \NN$. In either case, $T_l|_q(x^q)_i$ is uniquely determined by $i$ and $\fp{lx^q_i}$. It follows that $T_l|_q(x^q) \in S$ if and only if $\fp{lx^q} \in S'$ ($l \in \NN$). The claim now follows from Theorem \ref{thm:main-ur-torus} applied with $\tilde x^0 = \fp{x^q}$.
\end{proof}

\begin{remark}
	We could also have proven Corollary \ref{cor:main-ur-1} directly by an argument analogous to the proof of Theorem \ref{thm:main-ur-torus}. While this would have reduced the amount of technical issues we need to deal with, we chose the marginally longer route because we believe Theorem \ref{thm:main-ur-torus} to be more intuitively appealing than Corollary \ref{cor:main-ur-1} (in particular, the former avoids the use of the maps $T_l|_q$).
\end{remark}

{
We are now ready to prove a slightly more precise variant of Theorem \ref{thm:main-torus}. 
}

\begin{theorem}\label{thm:main-strong-torus}
	Let $q \in \beta \NN$ be a minimal idempotent, let $p \in \beta \NN + q$ and let $S \subset [0,1)^d$ be a semialgebraic set. Suppose that $x^n \in S$ for $p$-almost all $n$. Then the set of $l \in \NN$ such that $T_l(x^n) \in S$ for $q$-almost all $n$ is $\IP^*_+$.
\end{theorem}
\begin{proof}
We may assume without loss of generality that $S \subset V_p$. For any $l \in \NN$, it follows directly from the definitions that
\begin{equation}\label{eq:94:00}
	 \forall^q_{m} \ T_l(x^m) = l(x^m-x^q) + T_l|_q(x^q).
\end{equation}
In particular, putting $l =k^n$, we find that
\begin{equation}\label{eq:94:01}
	\forall^p_{n} \ \forall^q_{m} \
	x^{n + m} = k^n(x^m - x^q) + x^{n+q} \in S.
\end{equation}
Consider the set 
\begin{equation}\label{eq:94:02}
	R := \set{(x,t) \in \RR^d \times \RR }{\forall^q_{m} \ x + t(x^m - x^q) \in S }.
\end{equation}
It follows from Proposition \ref{prop:semialg-limit} (cf.\ Corollary \ref{cor:semialg-limit}) that $R$ is semialgebraic, and \eqref{eq:94:01} translates into
\begin{equation}\label{eq:94:11}
	\forall^p_{n} \ \bra{x^{n+q}, k^n} \in R.
\end{equation}
We aim to show that the set of $l \in \NN$ with $\bra{T_l|_q(x^q), l} \in R$ is $\IP^*_+$.

It follows from Proposition \ref{prop:sep-semialg-IP} (applied with $x_n = x^{n+q}$ and $t_n = k^n$, cf.\ Corollary \ref{cor:sep-poly-IP}) that there exist a semialgebraic set $Q$, an integer $l_0 \geq 0$ and an ultrafilter $r \in \NN + q$ such that 
\begin{align*}
	\forall^{r}_n \ x^{n+q} \in Q \quad \text{ and } \quad Q \times [l_0,\infty) \subset R.
\end{align*}
By Corollary \ref{cor:main-ur-1} there exists an $\IP^*_+$ set ${L} \subset \NN$ of such that $T_l|_q(x^q) \in Q$ for all $l \in {L}$. Hence, $\bra{T_l|_q(x^q), l} \in R$ for all $l$ in the $\IP^*_+$ set ${L} \cap [l_0, \infty)$.
\end{proof}
 
  \section{Generalised $\times k$ maps}\label{sec:setup}

\subsection{\checkmark\ Basic definitions}\mbox{}\label{ssec:setup-basic}
{
In this section we set up notation and introduce objects which will be crucial in the proof of the general case of Theorem \ref{thm:A}. The key insight is that for suitable families of generalised monomials one can construct well-behaved analogues of the $\times k$ maps familiar from Section \ref{sec:torus}.
}

{
In order to conveniently index the generalised monomials discussed in Section \ref{sec:prelims},
we introduce the set of bracket indices $\cB$, consisting of formal expressions containing positive integers and brackets $\fpp{\cdot}$. Formally, $\cB$ is defined as the smallest family such that $\NN \subset \cB$ and if $\kappa, \lambda \in \cB$ then $\kappa \fpp{\lambda} \in \cB$. Expressions differing only by the ordering of factors are considered equal: $\kappa \fpp{\lambda_1}\dots\fpp{\lambda_s} = \kappa \fpp{\lambda_{\pi(1)}}\dots\fpp{\lambda_{\pi(s)}}$ if $\pi$ is a permutation of $\{1,2,\dots,s\}$.

We use $\cB$ to index the generalised monomials which can be constructed from a given sequence of classical monomials in a single variable. Let $v_i \in \RR[t]$ ($i \in \NN$) be a sequence of monomials, given by
\begin{align}\label{eq:def-of-nu-N}
v_i(t) = \alpha_i t^{d_i}, \quad \text{ where } \alpha_i \in \RR \text{ and } d_i \in \NN \qquad (i \in \NN).
\end{align}
We then extend $v_\bullet$ and $d_\bullet$ to $\cB$, defining inductively
\begin{equation}\label{eq:def-of-nu-B}
v_\mu(t) := v_{\kappa}(t) \fp{v_\lambda(t)} \text{ and } d_{\mu} := d_{\kappa} + d_{\lambda}  \text{ for any } \mu = \kappa \fpp{\lambda} \in \cB \setminus \NN.
\end{equation}
The \emph{height} of an index $\mu \in \cB$, denoted by $h_\mu$, is defined as the maximal number of nested brackets appearing in $\mu$. More precisely, $h_i = 0$ for $i \in \NN$ and
\begin{equation}\label{eq:def-of-height}
h_\mu :=  \max\{ h_\kappa, h_\lambda+1\}  \text{ for any } \mu = \kappa \fpp{\lambda} \in \cB \setminus \NN.
\end{equation}
It is a matter of a simple exercise to see that these definitions are well-posed, in the sense that if $\mu \in \cB$ has two different representations $\mu = \kappa \fpp{\lambda} = \kappa' \fpp{\lambda'}$ then both of them give the same values for $v_{\mu}$, $d_\mu$ and $h_\mu$.
}

{
There are several partial orders of interest on $\cB$. We will say that an index $\nu \in \cB$ is \emph{derivable} from $\mu \in \cB$, denoted $\nu \preceq \mu$, if it can be obtained by removing some factors from $\mu$. Formally, $\preceq$ is defined as the smallest partial order on $\cB$ such that
\begin{align}\label{eq:def-of-derivable}
\kappa &\preceq \kappa\fpp{\lambda} 
&\text{ and }&& 
\text{ if } \kappa \preceq \kappa' \text{ and }\lambda \preceq \lambda'
\text{ then }
\kappa & \fpp{\lambda} \preceq \kappa' \fpp{\lambda'}  \quad (\kappa,\kappa',\lambda,\lambda' \in \cB).
\end{align}
The set $\cB$ also carries the orders induced from $\NN_0$ by the height $h_\bullet$ and the grading $d_\bullet$ maps (which, of course, depends on the degrees $d_i$, $i \in \NN$). These orders are compatible in the sense that $\nu \preceq \mu$ implies that $h_{\nu} \leq h_{\mu}$ and $d_{\nu} \leq d_{\mu}$ ($\mu,\nu \in \cB$). 
}

{
With the notation introduced above, Leibman's Theorem \ref{thm:Leibman} (in one variable) can be rephrased as saying that for any generalised polynomial $g \colon \ZZ \to \RR^e$ there exists a choice of coefficients $\alpha_i$ and degrees $d_i$ ($i \in \NN$) as well as a finite set $\cD \subset \cB$ such that $g(n)$ is a piecewise polynomial function of $\bra{ v_\kappa^{\alpha}(t) }_{\kappa \in \cD}$. Note that one can always enlarge $\cD$, so we can additionally assume that $\cD$ is downwards closed in the sense that
\begin{equation}\label{eq:def-of-D}
\text{ if } \mu \in \cD \text{ and } \nu \preceq \mu \text{ then } \nu \in \cD.
\end{equation}
We define the \emph{complexity} of $\cD$ to be the vector $\complexity(\cD) = (c_0,c_1,\dots) \in \NN_0^\infty$, where $c_i = \abs{\set{\mu \in \cD}{h_\mu = i}}$. The set of eventually zero sequences taking values in $\NN_0$ is well-ordered by the reverse lexicographic order, which can naturally be prolonged to sequences with values in $\NN_0 \cup \{\infty\}$. 

Our long term strategy is to use Theorem \ref{thm:Leibman} mentioned above to represent the characteristic function of the set $E$ appearing in Theorem \ref{thm:A}. This will ultimately allow us to deduce Theorem \ref{thm:A} from a recurrence statement for generalised monomials. 
Because the constructions we plan to carry out depend on the index set $\cD$ and the grading $d_\bullet$, we accept the following convention.
}

\textit{
We fix once and for all a grading $d_\bullet$ satisfying \eqref{eq:def-of-nu-B} and additionally assume that for each $d \in \NN$ there exist infinitely many $i \in \NN$ such that $d_i = d$.
Throughout the paper, $\cD \subset \cB$ denotes a finite set satisfying \eqref{eq:def-of-D}. All objects we construct are allowed to depend on $\cD$ and $d_\bullet$ unless explicitly stated otherwise. For $\alpha \in \RR^{\cD \cap \NN}$ and $\mu \in \cD$, we let $v^{\alpha}_{\mu}(t)$ denote the generalised polynomials given by \eqref{eq:def-of-nu-N} and \eqref{eq:def-of-nu-B}.
}

Having fixed the choice of $\cD$, we collect the generalised polynomials $v_\mu^\alpha$ for $\mu \in \cD$ into a single (multidimensional) generalised polynomial $v^\alpha \colon \RR \to \RR^{\cD}$ given by
\begin{equation}\label{eq:def-of-nu}
	v^{\alpha}(t) = \bra{ v_\mu^{\alpha}(t) }_{\mu \in \cD}. 
\end{equation}

\begin{remark}\label{rmk:apology}
	At this point we owe the Reader a few words of justification for the conventions we have assumed. Subsequent sections are quite heavy on definitions, many of which depend on $\cD$ and $d_\bullet$, as well as other objects. Bearing in mind that the notation already gets rather cumbersome in several places, we would rather not make the dependence on $\cD$ and $d_\bullet$ explicit. On the other hand, we do need to occasionally alter the choice of $\cD$; indeed, the proof of our main result proceeds by induction on $\cD$. For this reason, we give $\cD$ a rather unsatisfactory ontological status of an immutable object which we nevertheless occasionally alter. This forces us to use phrases such as ``let $ v^{\beta}(t)$ ($\beta \in \RR^{\cE \cap \NN}$) be the same as $v^{\alpha}(t)$ except with $\cE$ in place of $\cD$'', which is formally meaningless, but hopefully understandable for the Reader. The author considers this solution the lesser evil. Fortunately, we can avoid the analogous problem for the grading $d_\bullet$ by letting it be fixed throughout the paper. This has the minor downside that we cannot always assume that $\cD \cap \NN$ is an initial segment of $\NN$, but there are no particularly strong reasons why we would want to assume that in the first place. The values $d_i$ for $i \in \NN \setminus \cD$ are mostly irrelevant.  
\end{remark}

\begin{example}
\label{ex:toy-1}
	Let us consider a running example where $\cD =  \{1,2,1\fpp{2},2\fpp{1},3\}$ and $d_1 = d_2 = 1$ and $d_3 = 2$. This choice of $\cD$ satisfies \eqref{eq:def-of-D}, \eqref{eq:def-of-nu-B} yields $ d_{1\fpp{2}} = d_{2\fpp{1}} = 2$ and \eqref{eq:def-of-height} yields $h_{1} = h_{2} = h_3 = 0$ and $ h_{1\fpp{2}} = h_{2\fpp{1}} = 1$.
	
The order $\preceq$ is given by $1 \prec 1\fpp{2}$, $2 \prec 2\fpp{1}$; all remaining pairs of distinct indices are incomparable. The order induced by the degree $d_\bullet$ separates the elements of $\cD$ into two equivalence classes $\{1,2\}$ (degree $1$) and $\{1\fpp{2},2\fpp{1},3\}$ (degree $2$).
Similarly, the order induced by the height separates the elements of $\cD$ into two equivalence classes $\{1,2,3\}$ (height $0$) and $\{1\fpp{2},2\fpp{1}\}$ (height $1$).
	
For notational convenience we will always write the coordinates in $\cD$ in the same order as above. In particular,
\[v^{\alpha}(n) = \big( \alpha_1 n,\ \alpha_2 n,\ \alpha_1 n \fp{\alpha_2n},\ \alpha_2n \fp{\alpha_1n},\ \alpha_3 n^2 \big), \qquad (\alpha \in \RR^3). \qedhere\]
\end{example}
 
\subsection{Geometry}\label{ssec:setup-geo}\mbox{}
{
In this section we introduce some geometric objects and constructions which will be useful in subsequent sections. The space $\RR^{\cD}$ naturally decomposes into the direct sum of spaces corresponding to different degrees. Let
\begin{equation}\label{eq:def-of-s}
	s := \max\set{d_\kappa}{\kappa \in \cD} \text{ and } \cD_j := \set{\kappa \in \cD}{ d_\kappa = j} \text{ for } 1 \leq j \leq s.
\end{equation}
Define also the subspaces 
\begin{equation}\label{eq:def-of-Vj}
	\cV_j := \set{ x \in \RR^{\cD} }{ x_\kappa = 0 \text{ for all } \kappa \in \cD \setminus \cD_j} = \linspan\set{ e_\kappa}{\kappa \in \cD_j},
\end{equation}
where $e_\mu = (e_{\mu,\nu})_{\nu \in \cD}$ denotes the basis vector with $e_{\mu,\nu} = \braif{\mu = \nu}$ ($\mu,\nu \in \cD$). It is clear that $\RR^{\cD}$ is the orthogonal sum of $\cV_j$ over $1 \leq j \leq s$ (here and elsewhere, we endow $\RR^{\cD}$ with the standard scalar product).
}

\begin{example}\label{ex:toy-15}
	In the running example discussed above (see Example \ref{ex:toy-1}) we have $s = 2$, $\cD_1 = \{1,2\}$ and $\cD_2 = \{ 1\fpp{2},2\fpp{1},3\}$. Writing the coordinates in the usual order, we have
\begin{align*}
		\cV_1 &= \RR e_1 + \RR e_2 = \set{ (x_1,x_2,0,0,0)}{ x_1,x_2 \in \RR}, \text{ and }\\
		\cV_2 &= \RR e_{1\fpp{2}} + \RR e_{2\fpp{1}}+ \RR e_3 = \set{ (0,0,x_{1\fpp{2}},x_{2\fpp{1}},x_3)}{ x_{1\fpp{2}},x_{2\fpp{1}},x_3 \in \RR}.
\end{align*}
\end{example}

{
Following the usual convention, we write $\End(\cD) = \End(\RR^{\cD})$ for the ring of $\cD \times \cD$ matrices, and assume the same convention for $\GL(\cD),$ $\SL(\cD)$, and other matrix groups. For matrices with integer or rational coefficients we write $\End(\cD,\ZZ)$, $\End(\cD,\QQ)$, etc. We will say that a matrix $A = (A_{\mu,\nu})_{\mu,\nu \in \cD} \in \End(\cD)$ is \emph{lower triangular} (with respect to $\preceq$) if $A_{\mu,\nu} \neq 0$ implies $\mu \succeq \nu$ ($\mu,\nu \in \cD$). In particular, any lower triangular matrix written in the block form corresponding to the decomposition $\cD = \bigcup_{j=1}^s \cD_j$ has only zero blocks above the diagonal.

We further introduce the notion of a \emph{standard} matrix. We will say that a matrix $A \in \End(\cD)$ is \emph{standard} if it is lower triangular and there exists $t > 0$ such that the diagonal entries of $A$ are given by $A_{\mu,\mu} = t^{d_\mu}$ for all $\mu \in \cD$.  We let $\ST(\cD)$ denote the set of standard matrices:
\begin{equation}\label{eq:def-of-ST(D)}
	\ST(\cD) := \set{ A \in \End(\cD) }{ A \text{ is lower triang.{} and } \operatorname{diag}(A) = (t^{d_\mu})_{\mu \in \cD} }.
\end{equation}
(The choice of the name is motivated by the connection between these groups and the multiplication maps $x \mapsto \fp{kx}$, elucidated in the upcoming discussion.)
Standard lower triangular matrices form a group. If a standard lower triangular matrix is written in the block form then the diagonal blocks are proportional to the identity.
}

\begin{example}\label{ex:toy-16}
In the running example, a matrix $A \in \End(\cD)$ is a member of $\ST(\cD)$ if and only if it takes the form 
\[
	A = 
\begin{bmatrix}
t & 0 & 0 & 0 & 0  \\
0 & t & 0 & 0 & 0 \\
\ast & 0 & t^2  & 0 & 0 \\
0 & \ast & 0 & t^2 & 0 \\
0 & 0 & 0 & 0 & t^2
\end{bmatrix},
\]
where $t > 0$ and $\ast$ denote unspecified real entries.
\end{example}

{
We let $\st(\cD)$ denote the Lie algebra of $\ST(\cD)$, and $\Lambda$ denote the diagonal matrix with $\Lambda_{\mu,\mu} = d_\mu$ ($\mu \in \cD$). The spaces $\cV_j$ ($1 \leq j \leq s$) are the eigenspaces of $\Lambda$ and $\st(\cD)$ is spanned by $\Lambda$ and strictly lower triangular basic matrices $E^{\mu,\nu}$ ($\mu, \nu \in \cD,\ \mu \succ \nu$) given by $E^{\mu,\nu}_{\kappa,\lambda} = \braif{\kappa = \mu \wedge \lambda = \nu}$ ($\kappa,\lambda \in \cD$). We also define $\Delta_t$ to be the diagonal matrix with $(\Delta_t)_{\mu,\mu} = t^{d_\mu}$ ($t \in \RR$), so that in particular $\Delta_k = \exp(\log(k) \Lambda)$ ($k \in \NN$). We call a  matrix $A \in \ST(\cD)$ \emph{special} if all of its diagonal entries are equal to $1$ (i.e., if it is unipotent). The group of special lower triangular\ matrices is denoted by $\ST'(\cD)$ and the corresponding Lie algebra, denoted by $\st'(\cD)$, consists of strictly lower triangular matrices. Hence, we have the decomposition $\st(\cD) = \st'(\cD) + \RR \Lambda$ and any $A \in \ST(\cD)$ can be uniquely written as $A = \Delta_t A'$ with $t > 0$ and $A' \in \STp(\cD)$.
}

\begin{example}\label{ex:toy-17}
In the running example, the matrices $\Lambda$ and $\Delta_k$ are given by
\begin{align*}
	\Lambda = 
\begin{bmatrix}
1 & 0 & 0 & 0 & 0  \\
0 & 1 & 0 & 0 & 0 \\
0 & 0 & 2  & 0 & 0 \\
0 & 0 & 0 & 2 & 0 \\
0 & 0 & 0 & 0 & 2
\end{bmatrix}, \qquad
	\qquad  \Delta_k = 
\begin{bmatrix}
k & 0 & 0 & 0 & 0  \\
0 & k & 0 & 0 & 0 \\
0 & 0 & k^2  & 0 & 0 \\
0 & 0 & 0 & k^2 & 0 \\
0 & 0 & 0 & 0 & k^2
\end{bmatrix}.
\end{align*}

The Lie group $\ST'(\cD)$ and the Lie algebra $\st'(\cD)$ consist of all matrices $A$ and $X$ respectively which have the form 
\begin{align*}
	A = 
\begin{bmatrix}
1 & 0 & 0 & 0 & 0  \\
0 & 1 & 0 & 0 & 0 \\
\ast & 0 & 1  & 0 & 0 \\
0 & \ast & 0 & 1 & 0 \\
0 & 0 & 0 & 0 & 1
\end{bmatrix},  \qquad
	\qquad  X = 
\begin{bmatrix}
0 & 0 & 0 & 0 & 0  \\
0 & 0 & 0 & 0 & 0 \\
\ast & 0 & 0  & 0 & 0 \\
0 & \ast & 0 & 0 & 0 \\
0 & 0 & 0 & 0 & 0
\end{bmatrix}.
\end{align*}
\end{example}

In order to work with affine maps more conveniently, we identify $\RR^{\cD}$ with $\{1\} \times \RR^{\cD} \subset \RR \times \RR^{\cD}$. Under this identification, an affine map on $\RR^{\cD}$ given by $x \mapsto Ax + b$ corresponds to the linear map $(u,x) \mapsto (u,A x + ub)$ on $\RR \times \RR^{\cD}$. We define 
\begin{equation}\label{eq:def-of-bST}
\bST(\cD) := \set{\begin{bmatrix}
1 & 0 \\ b & A\end{bmatrix}}{ A \in \ST(\cD),\ b \in \RR^{\cD}}.
\end{equation}
\emph{Mutatis mutandis}, $\bST(\cD)$ can be identified with $\ST(\cD \cup \{0\})$ where $0 \prec \mu$ for all $\mu \in \cD$ and $d_0 = 0$. Because of this identification, we will occasionally apply to $\bST(\cD)$ results which were only formally proved for $\ST(\cD)$. We analogously let $\bar{\st}(\cD)$ denote the Lie algebra of $\bar{\ST}(\cD)$, and define $\bar{\ST}{}'(\cD)$ and $\bar{\st}{}'(\cD)$ accordingly. We further put
\begin{align}\label{eq:def-of-bLambda}
\bar \Lambda:= \begin{bmatrix} 0 & 0 \\ 0 & \Lambda \end{bmatrix},
\qquad \bar \Delta_k := \begin{bmatrix} 0 & 0 \\ 0 & \Delta_k \end{bmatrix}.
\end{align}
 
\subsection{Multiplication by $k$}\mbox{}\label{ssec:setup-Tk}
{
Recall that in the abelian case discussed in Section \ref{sec:torus}, the $\times k$ maps $T_k \colon [0,1)^d \to [0,1)^d$ played an important role. These simple maps almost trivially have several desirable properties: they are piecewise affine, $T_k(x)$ is given by a generalised polynomial formula in $x$ and $k$, and $T_k \circ T_l = T_{kl}$ for all $l,k \in \NN$. In this section we will construct maps on $[0,1)^\cD$ with analogous properties.
}

\begin{example}\label{ex:toy-2}
	We continue with the running example. A direct computation shows that 
\[
	S_k(v^{\alpha}(m)) = v^{\alpha}(km) 
	\quad \text{ for all $\alpha \in \RR^{\cD\cap \NN}$ and $k,m \in \NN$,}
\]
	where the maps $S_k \colon \RR^\cD \to \RR^\cD$ are defined for $x \in \RR^{\cD}$ by
	\[
		S_k(x) := \big( k x_1,\ k x_2,\ k^2 x_{1\fpp{2}} - kx_1[k\fp{x_2}],\ k^2 x_{2\fpp{1}} - kx_2[k\fp{x_1}],\ k^2 x_3 \big).
	\]
For $x \in [0,1)^{\cD}$, put also $T_k(x) := \fp{S_k(x)}$. Another standard computation yields
\[
	T_k(\fp{ v^{\alpha}(m) } ) = \fp{v^{\alpha}(km)} \quad
	\text{ for all $\alpha \in \RR^{\cD\cap \NN}$ and $k,m \in \NN$.}
\]

Let $k \in \NN$. In order to record some noteworthy properties of the maps $S_k$ and $T_k$ defined above, it is convenient to write them in the matrix form 
	\begin{align*}
		S_k(x) &= A_k(x) x &\text{and}&&  T_k(x) &= A_k(x)x - b_k(x),
	\end{align*}
	 where $b_k(x) := \ip{S_k(x)}$ and the matrix $A_k(x)$ is given by  
\begin{equation}
	A_k(x) :=	
\begin{bmatrix}
k & 0 & 0 & 0 & 0 \\
0 & k & 0 & 0 & 0 \\
-k[k\fp{x_2}] & 0 & k^2 & 0 & 0 \\
0 & -k[k\fp{x_1}] & 0 & k^2 & 0 \\
0 & 0 & 0 & 0 & k^2
\end{bmatrix}.
\end{equation}
Under the identification of affine maps on $\RR^\cD$ with linear maps on $\RR \times \RR^{\cD}$, for $x \in [0,1)^{\cD}$ we can write $T_k(x) = \bar{A}_k(x)x$, where the matrix $\bar{A}_k(x)$ is given by
\begin{equation}\label{eq:751:1}
	\bar{A}_k(x) :=	
\begin{bmatrix}
1 & 0 & 0 & 0 & 0 & 0 \\
- [kx_1] & k & 0 & 0 & 0 & 0 \\
- [kx_2] &0 & k & 0 & 0 & 0 \\
-\ip{k^2 x_{1\fpp{2}} - k x_1 \ip{kx_2}} &-k[k{x_2}] & 0 & k^2 & 0 & 0 \\
-\ip{k^2 x_{2\fpp{1}} - k x_2 \ip{kx_1}} &0 & -k[k{x_1}] & 0 & k^2 & 0 \\
- [k^2x_3] &0 & 0 & 0 & 0 & k^2
\end{bmatrix}.
\end{equation}
The matrices $A_k$ and $\bar{A}_k$ are standard lower triangular and have integer entries:
\[
	A_k(x) \in \ST(\cD,\ZZ) \text{ for all } x \in \RR^\cD
	\text{ and } \bar A_k(x) \in \bar \ST(\cD,\ZZ) \text{ for all } x \in [0,1)^\cD.
\]
In fact, more is true: Many of the entries of $\bar{A}_k(x)$ are divisible by powers of $k$, and extracting the highest powers of $k$ apparent from formula \eqref{eq:751:1} yields
\[ \bar{A}_k(x) \Delta_k^{-1} \in \bar\ST{}'(\cD,\ZZ) \text{ for all } x \in [0,1)^\cD.\]

The dependence of $\bar{A}_k(x)$ on $x$ is relatively mild. Firstly, $\bar{A}_k(x)$ is a generalised polynomial in $k$ and $x$. In the course of the argument it will be important to keep track of which entries of $\bar{A}_k(x)$ depend on which coordinates of $x$, but we defer the precise statement until later. Here, we just remark that $A_k(x)$ depends only on $x_1$ and $x_2$ but not the remaining coordinates $x_{1\fpp{2}}, x_{2\fpp{1}},x_3$ of $x$.

If both $x \in [0,1)^{\cD}$ and $T_k(x)$ are given, then $\bar{A}_k(x)$ can described by an even simpler formula. For $i \in \{1,2\}$ we have
\[
	-k[k{x_i}] = k^2 x_i - k\fp{kx_i} = k^2 x_i - k T_k(x)_i,
\]
whence $A_k(x)$ is a polynomial function in $k$, $x$ and $T_k(x)$. We can also express $b_k(x)$ as
\[
	b_k(x) = [S_k(x)] = A_k(x)x - T_k(x)
\]
and hence $\bar{A}_k(x)$ is a polynomial function in $k,x$ and $T_k(x)$. 
\end{example}

{
Our next goal is to generalise the construction from the above example. We will construct families of maps (generalising the $\times k$ map on the unit cube and the scaling by $k$, respectively)
	\begin{align}\label{eq:64:80}
	T_k \colon [0,1)^{\cD} &\to [0,1)^{\cD}, & 
	S_k \colon \RR^\cD &\to \RR^{\cD}, & (k \in \NN),&
	\end{align} 
	which are given in the matrix form by
	\begin{align}
	S_k(x) &:= A_k(x)x &(x \in \RR^{\cD}), \label{eq:def-of-T-tilde}
	\\ T_k(x) &:= \fp{S_k(x)} = A_k(x)x - b_k(x) = \bar{A}_k(x)x &(x \in [0,1)^{\cD}), \label{eq:def-of-T}
	\end{align}
	where we use the identification of affine maps on $\RR^\cD$ with linear maps on $\RR \times \RR^\cD$ discussed in Section \ref{ssec:setup-geo}, and $\bar A_k$ and $b_k$ are given by
	\begin{align}\label{eq:def-of-A-bar}
	b_k(x) := \ip{S_k(x)} = \ip{ A_k(x) x}, \qquad
	\bar A_k(x) := 
	\begin{bmatrix}
	1 & 0 \\
	-b_k(x) & A_k(x)
	\end{bmatrix}.
	\end{align}
Treating formulae \eqref{eq:def-of-T-tilde}, \eqref{eq:def-of-T} and \eqref{eq:def-of-A-bar} as definitions of $S_k$, $T_k$, $b_k$ and $\bar{A}_k$, it remains to define $A_k$. We do so inductively, row by row. Note that the corresponding entry $S_k(x)_\mu = \sum_{\nu \in \cD} A_k(x)_{\mu,\nu}x_\nu$ becomes determined as soon as we construct the row $A_k(x)_{\mu,\ast}$. For $i \in \cD \cap \NN$ we put $A_k(x)_{i,\nu} := k^{d_i} \braif{i=\nu}$, so that 
$S_k(x)_i = k^{d_i} x_i$. 

Consider now $\mu = \kappa \fpp{\lambda} \in \cD \setminus \NN$ and assume that the rows $A_{k}(x)_{\kappa,\ast}$ and $A_{k}(x)_{\lambda,\ast}$ have been constructed. As long as the fragment of $A_{k}(x)$ constructed so far is lower triangular (we will shortly show that this is the case) we have expansions
\begin{align}
\label{eq:76:00} 	S_{k}(x)_\kappa &= \sum_{\sigma \preceq \kappa} A_{k}(x)_{\kappa,\sigma} x_{\sigma},\\
\label{eq:76:01} 	S_{k}(x)_{\lambda} &= \sum_{\tau \preceq \lambda} A_{k}(x)_{\lambda,\tau} x_{\tau}.
\end{align}
Let $c_{k,\lambda}(x) \in \ZZ$ denote the correction term uniquely determined by
\begin{equation}
\label{eq:def-of-c} 
\fp{S_{k}(x)_{\lambda}} = \sum_{\tau \preceq \lambda} A_{k}(x)_{\lambda,\tau} \fp{x_{\tau}} + c_{k,\lambda}(x),
\end{equation}
Combining \eqref{eq:76:00} and \eqref{eq:def-of-c} gives
\begin{equation}
\label{eq:76:13} 
S_{k}(x)_\kappa \fp{S_{k}(x)_{\lambda}} = 
\sum_{\sigma \preceq \kappa}  \sum_{\tau \preceq \lambda} A_{k}(x)_{\kappa,\sigma} A_{k}(x)_{\lambda,\tau} x_{\sigma} \fp{x_{\tau}} + \sum_{\sigma \preceq \kappa}  A_{k}(x)_{\kappa,\sigma} c_{k,\lambda}(x) x_{\sigma}.
\end{equation}
This motivates us to define $A_k(x)_{\mu,\ast}$ by
\begin{equation}
\label{eq:def-of-A} 
A_k(x)_{\mu,\nu} = 
\sum_{\sigma \preceq \kappa} \sum_{\tau \preceq \lambda} A_{k}(x)_{\kappa,\sigma} A_{k}(x)_{\lambda,\tau} \braif{\sigma\fpp{\tau} = \nu }  + A_{k}(x)_{\kappa,\nu} c_{k,\lambda}(x).
\end{equation}
}

\begin{proposition}\label{prop:Tk-basic}
	The maps defined by  \eqref{eq:def-of-T-tilde}, \eqref{eq:def-of-T}, \eqref{eq:def-of-A-bar} and \eqref{eq:def-of-A} have the following properties.
\begin{enumerate}
\item\label{it:Tk-basic:integral} $A_k(x) \in \ST'(\cD,\ZZ)\Delta_k \subset \ST(\cD,\ZZ)$ for all $k \in \NN,\ x \in \RR^{\cD}$;
\item\label{it:Tk-basic:times-k} $v^{\alpha}(km) = S_k(v^{\alpha}(m))$ for all $k,m \in \NN$ and $\alpha \in \RR^{\cD}$; 
\item\label{it:Tk-basic:dependence} for any $\mu, \nu \in \cD$, the coefficient $A_k(x)_{\mu,\nu}$ depends only on $k$ and $\fp{x}_{\xi}$ where $\xi \in \cD$, $d_\xi + d_\nu \leq d_\mu$ and $h_{\xi} < h_{\mu}$;
\item\label{it:Tk-basic:gen-poly} there exists a generalised polynomial map 
\[ \NN \times [0,1)^\cD \ni (k,x) \mapsto A_k(x) \in \ST(\cD); \]
\item\label{it:Tk-basic:poly} there exists a polynomial map 
\[ \NN \times [0,1)^\cD \times [0,1)^\cD \ni (k,x,T_k(x))\mapsto A_k(x) \in \ST(\cD);\] 
\item\label{it:Tk-basic:commute} $S_k \circ S_l = S_l \circ S_k = S_{kl}$ and ${T}_k \circ {T}_l = {T}_l \circ {T}_k = {T}_{kl}$ for all $k,l \in \NN$.
\end{enumerate}
\end{proposition}
\begin{proof}
\begin{enumerate}[wide]
\item We need to show that for each $\mu,\nu \in \cD$ and $x \in \RR^{\cD}$ the entry $A_k(x)_{\mu, \nu}$ is either: zero if $\nu \not\preceq \mu$; equal to $k^{d_\mu}$ if $\nu = \mu$; or an integer divisible by $k^{d_\mu}$ if $\nu \prec \mu$.    We proceed by induction on $\mu$. If $\mu \in \cD \cap \NN$ then the above properties follow directly from the definition of $A_k(x)$. Suppose next that $\mu = \kappa \fpp{\lambda}$ and the claim has been shown for $\kappa$ and $\lambda$. In particular, since the row $A_k(x)_{\lambda,\ast}$ has integer entries, the term $c_{k,\lambda}(x)$ defined by \eqref{eq:def-of-c} takes integer values for all $x \in \RR^{\cD}$. It follows directly from \eqref{eq:def-of-A} that $A_k(x)_{\mu,\nu} = 0$ if $\nu \not\preceq \mu$ and that
\[
A_k(x)_{\mu,\mu} = 
A_{k}(x)_{\kappa,\kappa} A_{k}(x)_{\lambda,\lambda} = k^{d_\kappa+d_\lambda} = k^{d_\mu}.
\] 
Lastly, if $\nu \prec \mu$ then each of the summands $A_{k}(x)_{\kappa,\sigma} A_{k}(x)_{\lambda,\tau}$ with $\sigma \preceq \kappa$, $\lambda \preceq \tau$ and $\nu = \kappa \fpp{\lambda}$ appearing in \eqref{eq:def-of-A} is divisible by $k^{d_\tau + d_\sigma} = k^{d_\nu}$. Similarly $A_k(x)_{\kappa,\nu}$ is divisible by $k^{d_\nu}$. Hence, $A_k(x)_{\mu,\nu}$ is divisible by $k^{d_\nu}$. 

\item Because $A_k(x)$ are lower triangular, it follows directly from \eqref{eq:76:13} and \eqref{eq:def-of-A} that 
\begin{equation}
\label{eq:76:15} 
S_{k}(x)_{\mu} = 
S_{k}(x)_\kappa \fp{S_{k}(x)_{\lambda}} + 
\sum_{\sigma \preceq \kappa}  \sum_{\tau \preceq \lambda} A_{k}(x)_{\kappa,\sigma} A_{k}(x)_{\lambda,\tau} \bra{x_{\sigma\fpp{\tau}}- x_{\sigma} \fp{x_{\tau}}}
\end{equation}
 for any $\mu = \kappa \fpp{\lambda} \in \cD \setminus \NN$. Substituting $x = v^{\alpha}(m)$ and recalling that 
 \[
 v^{\alpha}_{\sigma\fpp{\tau}}(m) = v^{\alpha}_{\sigma}(m) \fp{v^{\alpha}_{\tau}(m)} \text{  for all } \sigma, \tau \in \cB \text{ such that } \sigma\fpp{\tau}\in \cD,
 \]
we conclude that 
\begin{equation}
\label{eq:76:70} 
S_{k}(v^{\alpha}(m))_{\mu} = S_{k}(v^{\alpha}(m))_\kappa \fp{S_{k}(v^{\alpha}(m))_{\lambda}}
\text{ for any }
\mu = \kappa \fpp{\lambda} \in \cD \setminus \NN.
\end{equation}
For $i \in \cD \cap \NN$ we have
\begin{equation}
\label{eq:76:71} 
 S_{k}(v^{\alpha}(m))_{i} = k^{d_i} v^{\alpha}_i(m) = \alpha_i k^{d_i} m^{d_i} = v^{\alpha}_i(km).
\end{equation}
Combining \eqref{eq:76:70}, \eqref{eq:76:71} and \eqref{eq:def-of-nu-B} yields $S_{k}(v^{\alpha}(m)) = v^{\alpha}(km)$.

\item Let $\mu = \kappa \fpp{\lambda} \in \cB$. 
It is clear from \eqref{eq:def-of-A} that if $A_k(x)_{\kappa,\tau}$ and $A_k(x)_{\lambda, \sigma}$ ($\tau, \sigma \in \cD$) all depend only on $k$ and $\fp{x}$ then so does and $A_k(x)_{\mu,\nu}$ ($\nu \in \cD$). It remains to pin down the set $\cC_{\mu,\nu} \subset \cD$ of indices on which $A_k(x)_{\mu,\nu}$ depends. Since the diagonal entries of $A_k(x)$ do not depend on $x$, we can take $C_{\mu,\mu} = \emptyset$ for all $\mu \in \cD$. For general $\mu, \nu$, direct inspection of \eqref{eq:def-of-A} show that if $\xi \in \cC_{\mu,\nu}$ then one of the following possibilities holds:
\begin{enumerate}
\item there exist $\sigma \preceq \kappa$ and $\tau \preceq \lambda$ such that $\nu = \sigma \fpp{\tau}$ and 
	\begin{inparaenum}
	\item $\xi \in \cC_{\kappa, \sigma}$, or
	\item $\xi \in \cC_{\lambda,\tau}$;
	\end{inparaenum}
\item $\xi \in \cC_{\kappa, \nu}$; 
\item $\nu \preceq \kappa$ and 
	\begin{inparaenum}
	\item $\xi \in \cC_{\lambda, \tau}$ for some $\tau \preceq \lambda$, or
	\item $\xi \preceq \lambda$.
	\end{inparaenum}		
\end{enumerate}
In each of these cases, we can easily verify the required bounds on the degree and height of $\xi$
\begin{enumerate}
\item 
	\begin{enumerate}
	\item $d_\xi \leq d_{\kappa } - d_{\sigma} = d_{\mu} - d_{\nu} - (d_{\lambda} - d_{\tau}) < d_{\mu} - d_{\nu}$ and $h_{\xi} < h_{\kappa} \leq h_{\mu}$; 
	\item $d_\xi \leq d_{\lambda } - d_{\tau} = d_{\mu} - d_{\nu} - (d_{\kappa} - d_{\sigma}) < d_{\mu} - d_{\nu}$ and $h_{\xi} < h_{\lambda} < h_{\mu}$ 
	 $\xi \in \cC_{\lambda,\tau}$;
	\end{enumerate}
\item $d_{\xi} \leq d_{\kappa} - d_{\nu} < d_{\mu} - d_{\nu}$ and $h_{\xi} < h_{\kappa} \leq h_{\mu}$; 
\item 
	\begin{enumerate}
	\item $d_{\xi} \leq d_{\lambda} - d_{\tau} \leq d_{\mu} - d_{\kappa} \leq d_{\mu} - d_{\nu}$ and $h_{\xi} < h_{\lambda} < h_{\mu}$
	\item $d_{\xi} \leq d_{\lambda} \leq d_{\mu} - d_{\nu}$ and $h_{\xi} < h_{\lambda} < h_{\mu}$.
	\end{enumerate}
\end{enumerate}

\item This point follow by direct inspection of \eqref{eq:def-of-A} and \eqref{eq:def-of-c}.

\item Ditto.

\item We will prove marginally more, namely that 
\begin{equation}
\label{eq:76:80}
A_k(S_l(x))A_l(x) = A_{kl}(x) \text{ for all } x \in \RR^{\cD}.
\end{equation}
Once this is proved, it immediately follows that 
\[
S_k \circ S_l(x) =  A_k(S_l(x))A_l(x)x = A_{kl}(x)x = S_{kl}(x) \text{ for all } x \in \RR^{\cD}.
\]
Since $A_k(x)$ depends only on $k$ and $\fp{x}$, it also follows that
\begin{align*}
 T_k \circ T_l(x)
&= \fp{ A_k(T_l(x))A_l(x)x }
\\& = \fp{ A_k(S_l(x))A_l(x)x }
= \fp{ S_{kl}(x) } = T_{kl}(x) 
\text{ for all } x \in [0,1)^{\cD}.
\end{align*}
Hence, it will suffice to prove \eqref{eq:76:80}. We prove this equality inductively, row by row, for fixed $x \in \RR^\cD$.
For rows with indices in $\cD \cap \NN$ the equality in \eqref{eq:76:80} is clear. Take any $\mu = \kappa \fpp{\lambda} \in \cD \setminus \NN$ and assume that the claim has already been proved for rows indexed by $\kappa$ and by $\lambda$. We will show that 
\begin{equation}\label{eq:def-of-W-good-y}
\bra{ A_k(S_l(x))A_l(x)y }_\mu = \bra{A_{kl}(x)y}_\mu
\end{equation} 
for all $y \in \RR^{\cD}$. Let $W \subset \RR^{\cD}$ be the set consisting of those $y \in \RR^{\cD}$ for which \eqref{eq:def-of-W-good-y} holds. Since both sides of \eqref{eq:def-of-W-good-y} are linear in $y$, $W$ is a vector space.

For $y,z \in \RR^\cD$ let us define, generalising \eqref{eq:def-of-c} slightly,
\begin{equation}
\label{eq:def-of-c-2} 
c_{k,\lambda}(z,y) = \fp{A_k(z) y}_\lambda - \sum_{\tau \preceq \lambda} A_{k}(z)_{\lambda,\tau} \fp{y}_{\tau} = -\ip{\sum_{\tau \preceq \lambda} A_{k}(z)_{\lambda,\tau} \fp{y}_{\tau}}.
\end{equation}
In particular, $c_{k,\lambda}(x,x) = c_{k,\lambda}(x)$. Let also put
\[ 
	C_{k,\lambda}(z) := \set{ y \in \RR^{\cD}}{ c_{k,\lambda}(z,y) = c_{k,\lambda}(z)}.
\]
If $y \in C_{k,\lambda}(z)$ then it follows by the same token as \eqref{eq:76:15} that
\begin{equation}
\label{eq:84:10} 
\bra{A_{k}(z)y}_{\mu} = 
\bra{A_{k}(z)y}_\kappa \fp{A_{k}(z)y}_{\lambda} + 
\sum_{\sigma \preceq \kappa}  \sum_{\tau \preceq \lambda} A_{k}(z)_{\kappa,\sigma} A_{k}(z)_{\lambda,\tau} \bra{y_{\sigma\fpp{\tau}}- y_{\sigma} \fp{y}_{\tau}}
\end{equation}
Hence, if $A_l(x)y \in C_{k,\lambda}(S_l(x))$, we can expand the left hand side of \eqref{eq:def-of-W-good-y} as
\begin{align*}
\bra{ A_k(S_l(x))A_l(x)y}_{\mu} &= \Sigma_1 + \Sigma_2,
\end{align*}
where $\Sigma_1$ and $\Sigma_2$ are given by
\begin{align*}
\Sigma_1 &:= \bra{ A_k(S_l(x))A_l(x)y}_{\kappa} \fp{ A_k(S_l(x))A_l(x)y}_{\lambda}
\\ \Sigma_2 &:=
\sum_{\sigma \preceq \kappa}  \sum_{\tau \preceq \lambda}
A_{k}(S_l(x))_{\kappa,\sigma} A_{k}(S_l(x))_{\lambda,\tau}
\bra{\bra{A_l(x)y}_{\sigma\fpp{\tau}}- \bra{A_l(x)y}_{\sigma} \fp{A_l(x)y}_{\tau} }.
\end{align*}
The first summand is, by the inductive assumption, given by
\[
\Sigma_1 =
\bra{A_{kl}(x)y}_\kappa \fp{A_{kl}(x)y}_{\lambda}.
\]
Using \eqref{eq:84:10} again to expand the second summand, under the additional assumption that $y \in C_{l,\lambda}(x)$ we obtain
\begin{equation*}
\Sigma_2 = \sum_{\sigma \preceq \kappa} \sum_{\tau \preceq \lambda}
A_{k}(S_l(x))_{\kappa,\sigma} A_{k}(S_l(x))_{\lambda,\tau} 
\sum_{\rho \preceq \sigma} \sum_{\theta \preceq \tau} 
A_{l}(x)_{\sigma,\rho} A_{l}(x)_{\tau,\theta} 
\bra{y_{\rho\fpp{\theta}}- y_{\rho} \fp{y_{\theta}}}.
\end{equation*}
Collapsing the sum over $\sigma$ and $\tau$ and using the inductive assumption transforms the above expression into
\begin{align*}
\Sigma_2 &= \sum_{\rho \preceq \kappa} \sum_{\theta \preceq \lambda}
\bra{ A_{k}(S_l(x)) A_l(x) }_{\rho,\kappa}
\bra{ A_{k}(S_l(x)) A_l(x) }_{\theta,\lambda}
\bra{y_{\rho\fpp{\theta}}- y_{\rho} \fp{y_{\theta}}}
\\ &= \sum_{\rho \preceq \kappa} \sum_{\theta \preceq \lambda}
\bra{ A_{kl}(x) }_{\rho,\kappa}
\bra{ A_{kl}(x) }_{\theta,\lambda}
\bra{y_{\rho\fpp{\theta}}- y_{\rho} \fp{y_{\theta}}}.
\end{align*}
Using \eqref{eq:84:10} once more, assuming that $y \in C_{kl,\lambda}(x)$, we conclude that 
\[\Sigma_1 + \Sigma_2 = \bra{ A_{kl}(x)y }_{\mu}.\]
Hence, $W$ contains any vector $y$ such that 
\begin{equation}\label{eq:84:50}
{A_l(x)y} \in C_{k,\lambda}({S_l(x)}),\ y \in C_{l,\lambda}(x) \text{ and }y \in C_{kl,\lambda}(x).
\end{equation}
It remains to show that the set of $y$'s satisfying the last three conditions spans $\RR^{\cD}$.

Note first that, almost trivially, $x$ satisfies all of the conditions in \eqref{eq:84:50}. Hence, we will be interested in $y$ close to $x$. Let $\delta > 0$ be a small parameter, to be determined in the course of the argument. We will consider $y$ of the form $y_\nu = x_\nu + \e_{\nu} z_\nu$ where $z \in [1/2,1]^\cD$ and $\e_\nu > 0$ are arranged so that $\e_{\rho} < \delta \e_{\sigma}$ for any $\rho,\sigma \in \cD$ with $\rho < \sigma$ and $\e_\sigma < \delta$ for any $\sigma \in \cD$. Because $y_{\nu} \in (x_\nu,x_\nu+\delta)$ for each $\nu \in \cD$, choosing sufficiently small $\delta$ we can guarantee that $y \in [0,1)^{\cD}$. Furthermore, recalling \eqref{eq:def-of-c-2} we note that
\[
	\sum_{\tau \preceq \lambda} A_l(x)_{\lambda,\tau}y_{\tau} =					\sum_{\tau \preceq \lambda} A_l(x)_{\lambda,\tau}x_{\tau} + \e_{\lambda}\bra{ k^{d_\lambda}z_{\lambda} + \sum_{\tau \prec \lambda} \frac{\e_\tau}{\e_\lambda} A_l(x)_{\lambda,\tau} z_\tau}
\]
approaches $\sum_{\tau \preceq \lambda} A_l(x)_{\lambda,\tau}x_{\tau}$ from above as $\delta \to 0$. Hence, $y \in C_{l,\lambda}(x)$ assuming $\delta$ is sufficiently small, and by the same token also $y \in C_{kl,\lambda}(x)$. Finally, assuming again that $\delta$ is small enough, we have
\[
	\bra{ A_l(x)y }_{\nu}= S_l(x)_{\nu} + \e_\nu (k^{d_\nu}z_\nu + O(\delta))
\] 
for any $\nu \in \cD$. Hence, using an argument fully analogous to the one above one can show that $A_l(x)y \in  C_{k,\lambda}({S_l(x)})$ for $\delta$ small enough. Consequently, if $\delta > 0$ is small enough then any $y$ of the form described above belongs to $W$. Since $z_\nu$ and $\e_\nu$ were allowed to vary freely in an open region, $W$ has a non-empty interior and hence $W = \RR^{\cD}$, as needed. \qedhere
\end{enumerate}\end{proof}

\subsection{\checkmark\ Lie algebra lemmas}\mbox{}
{
We record several basic lemmas concerning the objects defined in Sections \ref{ssec:setup-basic} and \ref{ssec:setup-geo}. Our first result will allow us to change the basis in a way that simplifies the reasoning (cf.\ Lemma \ref{lem:semialg-wlog-SL}). 
}

\begin{lemma}\label{lem:diag-in-ST}
	Let $A \in \ST(\cD) \setminus \ST'(\cD)$. Then $A$ is diagonalisable by a transition matrix in $\ST'(\cD)$.
\end{lemma}
\begin{proof}
	The claim amounts to the statement that for each $\mu \in \cD$, $A$ has an eigenvector in $e_\mu + \linspan\set{e_\nu}{\nu \prec \mu}$ corresponding to the eigenvalue $A_{\mu,\mu} \neq 0$. This is readily proved by downwards induction on $\mu$.   
\end{proof}

{
Recall that the Lie algebra $\stp(\cD)$ consists of strictly lower trinagular matrices (with respect to $\preceq$). Extending $\preceq$ to a total order in an arbitrary way, we can construe $\stp(\cD)$ as a Lie subalgebra of the algebra of (ordinary) strictly lower triangular matrices. As a consequence, the exponential map on $\stp(\cD)$ is particularly well-behaved.
}

\begin{example}\label{ex:exp-nil}
In the running example, $\bstp(\cD)$ consist of matrices $Z$ of the form
\begin{equation}\label{eq:421:0}
Z =  
\begin{bmatrix}
0 & 0 & 0 & 0 & 0 & 0 \\
Z_{1,0} & 0 & 0 & 0 & 0 & 0 \\
Z_{2,0} & 0 & 0 & 0& 0 & 0 \\
Z_{1\fpp{2},0} & Z_{1\fpp{2},1} & 0 & 0 & 0 & 0 \\
Z_{2\fpp{1},0} & 0 & Z_{2\fpp{1},2} & 0 & 0 & 0 \\
Z_{3,0} & 0 &0 & 0 & 0 & 0 
\end{bmatrix}.
\end{equation}
The exponential of a matrix $Z$ as above takes the form
\begin{equation}\label{eq:421:1}
\exp(Z) =  
\begin{bmatrix}
1 & 0 & 0 & 0 & 0 & 0 \\
Z_{1,0} & 1 & 0 & 0 & 0 & 0 \\
Z_{2,0} & 0 & 1 & 0& 0 & 0 \\
Y_{1\fpp{2},0} & Z_{1\fpp{2},1} & 0 & 1 & 0 & 0 \\
Y_{2\fpp{1},0} & 0 & Z_{2\fpp{1},2} & 0 & 1 & 0 \\
Z_{3,0} & 0 &0 & 0 & 0 & 1 
\end{bmatrix},
\end{equation}
where the remaining coefficients are given by
\begin{align*}
Y_{1\fpp{2},0} &= Z_{1\fpp{2},0} + \frac{1}{2} Z_{1\fpp{2},1}Z_{1,0}, &&&
Y_{2\fpp{1},0} &= Z_{2\fpp{1},0} + \frac{1}{2} Z_{2\fpp{1},2}Z_{2,0}.
\end{align*}
\end{example}

\begin{lemma}\label{lem:exp-is-poly-nil}
	The exponential map $\exp \colon \stp(\cD) \to \ST'(\cD)$ is a diffeomorphism given by rational polynomial formulae. Moreover, the same applies to the logarithmic map $\log \colon \ST'(\cD) \to \stp(\cD)$.
\end{lemma}
\begin{proof}
\cite[Thm.\ 1.2.1 + Prop.\ 1.2.7]{CorwinGreenleaf-book}
\end{proof}

{	
The exponential map on $\st(\cD)$ is marginally more complicated, but nevertheless quite tractable. 
}

\begin{example}\label{ex:exp}
Continuing the running example (cf.\ Example \ref{ex:exp-nil}), $\bar \st(\cD)\setminus \bstp(\cD)$ consists of matrices of the form $t(Z+\bar \Lambda)$, where $t \in \RR \setminus \{0\}$ and $Z$ takes the form \eqref{eq:421:0}. The exponential map
\[ \exp\colon \bar\st(\cD) \setminus \bar\st{}'(\cD) \to \bar\ST(\cD) \setminus \bar\ST{}'(\cD) \]
is described by the formula
\begin{equation}
	\exp \big( t (Z + \bar \Lambda) \big) = 
\begin{bmatrix}
1 & 0 & 0 & 0 & 0 & 0 \\
Y_{1,0} & e^t & 0 & 0 & 0 & 0  \\
Y_{2,0} & 0 & e^t & 0& 0 & 0 \\
Y_{1\fpp{2},0} & Y_{1\fpp{2},1} & 0 & e^{2t} & 0 & 0 \\
Y_{2\fpp{1},0} & 0 & Y_{1\fpp{2},2} & 0 & e^{2t} & 0 \\
Y_{3,0} & 0 & 0 & 0 & 0  & e^{2t}
\end{bmatrix},	
\end{equation}
where the coefficients $Y_{\mu,\nu}$ are given by
\begin{align*}
Y_{1,0} &= (e^t-1) Z_{1,0},&&&
Y_{2,0} &= (e^t-1) Z_{2,0},\\
Y_{1\fpp{2},1} &= e^t(e^t-1) Z_{1\fpp{2},1},&&&
Y_{2\fpp{1},2} &= e^t(e^t-1) Z_{2\fpp{1},2},\\
Y_{1\fpp{2},0} &= 
\frac{e^{2t}-1}{2} Z_{1\fpp{2},0} + \frac{(e^t-1)^2}{2} Z_{1\fpp{2},1} Z_{1,0} 
,&&&
Y_{2\fpp{1},0} &= 
\frac{e^{2t}-1}{2} Z_{2\fpp{1},0} + \frac{(e^t-1)^2}{2} Z_{2\fpp{1},2} Z_{2,0},\\
Y_{3,0} &= (e^{2t}-1)Z_{3,0}.
\end{align*}
\end{example}

\begin{lemma}\label{lem:exp-is-poly}
	The exponential map $\st(\cD) \to \ST(\cD)$ is a diffeomorphism. The map 
	\begin{align}
	\label{eq:69:50} \RR \times \stp(\cD) \ni (t,Z) \mapsto \exp( t(\Lambda +Z)) \in \ST(\cD)
	\end{align}
	 is given by a polynomial with rational coefficients in $e^t$ and $Z$. 
\end{lemma}
\begin{proof}\color{White}

Since $\exp$ is a local diffeomorphism, to verify that it is a diffeomorphism it is enough to check that it is bijective. We can express any $Z \in \st(\cD)$ in the form
	\begin{equation}\label{eq:60:00}\textstyle
		Z = \sum_{\kappa \succ \lambda} Z_{\kappa,\lambda} E^{\kappa,\lambda}.
	\end{equation}
	Expanding the exponential map into a power series we find that
	\begin{equation}\label{eq:60:01}
			\exp(t(\Lambda+Z)) = \sum_{n = 0}^\infty \frac{t^n}{n!} \bra{ \Lambda + Z}^n.
	\end{equation}
	Expanding $(\Lambda +Z)^n$ yields the sum of $\Lambda^n$ and terms of the form
	\begin{equation}\label{eq:60:70}
\Lambda^{n_1} Z \Lambda^{n_2} \dots \Lambda^{n_{r-1}} Z \Lambda^{n_r} 
	\end{equation}
	where $ r \geq 2$, $n_1,n_2,\dots,n_r \geq 0$ and $n_1+\dots+n_r = n-r+1$.
	
	Recall that $\Lambda E^{\kappa,\lambda} = d_\kappa E^{\kappa,\lambda}$ and $E^{\kappa,\lambda}\Lambda = d_\lambda E^{\kappa,\lambda}$, and that $E^{\kappa,\lambda} E^{\mu,\nu}$ is equal to either $E^{\kappa,\nu}$ or $0$, depending on whether $\lambda = \mu$ or not ($\kappa,\lambda,\nu,\mu \in \cD$). Hence, inserting \eqref{eq:60:00} into \eqref{eq:60:70} yields the sum of terms of the form	
	\begin{equation}\label{eq:60:71}
	\textstyle \bra{\prod_{i=1}^{r-1} Z_{\kappa_i,\kappa_{i+1}}}
	\bra{\prod_{i=1}^r d_{\kappa_i}^{n_i}} E^{\kappa_1,\kappa_r}
	\end{equation}
	where $\kappa_1 \succ \kappa_2 \succ \dots \succ \kappa_r$ and $r, n_1,\dots,n_r$ are as above. It is elementary (even if somewhat mundane) to derive for pairwise distinct $x_i \in \RR$ a formula of the form
	\begin{equation}\label{eq:60:02}
	\sum_{n_1,\dots,n_r} \braif{n_1+\dots+n_r = n} \prod_{i=1}^r x_i^{n_i} = \sum_{i=1}^r a_i(x_1,\dots,x_r) x_i^n,
	\end{equation}
	where $a_i$ are coefficients dependent only on $r$ and $x_1,\dots,x_r$. (In fact, these coefficients are given by $a_i = x_i^{r-1}/\prod_{j \neq i}(x_i - x_j)$, but the values do not play a role in the reasoning). Thus, letting $\vec \kappa := (\kappa_1,\dots,\kappa_r)$, and setting
\begin{align*}
	a_i'(\vec \kappa) &:= a_i(d_{\kappa_1},\dots,d_{\kappa_r}), &&&
	Z(\vec{\kappa}) &:= {\textstyle\prod_{i=1}^{r-1} Z_{\kappa_i,\kappa_{i+1}}}
\end{align*}	
we obtain the expansion
	\begin{equation}\label{eq:60:03}
	\bra{ \Lambda + Z}^n 
	= \Lambda^n + \sum_{\kappa \succ \lambda} E^{\kappa,\lambda} 
		\sum_{\substack{ \kappa_1 \succ \dots \succ \kappa_r\\ \kappa_1 = \kappa,\, \kappa_r = \lambda} } Z(\vec{\kappa}) \sum_{i=1}^r
	a_i'(\vec\kappa) d_{\kappa_i}^n .
	\end{equation}
	Plugging \eqref{eq:60:03} into \eqref{eq:60:01} and changing the order of summation, we conclude that
	\begin{equation}\label{eq:60:04}
	\exp(t(\Lambda+Z)) = \exp(t\Lambda)+ \sum_{\kappa \succ \lambda} E^{\kappa,\lambda}
	\sum_{\substack{ \kappa_1 \succ \dots \succ \kappa_r\\ \kappa_1 = \kappa,\, \kappa_r = \lambda} } Z(\vec{\kappa})
	 \sum_{i=1}^r a_i'(\vec{\kappa}) \exp\bra{d_{\kappa_i}t}.
	\end{equation}
	Since the length $r$ of any decreasing sequence $\kappa_1 \succ \kappa_2 \succ \dots \succ \kappa_r$ is bounded by $\abs{\cD}$, this implies that the polynomial map in \eqref{eq:69:50} is a polynomial in $e^t$ and $Z$.
	
	Next, we construct the inverse of the exponential map. Suppose that $A \in \ST(\cD) \setminus \STp(\cD)$. Let  $t = t(A) := \log(A_{\mu, \mu})/d_\mu$ for some $\mu \in \cD$; the definition of $\ST(\cD)$ ensures that this is well-defined and independent of the choice of $\mu$. We next construct $Z = Z(A) \in \st'(\cD)$ such that $\exp(t(\Lambda + Z)) = A$; so far, we have ensured that the diagonal entries agree.
	
	We assign values to the matrix entries $Z_{\kappa,\lambda}$ ($\kappa, \lambda \in \cD$, $\kappa \succ \lambda$) by induction on the length of the longest path 
	\[
		r(\kappa, \lambda) := \max\set{r \in \NN}{\text{ there exists } \vec\kappa \text{ such that }\kappa = \kappa_1 \succ \dots \succ \kappa_r = \lambda}.
	\]
	By \eqref{eq:60:04}, the requirement that $\exp(t(\Lambda + Z))_{\kappa,\lambda} = A_{\kappa,\lambda}$ is equivalent to
	\begin{align*}
			A_{\kappa,\lambda} 
	&= \sum_{\substack{ \kappa_1 \succ \dots \succ \kappa_r\\ \kappa_1 = \kappa,\, \kappa_r = \lambda} } Z(\vec{\kappa})
	 \sum_{i=1}^r a_i'(\vec{\kappa}) \exp\bra{d_{\kappa_i}t}
	 \\& = Z_{\kappa,\lambda}\bra{a_1'(\kappa,\lambda)\exp(d_{\kappa}t) + a_2'(\kappa,\lambda)\exp(d_{\lambda}t) } + R(A),
	\end{align*}	
	where the remainder term $R(A)$ depends only on the entries of $Z$ that have already been computed in the previous steps. It follows that $Z_{\kappa,\lambda}$ with the required properties exists and is determined uniquely.
\end{proof}

{
As a consequence of Lemma \ref{lem:exp-is-poly}, it makes sense to speak of non-integer powers of matrices in $\ST(\cD)$, given by $A^t = \exp(t \log(A))$ for $A \in \ST(\cD)$ and $t \in \RR$. Moreover, we have the following analogue of Lemma \ref{lem:stab-is-alg}.
}

\begin{lemma}\label{lem:stab-is-linear}
	Let $V \subset \RR^{\cD}$ be an algebraic variety, let $A \in \ST(\cD)$, and suppose that $A(V) \subset V$. Then $A^t(V) = V$ for all $t \in \RR$.
\end{lemma}
\begin{proof}
	Recall that by Lemma \ref{lem:stab-is-alg}\eqref{it:stab-is-alg:A} for any $t \in \RR$ the condition $A^t(V) \subset V$ is equivalent to the ostensibly stronger condition $A^t(V) = V$. Also note that $A^n(V) = V$ for all $n \in \NN$.

	Suppose first that $A \in \ST'(\cD)$. Then $A^t$ is a polynomial in $t$. Hence, by Lemma \ref{lem:stab-is-alg}\eqref{it:stab-is-alg:B} the set $S := \set{t \in \RR}{A^t(V) \subset V}$ is algebraic, and it is also infinite since $n \in S$ for all $n \in \NN$. It follows that $S = \RR$, meaning that $A^t(V) = V$ for all $t \in \RR$.

	Secondly, suppose that $A \in \ST(\cD) \setminus \ST'(\cD)$. By Lemma \ref{lem:diag-in-ST}, we may assume without loss of generality that $A = \Delta_a$ for some $a \in (0,1) \cup (1,\infty)$. By Lemma \ref{lem:stab-is-alg}, the set $R:= \set{t \in \RR}{\Delta_t(V) \subset V}$ is algebraic, and it is also infinite since $a^n \in A$ for all $n \in \NN$. It follows that $R = \RR$, and $A^t(V) = V$ for all $t \in \RR$.
\end{proof} 
 \section{Recurrence theorem}\label{sec:recurrent}

In this section we will prove Theorem \ref{thm:recurrence}, asserting a rather strong recurrence property of the $\times k$ maps $T_k$ introduced in the previous section (cf.{} Prop.{} \ref{prop:Tk-basic}). Later, in Section \ref{sec:main}, we will see that our main result, Theorem \ref{thm:A}, follows relatively easily from the aforementioned recurrence theorem.

\subsection{\checkmark\ Setup}

{
In order to state our next result, we will need to introduce some further objects and conventions. The content of this section is closely analogous to the material in Section \ref{ssec:torus-setup}, concerning the special case of the torus.
}

\textit{
Throughout the remainder of the paper, the basis $k \geq 2$ is fixed, and $x^0$ denotes a point in $[0,1)^{\cD}$. All objects we construct are allowed to depend on $k$ and $x^0$, unless explicitly stated otherwise.}

{
For $p \in \beta \NN$ we define (with $T_k$ given by \eqref{eq:def-of-T})
\begin{equation}\label{eq:def-of-x^p}
	x^p := \lim_{n \to p} T^n_k(x^0).
\end{equation}
In particular, $x^n = T^n_k(x^0)$ for $n \in \NN$. Recall that for any $l \in \NN$ (under the identification discussed in Section \ref{ssec:setup-geo} and using the notation introduced in Section \ref{ssec:setup-Tk}) we have
\begin{equation}
T_l(x) = \bar{A}_l (x)x = A_l(x)x-b_l(x) \text{ for } x \in [0,1)^\cD.
\end{equation}
We further define matrices $\bar A_l|_p$, $A_l|_p$ and vectors $b_l|_p$ by
\begin{align}\label{eq:def-of-Alp}
\bar A_{l}|_p = \begin{bmatrix}
1 & 0 \\
-b_{l}|_p & A_{l}|_p   
\end{bmatrix} := \lim_{n \to p} \bA_l(x^n) \in \bar\ST(\cD,\ZZ).
\end{align}
Since $\bar{A}_l(x)$ have bounded integer entries as $x$ ranges over $[0,1)^\cD$ we have an ostensibly stronger property that 
\begin{align}
\forall^p_n\  \bar A_l(x^n) &= \bar A_{l}|_p,  &\text{ i.e., }&& \forall^p_n\ A_l(x^n) &= A_{l}|_p \ \text{ and } \ \forall^p_n\ b_{l}(x^n) = b_{l}|_p.
\end{align}
Accordingly, slightly abusing the notation, we define the affine maps $T_l|_p \colon \RR^{\cD} \to \RR^{\cD}$ by
\begin{equation}\label{eq:def-of-Tlp}
T_l|_p(x) := \bar{A}_l|_p x = A_{l}|_p x - b_l|_p.
\end{equation}
For $p,q \in \beta \NN_0$, we define the irreducible varieties 
\begin{equation}\label{eq:def-of-V-2}
	V_p^q := \alglim_{n \to p}(x^{n+q})= \bigcap_{I \in p} \algcl\set{ T_k^n|_q(x^q) }{ n \in I}.
\end{equation}
This construction is slightly more general than the construction of the varieties $V_p$ is \eqref{eq:def-of-Vp-ab}, which can be recovered by taking $q = 0$. We record some basic properties of thus defined objects (cf.\ Lemma \ref{lem:translation}). 
}

\begin{lemma}\label{lem:translation-ur}
Let $p,q,r \in \beta \NN_0$.
\begin{enumerate}[wide]
\item\label{it:trans-ur:A} $\forall^r_n \ T_k^n|_{p+q}(V_p^q) \subset V_{r+p}^q$ with equality if $p$ is minimal.
\item\label{it:trans-ur:B} $V_{r}^{p+q} \subset V_{r+p}^q$ with equality if $q$ and $r$ are minimal.
\end{enumerate}
\end{lemma}
\begin{proof}
Item \eqref{it:trans-ur:A} follows by the same argument as Lemma \ref{lem:translation} (with $x^q$ in place of $x^0$). The inclusion in  \eqref{it:trans-ur:B} follows directly from Lemma \ref{lem:alg-lim-basic}\eqref{it:alg-lim-basic:E}. We also record the following consequence of Lemma \ref{lem:stab-is-alg} and \eqref{it:trans-ur:A}: If $p$ and $r$ belong to the same minimal left ideal and $V^{q}_{r} \subset V^q_{p}$ then $V^{q}_{r} = V^q_{p}$.

Assume now that $q$ and $r$ are minimal. Then there exists $u \in \beta \NN + r$ such that $u+p+q = q$. Hence, we have the chain of inclusions
\begin{equation}\label{eq:172:60}
	V^{p+q}_{r} \subset V^{q}_{r+p} = V^{u+p+q}_{r+p} \subset V^{p+q}_{r+p+u}.
\end{equation}
Since $r$ and $r+p+u$ belong to the same minimal left ideal $\beta \NN_0 + r$, the extreme terms in \eqref{eq:172:60} are equal. In particular, the second part of \eqref{it:trans-ur:B} holds.
\end{proof}

{Recall that if $p,q \in \beta \NN_0$ then $p \sim q$ means that $p$ and $q$ generate the same left ideals.
For minimal $p \in \beta \NN$ we additionally define 
\begin{equation}\label{eq:def-of-V_p^q}
	\Vp_p := V^{q}_p \text{ where } q \sim p \text{ and } q \text{ is an idempotent}.
\end{equation}
It follows from Lemma \ref{lem:translation-ur}\eqref{it:trans-ur:B} that this definition is well posed, i.e., $\Vp_p$ does not depend on the choice of the idempotent $q$ subject to the above constraints. Indeed, if $q$ and $q'$ are two idempotents generating the same minimal left ideal as $p$ then
\(
	V^{q}_{p} = V^{q}_{p+q'} = V^{q'+q}_{p} = V^{q'}_p.
\)

{\color{NavyBlue} The proof of the following result will occupy the remainer of this section. Recall that $\KN \subset \beta \NN$ and $\Xi \subset \beta \NN$ were defined in Section \ref{ssec:prelims-ultrafilters}.
}

\begin{theorem}\label{thm:recurrence}\color{Mahogany}
	Let $p \sim q \in \KN$ and assume that $x^p,x^q \in (0,1)^{\cD}$. Then there exists $a \in \Xi$ such that
\begin{enumerate}[label={\normalfont({$\mathrm{R}_\arabic*$})}]
\item\label{it:1:A} $\displaystyle \lim_{l \to a} T_l|_p(x^p) = x^q$,
\item\label{it:1:B} $\displaystyle \forall^a_l \ T_l|_p(\Vp_p) = \Vp_q$.
\end{enumerate}
\end{theorem}

\begin{remark}
	The assumption that $x^p$ and $x^q$ should avoid the boundary of the cube $[0,1]^d$ is added for technical reasons. The author believes that it should be possible to remove it. In many (but unfortunately not all) cases this can be achieved by a change of basis (cf.\ Lemma \ref{lem:semialg-wlog-SL}). In applications, we are primarily interested in the case where $p = q$ is idempotent. We consider the more general situation mostly because it is necessary for the inductive argument. 
\end{remark}

\renewcommand{\t}{t}
\subsection{Proof strategy}\label{ssec:recurrent-strategy}

Let us now discuss the outline of the proof of Theorem \ref{thm:recurrence} and introduce the remaining notation which will be needed in the course of the argument.

The proof of Theorem \ref{thm:recurrence} proceeds by induction with respect to $\cD$. 
In each step, we remove from $\cD$ the indices highest degree, not counting indices in $\NN$; we informally refer to these as the ``top'' indices. More precisely, we put
\begin{align}\label{eq:def-of-max-d}
	\s &:= \max\set{d_\mu}{\mu \in \cD \setminus \NN},
\end{align}
and accordingly we introduce the subsets of $\cD$ given by
\begin{align}
\label{eq:def-of-D-top}
	\cD_{\mathrm{top}} &:= \set{\mu \in \cD}{d_\mu = \s}, &&&
	\cD_{\mathrm{ab}} &:= \set{\mu \in \cD}{d_\mu > \s} \subset \NN,\\
\label{eq:def-of-D-low}	\cD_{\mathrm{low}} &:= \set{\mu \in \cD}{d_\mu < \s}, &&&
	\cE &:= \cD_{\mathrm{ab}} \cup \cD_{\mathrm{low}} =  \cD \setminus \cD_{\mathrm{top}}.
\end{align}
Using the notion of complexity introduced in Section \ref{ssec:setup-basic}, we observe that $\complexity(\cE) < \complexity(\cD)$, as long as $\cD \not \subset \NN$. We also put
\begin{equation}\label{eq:def-of-Vtop}
	\cV_{\mathrm{top}} := \set{ x \in \RR^{\cD} }{ x_\mu = 0 \text{ for all } \mu \not \in \cD_{\mathrm{top}} }.
\end{equation}

Modulo some technical issues, we plan to derive the Theorem \ref{thm:recurrence} for $\cD$ from the same theorem for $\cE \cup \cN$, where $\cN \subset \NN$. For this reason, we treat the case $\cD \subset \NN$ independently (see Section \ref{ssec:rec-combined}). We further subdivide the inductive step described above into two stages: Firstly, we show that Theorem \ref{thm:recurrence} for $\cE \cup \cN$ implies a weaker variant of the same theorem for $\cD$, where conditions \ref{it:1:A} and \ref{it:1:B} are replaced with weaker conditions \ref{it:3:A} and \ref{it:3:B} which will be introduced shortly.
Secondly we show that this weaker version can be bootstrapped to obtain the original version. Additionally, we keep track of relations between the maps $T_{lm}|_p$ and $T_l|_p \circ T_m|_p$ recorded by condition \ref{it:1:C}, \ref{it:2:C}, \ref{it:3:C}. 

Let us now state the conditions alluded to above.
For a vector $v \in \cV_{\mathrm{top}}$, we consider the analogues of \ref{it:1:A} and \ref{it:1:B} that additionally include a shift by $v$:
\begin{enumerate}[label={\normalfont({$\mathrm{R}_\arabic*'$})}]
\item\label{it:2:A} $\displaystyle \lim_{l \to a} T_l|_p(x^p) = x^q + v$,
\item\label{it:2:B} $\displaystyle \forall^a_l \ T_l|_p(\Vp_p) = \Vp_q + v$.
\end{enumerate}
Accordingly, for a bounded sequence $v_l \in \cV_{\mathrm{top}}$ ($l \in \NN$), we consider the analogues of \ref{it:1:A} and \ref{it:1:B} where the shifts are allowed to depend on $l$:
\begin{enumerate}[label={\normalfont({$\mathrm{R}_\arabic*''$})}]
\item\label{it:3:A} $\displaystyle \lim_{l \to a} T_l|_p(x^p) = x^q + \lim_{l \to a} v_l$,
\item\label{it:3:B} $\displaystyle \forall^a_l \ T_l|_p(\Vp_p) = \Vp_q + v_l$.
\end{enumerate}

For technical reasons, we are also need to keep track of how the maps $T_{lm}|_p$ relate to the maps $T_l|_p$ and $T_m|_p$ ($m,l \in \NN$). As a motivating example, we point out that conditions \ref{it:1:A} and \ref{it:1:B} almost --- but not quite --- imply a new condition
\begin{enumerate}[label={\normalfont({$\mathrm{R}_\arabic*$})}]\setcounter{enumi}{2}
\item \label{it:1:C} $\displaystyle \forall_m \ \forall^a_l \ T_{ml}|_p = T_{m}|_q \circ T_l|_p$.
\end{enumerate}
Indeed, for any $m \in \NN$ and $a$-almost all $l$ it follows from the definitions that
\[
	\bar{A}_{ml}|_p = \lim_{n \to p} \bar{A}_m\bra{T_l|_p(x^n)} \cdot \bar{A}_l|_p \overset{!}{=} \bar{A}_m|_q \cdot \bar{A}_l|_p,
\]
where the equality labelled with the exclamation mark holds as long as $\bar{A}_m$ is continuous at $x^q$. For a sequence $w_m \in \cV_{\mathrm{top}}$ ($m \in \NN$), we introduce the analogue of \ref{it:1:C} that includes a shift by $w_m$:
\begin{enumerate}[label={\normalfont({$\mathrm{R}_\arabic*'$})}]\setcounter{enumi}{2}
\item \label{it:2:C} $\displaystyle \forall_m \ \forall^a_l \ T_{ml}|_p = T_{m}|_q \circ T_l|_p + w_m e_0^{\mathrm{T}}$.
\end{enumerate}
(Here and elsewhere, $e_0^{\mathrm{T}}$ denotes the constant map $1$ on $\RR^{\cD}$, which is consistent with the identification of $\Aff(\cD)$ with $\End(\{0\} \cup \cD)$ discussed in Section \ref{ssec:setup-geo}.) Similarly, for a sequence $w_{m,l} \in \cV_{\mathrm{top}}$ ($l,m \in \NN$), we consider the condition
\begin{enumerate}[label={\normalfont({$\mathrm{R}_\arabic*''$})}]\setcounter{enumi}{2}
\item \label{it:3:C} $\displaystyle \forall_m \ \forall^a_l \ T_{ml}|_p = T_{m}|_q \circ T_l|_p + w_{m,l} e_0^{\mathrm{T}}$,
\end{enumerate}

We introduce names for sets of ultrafilters which satisfy combinations of the above conditions that will appear in the argument.
\begin{align}
\label{eq:def-of-cH}
	\cH^q_p &:= \set{ a \in \beta \NN \setminus \NN }{\text{\ref{it:1:A}, \ref{it:1:B} and \ref{it:1:C} hold}}, \\
		\label{eq:def-of-tcH} \tcH^q_p &:= \set{ a \in \beta \NN \setminus \NN}{ \text{\ref{it:1:A}, \ref{it:1:B} and \ref{it:2:C} hold for some 
	} w_m \in \cV_{\mathrm{top}}},\\
	\label{eq:def-of-cG} \cG^q_p &:= \set{ a \in \beta \NN \setminus \NN}{ \text{\ref{it:2:A}, \ref{it:2:B} and \ref{it:2:C} hold for some 
	} v, w_m \in \cV_{\mathrm{top}}},\\
	\label{eq:def-of-tcG} \tcG^q_p &:= \set{ a \in \beta \NN \setminus \NN}{ \text{\ref{it:3:A}, \ref{it:3:B} and \ref{it:3:C} hold for some 
	} v_l, w_{m,l} \in \cV_{\mathrm{top}}}.
\end{align}
It follows directly from the definitions have the chain of inclusions
\[
	\cH^q_p \subseteq \tcH^q_p \subseteq \cG^q_p \subseteq \tcG^q_p.
\]
\begin{remark}\label{rmk:dep-on-D}
\begin{enumerate}[wide]
\item In the rare cases when we need to track dependence on $\cD$, we write $\cH^q_p[\cD]$, $\cG^q_p[\cD]$, etc., but we try to avoid this rather cumbersome piece of notation. (This happens primarily in Section \ref{ssec:recurrent-induction} where we make use of the inductive assumption.) 
\item For any $a \in \cG^q_p$, the vectors $v$ and $w_m$ ($m \in \NN$) are uniquely determined by conditions \ref{it:2:A} and \ref{it:2:C}. For fixed $p \sim q \in \KN$, this gives rise to the maps 
\begin{align}\label{eq:def-of-v-and-w}
	\cG^q_p \ni a &\mapsto v^a \in \cV_{\mathrm{top}},&&& \cG^q_p \ni a &\mapsto w^a_m \in \cV_{\mathrm{top}} \quad (m \in \NN).
\end{align}
(The ultrafilters $p,q$ will always be clear from the context and we omit them from the notation for the sake of not obfuscating the formulae excessively.) 
We may now characterise $\tcH^q_p$ (resp.\ $\cH^q_p$) as the set of those $a \in \cG^q_p$ such that $v^a = 0$ (resp.\ $v^a = w^a_m = 0$ for all $m \in \NN$).
\end{enumerate}
\end{remark}

We can now state the slight strengthening of Theorem \ref{thm:recurrence} which is the theorem which we will actually prove in this section.
\begin{theorem}\label{thm:recurrence-strong}\color{Mahogany}
Let $p \sim q \in \KN$. If $x^p,x^q \in (0,1)^{\cD}$ then $\tcH^q_p \cap \Xi \neq \emptyset$.
\end{theorem}

Before return to the discussion of the outline of the argument, we mention two final pieces of relevant notation. (This is also meant to keep the upcoming subsections self-contained, at least in terms of definitions.) For $q \in \KN$ we put 
\begin{equation}\label{eq:def-of-cU}
	\cU_q := \set{ u \in \cV_{\mathrm{top}}}{\Vp_q+u = \Vp_q}.
\end{equation}
and note that for any $p \sim q$ and $a \in \beta \NN$, condition \ref{it:2:B} determines $v^a$ uniquely up to a shift in $\cU_q$. For $p \sim q \in \KN$, we define
\begin{equation}\label{eq:def-of-Zeta}
	\Zeta_p^q := \set{ a \in \beta \NN_0}{ \lim_{l \to a \cdot k^p} \alpha l^d = \lim_{l \to  k^q } \alpha l^d \text{ for all } d \in \NN,\ \alpha \in \RR/\ZZ}.
\end{equation}
Intuitively, one can think of $\Zeta_p^q$ as the set of those $a \in \beta \NN_0$ that are guaranteed to satisfy condition \ref{it:1:A} in the case where $\cD$ is a subset of $\NN$ (but $\cD$ and $x^0$ are otherwise arbitrary).

With this notation in hand, we can explain our strategy in a more concrete way. (We still allow for a certain level of imprecision; for exact statements of the relevant results, see subsequent sections.) We will show the following:
\begin{enumerate}[label={(\arabic*)},ref={(\arabic*)},wide]
\item\label{it:summary-1} If $\tcH^p_q[\cE] $ intersects $\Delta$ for any $p \sim q$ ($q \in \KN$ and $x^p,x^q \in (0,1)^\cD$), then also $\tcG^q_p = \tcG^q_p[\cD]$ intersects $\Delta$ for any $p \sim q$. This is the only point where we use the inductive hypothesis. (See Section \ref{ssec:recurrent-induction}, Proposition \ref{prop:induction}.) The key idea is to describe the set $\tcG^p_q[\cD]$ in terms of semialgebraic geometry and the projection of the orbit $T_l|_p(x^p)$ ($l \in \NN$) to $\RR^{\cE}$.
\item\label{it:summary-2} The sets $\tcG^q_p$ and $\cG^q_p$ are actually equal. A key point in the argument is the observation that $\cU_q$ is a vector space defined over $\QQ$, which is also used later. (See Section \ref{ssec:rec-rigidity}, Proposition \ref{prop:rigidity}.)
\item\label{it:summary-3} The elements of $\cG^q_p$, $\cG^r_q$  for $p \sim q \sim r$ can be combined together to produce new elements of $\cG^{r}_{p}$: if $a \in \cG^r_q$ and $b \in \cG^q_p$ then $a \cdot b \in \cG^{r}_{p}$. (See Section \ref{ssec:rec-semigrp}, Proposition \ref{lem:cGp-basic}.)
\item\label{it:summary-4} In the discussion above, one can freely replace $\cG^q_p$ with $\Sigma^q_p \cap \Zeta^q_p$ (this corresponds, roughly, to replacing $\cD$ with $\cD \cup \cN$ for $\cN \subset \NN$, see Section \ref{ssec:rec-abel}). This leads to a simpler formula for $v^{a \cdot b}$ in \ref{it:summary-3}, which, after iteration, allows us to obtain elements of $\tcH^q_p \cap \Xi$ from elements of $\cG^q_p \cap \Xi$.
\end{enumerate}

The steps described above are mostly independent, so the order in which they are presented is to a certain degree arbitrary. We opt to start with \ref{it:summary-2}, which allows us to avoid distinguishing between $\tcG^q_p$ and $\cG^q_p$ in the remainder of the discussion. It is then natural to follow with \ref{it:summary-3} and \ref{it:summary-4}, and finally \ref{it:summary-1}.

\subsection{\checkmark\ Rigidity}\label{ssec:rec-rigidity} {\color{NavyBlue}

The main goal of this subsection is to prove the following.
\begin{proposition}\label{prop:rigidity}\color{Mahogany}
	If $p \sim q \in \KN$ then $\cG^q_p = \tilde \cG^q_p$.
\end{proposition}
 In other words, we claim that if $p\sim q \in \KN$ and $a \in \beta \NN \setminus \NN$ satisfies conditions \ref{it:3:A}, \ref{it:3:B} and \ref{it:3:C} for some sequences $v_l \in \cV_{\mathrm{top}}$ and $w_{m,l} \in \cV_{\mathrm{top}}$ ($l,m \in \NN$) then there exists a single vector $v \in \cV_{\mathrm{top}}$ and a sequence $w_m \in \cV_{\mathrm{top}}$ ($m \in \NN$) such that the same conditions  \ref{it:3:A}, \ref{it:3:B} and \ref{it:3:C} are satisfied for $v_l = v$ and $w_{m,l} = w_m$. We first address condition \ref{it:3:C}.
}

\begin{lemma}\label{lem:R3-wk<=>R3-str}\color{BurntOrange}
	Let $p \sim q \in \KN$, $a \in \beta \NN$ and $m \in \NN$. The following conditions are equivalent:
	\begin{enumerate}
	\item\label{it:51:A} There exists $w \in \cV_{\mathrm{top}}$ such that $\forall^a_l \ T_{ml}|_p = T_{m}|_q \circ T_{l}|_p + w e^{\mathrm{T}}_0$.
	\item\label{it:51:B} There exist $w_{l} \in \cV_{\mathrm{top}}$ ($l \in \NN$) such that $\forall^a_l \ T_{ml}|_p = T_{m}|_q \circ T_{l}|_p + w_{l} e^{\mathrm{T}}_0$.
	\end{enumerate}
\end{lemma}
\begin{proof}\color{White}
	Condition \eqref{it:51:A} clearly implies \eqref{it:51:B}, so we only need to verify the reverse implication. Assume that \eqref{it:51:B} holds. For each $l \in \NN$ the vector $w_l$ has integer coefficients because all of the maps $T_{ml}|_p$, $T_m|_q$ and $T_l|_q$ have integer coefficients. Moreover, $w_l$ is bounded uniformly in $l$ since 
\[
	\norm{w_l} = \norm{ T_{ml}|_p(x^p) - T_m|_q(T_l|_p(x^p))} \leq (1+\norm{T_m|_q})\sqrt{\abs{\cD}}.
\]
	Hence, the sequence $w_l$ is finitely-valued and \eqref{it:51:A} holds with $w := \lim_{l \to a} w_l$. 
\end{proof}

{\color{Blue}
The situation is more complicated when it comes to conditions \ref{it:3:A} and \ref{it:3:B}. A particular source of difficulties lies in that fact that these conditions do not imply that $v_l$ is essentially a constant sequence. 
Instead, \ref{it:3:B} determines $v_l$ uniquely up to a shift in $\cU_q$ (cf.\ Lemma \ref{lem:stab-is-alg}). Recall that the space $\cU_q$ was defined in \eqref{eq:def-of-cU}. This leaves us with the task of proving the following.
}

\begin{proposition}\label{prop:v_perp-discrete-general}\color{Mahogany}
	
	For any $q \in \KN$ and $C > 0$ there exists a finite set $F \subset \RR^{\cD}$ such that for any $p \in \beta\NN +q$, any $v \in \cV_{\mathrm{top}}$ with $\norm{v}_{\infty} \leq C$ and any $l \in \NN$, if $T_l|_p(\Vp_p) = \Vp_q + v$ then $v \in F+\cU_q$. 
\end{proposition}

{\color{NavyBlue}
Once Proposition \ref{prop:v_perp-discrete-general} is proved, the main result of this section easily follows.
}

\begin{proof}[Proof of Proposition \ref{prop:rigidity}, assuming Proposition \ref{prop:v_perp-discrete-general}]\color{White}
	Suppose that $a \in \tilde \cG^q_p$ and that conditions \ref{it:3:A}, \ref{it:3:B} hold for a sequence $v_l \in \cV_{\mathrm{top}}$ ($l \in \NN$). Let $v_l^\perp$ denote the orthogonal projection of $v_l$ to the orthogonal complement $\cU^\perp_q$. Note that \ref{it:3:A} implies that 
	\[ \lim_{l \to a} \norm{v_l}_{\infty} = \lim_{l \to a}\norm{ T_l|_p(x^p) - x^q}_{\infty} \leq 1,\]
	so we may without loss of generality assume that $\norm{v_l}_\infty \leq 2$ for all $l \in \NN$.    
	It now follows from Proposition \ref{prop:v_perp-discrete-general} there exist a finite set $F^{\perp} \subset \cU^\perp_q$ such that $v_l^\perp \in F$ for $a$-almost all $l$. Since $F$ is finite, there exists $v^{\perp} \in F$ such that $v^{\perp}_l = v^\perp$ for $a$-almost all $l$. Put 
	\[v := \lim_{l \to a} v_l \in v^\perp + \cU_q.\]
	 Then $v$ satisfies conditions \ref{it:2:A}, \ref{it:2:B}. Likewise, if condition \ref{it:3:C} holds for a sequence $w_{m,l} \in \cV_{\mathrm{top}}$ ($l,m \in \NN$) then Lemma \ref{lem:R3-wk<=>R3-str} guarantees that there exists a sequence $w_{m} \in \cV_{\mathrm{top}}$ ($m \in \NN$) satisfying condition \ref{it:2:C}. It follows that  $a \in \cG^q_p$.
\end{proof}

{\color{NavyBlue}
The remainder of this subsection is devoted to the proof of Proposition \ref{prop:v_perp-discrete-general}. We have an almost immediate reduction to the case when $p = q$.
}
\begin{proposition}\label{prop:v_perp-discrete-diagonal}\color{Mahogany}
	For any $p \in \KN$ and $C > 0$ there exists a finite set $F \subset \RR^{\cD}$ such that for any $v \in \cV_{\mathrm{top}}$ with $\norm{v}_\infty \leq C$ and any $l \in \NN$, if $T_l|_p(\Vp_p) = \Vp_p + v$ then $v \in F+\cU_p$. 
\end{proposition}
\begin{proof}[Proof of Proposition \ref{prop:v_perp-discrete-general} assuming Proposition \ref{prop:v_perp-discrete-diagonal}]\color{White}
Pick $r \in \beta \NN_0$ such that $r+q=p$. For any $l \in \NN$ it follows from Lemma \ref{lem:translation-ur} and the definition of $T_l|_p$ that 
\begin{align*}
	\forall^r_n \ T_k^n|_q(\Vp_q) &= \Vp_p &\text{and}&& \forall^r_n \ T_l|_p \circ T_k^n|_q &= T_{lk^n}|_q.
\end{align*}
Suppose that $T_l|_p(\Vp_p) = \Vp_q+v$ for some $l \in \NN$ and $v \in \cV_{\mathrm{top}}$ with $\norm{v}_{\infty} \leq C$.
Then there exists $m \in \NN$ such that $T_{m}|_q(\Vp_q) = \Vp_q + v$ (for $r$-almost all $n$, one can take $m = lk^n$). By Proposition \ref{prop:v_perp-discrete-diagonal}, $v \in F + \cU_q$ for a finite set $F$.  
\end{proof}

{\color{Blue}
Let $p \in \KN$. In order to prove Proposition \ref{prop:v_perp-discrete-diagonal}, we need to look more closely into the geometry of $\Vp_p$. We recall that $\bar\ST(\cD)$ denotes the set of standard affine maps $\RR^{\cD} \to \RR^{\cD}$, cf.{} \eqref{eq:def-of-ST(D)}. The definitions made below are local to this subsection, that is, they do not appear later in the paper. 
Let 
\begin{equation}
	\HH_p := \set{T \in \bST(\cD)}{ T(\Vp_p) = \Vp_p}
\end{equation}
be the group consisting of all maps in $\bST(\cD)$ which preserve $\Vp_p$. It follows from Lemma \ref{lem:stab-is-linear} that $\HH_p$ is a connected Lie group. Further, let
\begin{equation}
	\GG_p := \grp{T_k^n|_p^t}{n \in \NN,\ T_k^n|_p(\Vp_p) = \Vp_p,\ t \in \RR}
\end{equation}
be the smallest connected Lie group containing all of maps $T_k^n|_p$ which preserve $\Vp_p$. In particular, if $r \in \beta \NN_0$ is such that $r+p = p$ then $T^n_{k}|_p \in \GG_p$ for $r$-almost all $n$ by Lemma \ref{lem:translation-ur}. Accordingly, we define the Lie algebras $\fh_p$ and $\fg_p$ associated to $\GG_p$ and $\HH_p$, as well as the normal subgroups 
\begin{align}
	\HH_p' &= \HH_p \cap \bSTp(\cD),&&& \GG_p' &= \GG_p \cap \bSTp(\cD),
\end{align}
and their Lie algebras $\fh'_p$ and $\fg'_p$. Finally, fix $n = n_p$ such that $T_k^{n}|_p \in \GG_p$ and put 
\begin{align}
	X_p &:= (\log T_k^{n}|_p)/(n \log k) \in \Lambda + \fg_p' \subset \fg_p, \\
	S_{l}|_{p} &:= \exp(\log(l)X_p) \in \GG_p \cap \Delta_l \bSTp(\cD), &(l \in \NN).
\end{align}
Note that $\GG_p$ is the semidirect product of the one parameter group $\set{\exp(tX_p)}{t \in \RR}$ and $\GG_p'$, and the analogous remark applies to $\HH_p$. It follows from Lemma \ref{lem:exp-is-poly} that $X_p$ and $S_l|_p$ have rational coefficients.
}

{\color{NavyBlue}
Recall that each map in $\bST{}(\cD) \setminus \bSTp(\cD)$ has exactly one fixed point in $\RR^{\cD}$. For $\GG \subset \Aff(\cD)$ and $x \in \RR^\cD$, let $\GG(x) = \set{ T(x)}{ T \in \GG}$ denote the orbit of $x$. We will need the following well-known fact.
}
\begin{theorem}\label{lem:unipotent-orbit-closed}\color{Mahogany}
	Let $d \in \NN$,  $x \in \RR^d$ and let $\GG < \GL(d)$ be a Lie group consisting of lower triangular unipotent matrices. Then $\GG(x)$ is an algebraic variety. 
\end{theorem}
\begin{proof}\color{LightGray}
	\cite[Thm.\ 22.3.6]{TauvelYu-book}.
\end{proof}

\begin{proposition}\label{prop:Vp=G(z)}\color{Mahogany}
	Let $p \in \KN$ and let $z \in \RR^{\cD}$ be the fixed point of a map  in $\GG_p \setminus \GG_p'$. Then $\Vp_p = \GG_p'(z)$. In particular, $\Vp_p$ is defined over $\QQ$.
\end{proposition}
\begin{proof}\color{White} 
Pick $T \in \GG_p \setminus \GG_p'$ with $T(z) = z$. It follows from Lemma \ref{lem:stab-is-linear} that we may replace $T$ with $T^t$ for any $t \in \RR \setminus \{0\}$. Hence, by Lemma \ref{lem:diag-in-ST} we may assume that $T$ is similar to the diagonal matrix $\bar \Delta_k$ and there exists $P \in \bSTp(\cD)$ such that $T = P \bar \Delta_k P^{-1}$. 
Then for any point $x \in \RR^{\cD} \setminus \{z\}$ we have
\[
\norm{T^n(x)-z} = \norm{P \bar \Delta{}_k^n P^{-1}(x-z)} \geq \frac{ k^n \norm{x-z}}{\norm{P}\norm{P^{-1}}}.
\]
As a consequence, for any $\e > 0$ we have
\begin{align}
\label{eq:78:99}
T^{-n}(x) &\to z \text{ as } n \to \infty \text{ for any } x \in \RR^\cD,\\
\label{eq:78:00}
\norm{T^n(x)-z} &\to \infty  \text{ as } n \to \infty, \text{ uniformly in } x \in \RR^\cD \setminus \Ball(z,\e).
\end{align}

It follows from \eqref{eq:78:99} that $z \in \Vp_p$. Directly by the definition of $\GG_p'$, all maps in $\GG_p'$ preserve $\Vp_p$. Thus, $\GG_p'(z) \subset \Vp_p$. It remains to show the reverse inclusion.

Let $r \in \beta \NN$ be such that $r+p = p$, so that $T_k^n|_p \in \GG_p$ for $r$-almost all $n$. Because of the semidirect product structure of $\GG_p$, there exist $S_n' \in \GG_p'$ such that $T_k^n|_p = T^n \circ S_n'$ for $r$-almost all $n$, and thus
\begin{equation}\label{eq:18:80}
	\forall^r_n \ x^{n+p} = T_k^n|_p(x^p) = T^n \circ S_n'(x^p),
\end{equation}
Since the sequence $x^{n+p}$ is bounded, in light of \eqref{eq:78:00} it follows from \eqref{eq:18:80} that
\begin{equation}\label{eq:18:81}
\lim_{n \to r} S_n'(x^p) = z.
\end{equation}
Hence, $z$ belongs to the orbit $\GG_p'(x^p)$, which is closed as a consequence of Theorem \ref{lem:unipotent-orbit-closed}. It follows that there exists a map $S' \in \GG'_p$ such that $S'(x^p) = z$. Let $R_n'$ denote the unique map in $\GG_p'$ such that $R_n' \circ T^n \circ S' = T^n \circ S_n'$. Then
\begin{equation}\label{eq:18:82}
	\forall^r_n \ x^{n+p} = T^n \circ S_n' (x^p) = R_n' \circ T^n(z) = R_n'(z) \in \GG_p'(z).
\end{equation}
Since $\GG'_p(z)$ is an algebraic variety by Theorem \ref{lem:unipotent-orbit-closed}, it follows from \eqref{eq:18:82} and Lemma \ref{lem:alg-lim-basic} that $\Vp_p \subset \GG_p'(z)$ and the first part of the statement follows. 

For the second part, let $z$ be the fixed point of $S_k|_{p}$. Then $z$ is rational.
The Lie algebra $\fg_p'$ is spanned by rational vectors, namely the logarithms of these among the products of maps $T_k^n|_p$ which belong to $\GG_p'$. It follows that $\GG_p'(z)$ contains a topologically (hence also Zariski) dense set of rational points, namely all points of the form $\exp(Z)z$ where $Z \in \fg_p'$ has rational coefficients. Hence, $\GG_p'(z)$ is defined over $\QQ$ by Lemma \ref{lem:alg-cl-Q}.
\end{proof}

{\color{NavyBlue}
As alluded to before, the vector space $\cU_p$ plays an important role in this section. It is a fairly immediate consequence of Proposition \ref{prop:Vp=G(z)} that $\cU_p$ is defined over $\QQ$ as an algebraic variety. With some additional work, we extract a more precise statement, which is related to Proposition \ref{lem:idempotent}.
}

\begin{proposition}\label{lem:U-is-Q}\color{Mahogany}
	Let $p \in \KN$. Then the vector space $\cU_p$ is defined over $\QQ$.
\end{proposition}
\begin{proof}\color{White}
	Let $T \in \GG_p \setminus \GG_p'$ be a map with rational coefficients and let $z \in \QQ^{\cD}$ be the fixed point of $T$. The key idea is to obtain an alternative characterisation of $\cU_p$, namely $\cU_p = \fg_p'z \cap \cV_{\mathrm{top}}$.
		
	Consider any vector $u \in \cU_p$. Then $z+u \in \Vp_p$, so it follows from the description of $\Vp_p$ in Proposition \ref{prop:Vp=G(z)} and the fact that the exponential map $\fg_p' \to \GG_p'$ is surjective that there exists $Z_u\in \fg_p'$ such that $\exp(Z_u)z = z+u$. Because all maps in $\GG_p'$ map $\cV_{\mathrm{top}}$ to $0$, we also have $\exp(Z_u)u = u$ and hence $\exp(Z_u)(z+tu) = z+(t+1)u$ for any $t \in \RR$. It follows that $\exp(Z_u)$ preserves the line $z + \RR u$. Hence, $\exp(t Z_u)z \in z+\RR u$ for any $t \in \RR$ and consequently $Z_uz \in \RR u$. Since $u \in \cU_p$ was arbitrary, $\cU_p \subset \fg_p'z$.
	
	Suppose conversely that $Z \in \fg_p'$ is such that $u_z := Zz \in \cV_{\mathrm{top}}$. Then $ z + u_z = \exp(Z)z \in \Vp_p$. It follows that
	\[
		\Vp_p \supseteq \GG_p'(z+u_z) = \GG_p'(z)+u_z = \Vp_p + u_z,
	\]
	meaning that $u_z \in \cU_p$. 
	
	Thus, we have shown that $\cU_p = \fg'_p z \cap \cV_{\mathrm{top}}$. Since $\fg_p'$ is defined over $\QQ$ and $z$ is rational, $\cU_p$ is also defined over $\QQ$.
\end{proof}

\begin{proof}[Proof of Proposition \ref{prop:v_perp-discrete-diagonal}]\color{White}
Since for any $l \in \NN$, the map $S_l|_p$ preserves $\Vp_p$, we have
\[
	T_l|_p(\Vp_p) = T_l|_p \circ S_l|_p^{-1} (\Vp_p).
\]
Note that $T_l|_p \circ S_l|_p^{-1} \in \bSTp(\cD)$. Recall that $A_l|_p$ is the matrix representation of $T_l|_p$ and let $B_l$ and $C_l = \bar A_l|_p B_l^{-1}$ be the matrix representations of $ S_l|_p $ and $T_l|_p \circ S_l|_p^{-1}$ respectively.  It follows from Lemma \ref{lem:diag-in-ST} that there exists $P \in \bST'(\cD,\QQ)$ such that $X_p = P \Lambda P^{-1}$ and hence $B_l = P \bar \Delta_l P^{-1}$. It follows from Proposition \ref{prop:Tk-basic} that there exists $A_l' \in \bSTp(\cD, \ZZ)$ such that $\bar A_l|_p = A_l' \bar \Delta_l$. Hence,
	\[
		 C_l = A_l' \bar \Delta_l P \bar \Delta_l^{-1} P^{-1}.
	\]
	Since $\bSTp(\cD, \ZZ)$ is preserved under taking inverses and under conjugation by $\Delta_l$, each of $A_l'$, $P^{-1}$ and $\bar \Delta_l P \bar \Delta_l^{-1}$ has rational entries with denominators bounded uniformly with respect to $l$, and hence so does $C_l \in \bSTp(\cD,\QQ)$. Let $Y_l := \log C_l \in \bSTp(\cD)$. Recall that by Lemma \ref{lem:exp-is-poly-nil} the logarithmic map $\bSTp(\cD) \to \bstp(\cD)$ is given by a polynomial formula with rational coefficients, so the coefficients of $Y_l$ are rational and have denominators bounded uniformly in $l$. 
		
	Suppose now that $v \in \cV_{\mathrm{top}}$ and there exists $l \in \NN$ such that
	\[
		T_l|_{p}(\Vp_p) = \Vp_p + v.
	\]	
	Since we are only interested in the equivalence class of $v$ modulo $\cU_p$, we may assume without loss of generality that $v \perp \cU_p$. By the usual abuse of notation, let $I + v e_0^{\mathrm{T}} = \exp(v e_0^{\mathrm{T}}) \in \bSTp(\cD)$ denote the shift by $v$. Put $Z_l := Y_l - v e_0^{\mathrm{T}}$. Since $I + v e_0^{\mathrm{T}}$ commutes with $\bSTp(\cD)$, 
	\[ \exp(Z_l)(\Vp_p) = T_l|_p(\Vp_p) - v = \Vp_p, \]
	meaning that $Z_l \in \fh_p'$. 
Written in the block form corresponding to the partition $\{0\} \cup \cD_{\mathrm{low}} \cup \cD_{\mathrm{top}} \cup \cD_{\mathrm{ab}}$, the matrices $C_l$, $Y_l$ and $Z_l$ take, for some $u_l \in \cV_{\mathrm{top}}$, the form
	\[
	C_l = 
	\begin{bmatrix}
	1 &   &  &  \\
	\ast & I &  &  \\
	\ast & \ast & I &  \\
	* & 0 & 0 & I 
	\end{bmatrix},
	\qquad 
	Y_l = 
	\begin{bmatrix}
	0 &  &  &  \\
	\ast & 0 &  &  \\
	u_l & \ast & 0 &  \\
	* & 0 & 0 & 0 
	\end{bmatrix},
	\qquad
	Z_l = 
	\begin{bmatrix}
	0 &  &  &  \\
	\ast & 0 &  &  \\
	u_l-v & \ast & 0 &  \\
	* & 0 & 0 & 0 
	\end{bmatrix}.
	\] 
	
	In order to avoid excessively complicating the notation, let us assume that $P = I$, meaning that $X_p = \Lambda$ and $S_l|_p = \bar \Delta_l$; we can always reduce to this case by change of basis (cf.\ Lemma \ref{lem:diag-in-ST}). Since $\fh_p$ is a Lie algebra containing $\Lambda = X_p$, it is closed under the map $Z \mapsto [\Lambda, Z]$. In fact, if follows by direct inspection of the diagonal entries that $[\Lambda,Z] \in \fh_p'$ for any $Z \in \fh_p$. As a consequence, for any $Z \in \fh_p$ there exists a decomposition $Z = \sum_\lambda Z^{(\lambda)}$, where  $Z^{(\lambda)} \in \fh_p'$ ($\lambda \in \NN_0$) are such that $[\Lambda,Z^{(\lambda)}] = \lambda Z^{(\lambda)}$. Note that $Z_{\mu, \nu}^{(\lambda)} = 0$ unless $d_{\mu} - d_{\nu} = \lambda$. 

In particular, inspecting the decomposition of $Z_l$ we conclude that $\fh_p'$ contains the matrix $(u_l-v)e_0^{\mathrm{T}} = Z^{(D)}_l$. Hence, $v \in u_l + \cU_p$, which is tantamount to saying that $v$ is the projection of $u_l$ onto $\cU_p^{\perp}$. Since $u_l$ is rational with bounded height and $\cU_p$ is defined over $\QQ$, it follows that $v$ is also rational with bounded height. In particular, there are finitely many possible values of $v$.	
\end{proof}

\subsection{\checkmark\ Semigroupoid structure}\label{ssec:rec-semigrp}\mbox{}
{\color{NavyBlue}
In this section we investigate general properties of the sets of ``ultrafilters of good recurrence'' $\cG^q_p$, $\tcH^q_p$ and $\cH^q_p$ defined by  \eqref{eq:def-of-cH}, \eqref{eq:def-of-tcH}, and \eqref{eq:def-of-cG}. To begin with, we note that these sets are not empty; indeed, they contain an element of a very specific form.
}

\begin{lemma}\color{BurntOrange}\label{lem:k^r-in-H}
	If $p \in \KN$ and $r \in \beta \NN_0$ then $k^r \in \cH^{r+p}_p$.
\end{lemma}
\begin{proof}\color{White}
Put $q = r+p$. It follows directly from definitions of relevant terms that
\[
	\lim_{n \to r} T_k^n|_p(x^p) = x^{q}.
\]
It follows from Lemma \ref{lem:translation-ur} that
\[
	\forall_n^r \ T_k^n|_p(\Vp_p) = \Vp_{q}
\]
For any $m \in \NN$ it follows from the basic definitions that
\begin{align*}
	\forall^r_n \ \bA_{mk^n}|_p &= \lim_{h \to p} \bA_{mk^n}(x^h) 
	\\&= \lim_{h \to p} \bA_{m}(x^{h+n})  \bA_{k}^n(x^{h})
	= \bA_{m}|_{q}  \bA^n_{k}|_p =  \bA_{m}|_{q} \bA^n_{k}|_p,
\end{align*}
from which it follows that
\[ \forall^r_n \ T_m|_{q} \circ T_k^n|_p = T_{mk^n}|_p. \]
Hence, $k^r$ satisfies properties \ref{it:1:A}, \ref{it:1:B} and \ref{it:1:C}. 
\end{proof}

We next show some basic closure properties of the sets $\cG^q_p$ ($p,q \in \KN$), and of the vectors $v^{a}, w^{a}_m$ ($a \in \cG^q_p, m \in \NN$) introduced in \eqref{eq:def-of-v-and-w}.

\begin{proposition}\label{lem:cGp-basic}\color{Mahogany}
	Let $p \sim q \sim r \in \KN$.
	\begin{enumerate}
	\item\label{it:19:B} The maps $a \mapsto v^a$ and $a \mapsto w^a_m$ ($m \in \NN$) are continuous ($a \in \cG^q_p$).
	\item\label{it:19:A} The sets $\cG^q_p = \tcG^q_p$, $\tcH^q_p$ and $\cH^q_p$ are closed.
	\item\label{it:19:C} We have the inclusion $\cG^r_q \cdot \cG^q_p \subset \cG^r_p$ and if $a \in \cG^r_q$, $b \in \cG^q_p$ then 
	\begin{align}
	\label{eq:99:01}
	 v^{a \cdot b} &\equiv v^a + \lim_{m \to a} \fp{m^{\s} v^b} \pmod{\mathbb{Z^{\cD}}}, 
	 \\ \label{eq:99:02}
	 \forall_n \ \forall^a_l \ w_{n}^{a \cdot b} &= w^b_{nl} + w^a_{n} - n^{\s} w^b_l.
	\end{align}
	\end{enumerate}
\end{proposition}
\begin{proof}\color{White}
\begin{enumerate}[wide]
\item The maps $a \mapsto v^a$ and $a \mapsto w^a_m$ can be defined as limits:
\[ v^a = x^q - \lim_{l \to a} T_l|_p(x^p) \qquad \text{ and } \qquad w^a_m = \lim_{l \to a} (T_{ml}|_p(x^p) - T_m|_q(T_l|_p(x^p)). \] Hence they are continuous by the definition of the limit (cf.{} Sec.{} \ref{ssec:prelims-ultrafilters}).

\item We first show that $\cG^q_p$ is closed. Let 
\[
	E_2 := \set{ l \in \NN}
		{ T_l|_p(\Vp_p) = \Vp_q + v_l^\perp \text{ for some } 
			v_l^{\perp} \in \cV_{\mathrm{top}} \cap \cU_q^\perp
	}.
\]
This definition is motivated by the observation that for $a \in \beta \NN$, the condition \ref{it:3:B} is satisfiable if and only if $a \in \bar E_2$. Note that for $l \in E_2$ the vector $v_l^\perp$ is determined uniquely, so we have a well-defined map $l \mapsto v_l^\perp$, $E_2 \to \cV_{\mathrm{top}} \cap \cU_q^\perp$. Since $\Sigma^q_p \subset \bar E_2$ and $\bar E_2$ is closed, also $\cl(\Sigma^q_p) \subset \bar E_2$.
Let us also consider the map $u \colon \cG^q_p \to \RR^{\cD}$ given by
\begin{equation}\label{eq:def-of-ua}
	u^a := \lim_{l \to a} T_l|_p(x^p) - x^q - \lim_{l \to a} v_l^{\perp} 
	=\lim_{l \to a} \bra{ v^a -  v_l^{\perp}},
\end{equation}
where the second equality follows from \ref{it:2:A}. For each $a \in \cG^q_p$ we have
\[
	\forall^a_l \ \Vp_q + v^a = T_l|_p(\Vp_p)= \Vp_q + v_l^\perp,
\]
and consequently $u^a \in \cU_q$. This means that, for $a$-almost all $l$, the vectors $u^a$ and $v_l^{\perp}$ are the orthogonal projections of $v^a$ onto $\cU_q^\perp$ and $\cU_q^\perp$ respectively, and in particular $\norm{u^a}_\infty \leq 1$. Put
\[
	E_1 := \set{ l \in \NN}{ \norm{ T_l|_p(x^p) - x^q - v_l^{\perp} }_\infty \leq 2}.
\]	
Then $\Sigma^q_p \subset \bar E_1$. We may extend $u$ to a map $\cl(\Sigma^q_p) \to \cU_q^\perp$ defined by the same formula \eqref{eq:def-of-ua} and (by a slight abuse of notation) denoted with the same symbol. One can now directly verify that conditions \ref{it:3:A} and \ref{it:3:B} are satisfied for arbitrary $a \in \cl(\Sigma^q_p)$ and the sequence $v_l$ given by $v_l = u^a + v_l^{\perp}$.

We next address for condition \ref{it:3:C}. Consider the set
\[
	E_3 := \set{ l \in \NN}{
		\text{ if } m \in \NN \text{ then } T_{ml}|_p = T_{m}|_q \circ T_l|_p + w_{m,l}e_0^{\mathrm{T}} \text{ for some } w_{m,l} \in \cV_{\mathrm{top}} 
	}
\]
For $l \in E_3$, the vector $w_{m,l}$ is determined uniquely by the conditions above, and condition \ref{it:3:C} implies that $\cG^q_p \subset \cl(\cG^q_p) \subset \bar E_3$. It follows that for each $a \in \cl(\cG^q_p)$ there exists a sequence $w_{m,l}$ such that condition \ref{it:3:C} is satisfied. Together with previous considerations, this implies that $\cl(\cG^q_p) \subset \tcG^q_p = \cG^q_p$, meaning that $\cG^q_p$ is closed.

We now turn to $\cH^q_p$ and $\tcH^q_p$. Recall that an ultrafilter $a \in \cG^q_p$ belongs to $\tcH^q_p$ if and only if $v^a = 0$. Hence the fact that $\tcH^q_p$ is closed follows from the previous item \eqref{it:19:B}. Likewise, $a \in \cH^q_p$ if and only if additionally $w^a_m = 0$ for all $m \in \NN$, so $\cH^q_p$ is closed for analogous reasons.

\item 

Let $a\in \cG^r_q$ and $b \in \cG^q_p$. Then, by \ref{it:2:A} and \ref{it:2:C} 
\begin{align*}
	\lim_{l\to a} \lim_{m \to b} \ T_{lm}|_p(x^p) 
	&= \lim_{l\to a} \lim_{m \to b} \bra{ T_l|_q \circ T_m|_p(x^p) + w_{l}^b}
	\\ &= \lim_{l\to a} \bra{ T_l|_q \bra{ x^p } + l^{\s} v^b + w_l^b }
	= x^r + v^a + \lim_{l\to a} \bra{l^{\s} v^b + w_l^b}.
\end{align*}
In particular, the limit exists since the sequence $T_{lm}|_p(x^p)$ is bounded, and \ref{it:2:A} holds for $a \cdot b$ with 
\begin{align}\label{eq:04:00}
	v^{a \cdot b} = v^a + \lim_{l\to a} \bra{l^{\s} v^b + w_l^b}.
\end{align}
Since $w_l^b$ has integer coordinates, formula \eqref{eq:99:01} follows by projecting \eqref{eq:04:00} to $\RR^{\cD}/\ZZ^{\cD}$. A similar computation shows that 
\begin{align*}
	\forall^a_l \ \forall^b_m \ T_{lm}|_p(\Vp_p) &= T_{l}|_q \circ T_m|_p(\Vp_p) + w_l^b \\
	&= 	T_{l}|_q(\Vp_q + v^b) + w_l^b = \Vp_r+ v^a + l^{\s} v^b + w_l^b.
\end{align*}
Let $\pi \colon \cV_{\mathrm{top}} \to \cU_r^{\perp}$ be the orthogonal projection. It follows from Proposition \ref{prop:v_perp-discrete-diagonal} and \eqref{eq:04:00} that
\[
	\forall^a_l \ \forall^b_m \ T_{lm}|_p(\Vp_p) = \Vp_r + \pi\bra{v^a + l^{\s} v^b + w_l^b} = \Vp_r + \pi\bra{ v^{a \cdot b} } = \Vp_p + v^{a \cdot b}.
\]
Hence, \ref{it:2:B} also holds for $a \cdot b$ with $v^{a \cdot b}$ given by \eqref{eq:04:00}. Finally, if $n \in \NN$ then 
\begin{align*}
	\forall^a_l \ \forall^b_m &\ T_{nlm}|_p 
		= T_{nl}|_{q} \circ T_{m}|_p + w^b_{nl} e_0^{\mathrm{T}} 
		= T_n|_r \circ T_l|_q \circ T_m|_p + w^b_{nl} e_0^{\mathrm{T}} + w^a_{n} e_0^{\mathrm{T}}, 
	\\ \forall^a_l \ \forall^b_m &\ T_n|_r \circ T_{lm}|_p = T_n|_r \circ T_l|_q \circ T_m|_p + n^{\s} w^b_l  e_0^{\mathrm{T}}.
\end{align*}
Comparing the two expressions we conclude that
\[
\forall^a_l \ \forall^b_m \ T_{nlm}|_p =  T_n|_p \circ T_{lm}|_p 
	+ \bra{w^b_{nl} + w^a_{n} - n^{\s} w^b_l}e_0^{\mathrm{T}}.
\]
Evaluating both sides at $x^p$ and passing to the relevant limits we obtain:
\[
\lim_{l \to a} \lim_{m \to b} T_{nlm}|_p(x^p) = 
\lim_{l \to a} \lim_{m \to b} \bra{ T_n|_p \circ T_{lm}|_p(x^p)} + 
\lim_{l \to a} \lim_{m \to b} \bra{ w^b_{nl} + w^a_{n} - n^{\s} w^b_l}.
\]
Note that the first two out of three limits above exist because $T_{lm}|_{p}(x^p)$ and $T_{nlm}|_{p}(x^p)$ are both bounded. It follows that the third limit exists as well. Hence, condition \ref{it:2:C} holds for $a \cdot b$ and for each $n \in \NN$ we have
\[
\forall^a_l \ w_{n}^{a \cdot b} =w^b_{nl} + w^a_{n} - n^{\s} w^b_l. \qedhere
\]
\end{enumerate}
\end{proof}

\subsection{Abelian component}\label{ssec:rec-abel}

In this section, we investigate in more detail the sets $\Zeta^q_p$ defined in \eqref{eq:def-of-Zeta}. For this purpose, it will be helpful to consider the relation $\eqab$ on $\beta \NN_0$ defined by declaring for two ulrafilters $a,a' \in \beta \NN_0$ that
\begin{equation}
	a \eqab a' \text{ if and only if } \lim_{m \to a} m^d \alpha = \lim_{m \to a'} m^d \alpha \text{ for all } d \in \NN,\ \alpha \in \RR/\ZZ.
\end{equation}
With this piece of notation, for $p \sim q \in \KN$ we obtain an alternative description
\begin{equation}
	\Zeta_p^q = \set{ a \in \beta \NN_0}{ a \cdot k^p \eqab k^q }.
\end{equation}
The following lemma shows basic properties of the relation we have just defined.

\begin{lemma}\color{BurntOrange}
	\begin{enumerate} 
	\item If $a,a' \in \beta \NN_0$ are two ultrafilters with $a \eqab a'$ then for each polynomial map $f \colon \ZZ \to \RR/\ZZ$ it holds that $\displaystyle\lim_{m \to a} f(m) = \lim_{m \to a'} f(m)$ 
	\item The relation $\eqab$ is an equivalence relation on $\beta \NN_0$.
	\item The equivalence classes of $\eqab$ are closed.
	\item If $a,a',b,b' \in \beta \NN_0$ and $a \eqab a'$ and $b \eqab b'$ then $a + b \eqab a'+b'$ and $a \cdot b \eqab a' \cdot b'$.
	\item If $p,a_n,a_n' \in \beta \NN_0$ and $a_n \eqab a_n'$ for $p$-almost all $n$, then 
	$\displaystyle\lim_{n \to p} a_n \eqab \lim_{n \to p} a_n' $.
	\end{enumerate}
\end{lemma} 
\begin{proof}\color{White}
	\begin{enumerate}[wide]
	\item Follows from the fact that any polynomial is a sum of monomials.
	
	\item For $a \in \beta \NN_0$, let $\varphi_a$ denote the map $\RR[x] \to \RR/\ZZ$ given by $f \mapsto \lim_{m \to a} f(m)$. Then $a \eqab a'$ if and only if $\varphi_a = \varphi_{a'}$. It follows that $\eqab$ is an equivalence relation.		
  
	\item The equivalence class of $a$ is the intersection of closed sets of the form 
	\[ \set{b \in \beta \NN_0}{ \lim_{m \to b} f(m) = \lim_{m \to a} f(m) },\]
	 where $f \colon \ZZ \to \RR/\ZZ$ is a  polynomial. 
	\item Note that polynomials are closed under limits in the sense that if $g \colon \ZZ \to \RR/\ZZ$ is a polynomial and $p \in \beta \NN_0$ is an ultrafilter then the map $m \mapsto \lim_{l \to p} g(m+l)$ is again a polynomial. Thus, 	for each polynomial $f \colon \ZZ \to \RR/\ZZ$, 
	\begin{align*}
			\lim_{m \to a + b} f(m) &= \lim_{m \to a} \lim_{l \to b} f(m+l) 
			= \lim_{m \to a} \lim_{l \to b'} f(m+l) 
			\\& = \lim_{m \to a'} \lim_{l \to b'} f(m+l) = \lim_{m \to a' + b'} f(m).
	\end{align*}	
Hence $a + b \eqab a' + b'$. The proof that $a \cdot b \eqab a' \cdot b'$ is fully analogous.
\item Put $a = \lim_{n \to p} a_n$ and $a' = \lim_{n \to p} a_n'$. Then for each polynomial $f \colon \ZZ \to \RR/\ZZ$ we have
\begin{align*}
\lim_{m \to a} f(m) = \lim_{n \to p} \lim_{m \to a_n} f(m) = 
\lim_{n \to p} \lim_{m \to a_n'} f(m) = \lim_{m \to a'} f(m).
\end{align*}
Hence, $a \eqab a'$. 	\qedhere
\end{enumerate}
\end{proof}

\begin{example}\color{OliveGreen}
	If $a \in \EN$ then $a \eqab 0$. More generally, the equivalence class of $0$ is a two-sided multiplicative ideal and an additive semigroup containing $\EN$.
\end{example}

Using the lemma above, we can derive some basic properties of the sets $\Zeta^q_p$. Note that they bear a strong resemblance to properties enjoyed by the sets $\cG^q_p$. 

\begin{lemma}\color{BurntOrange}\label{lem:Zeta-basic}
	\begin{enumerate} 
	\item If $p,r \in \beta \NN_0$ then $k^r \in \Zeta^{r+p}_p$.
	\item If $p,q \in \beta \NN_0$ then $\Zeta_p^q$ is closed and is a union of equivalence classes of $\eqab$.
	\item If $p, q, r \in \beta \NN_0$ then $\Zeta^r_q \cdot \Zeta^q_p \subset \Zeta^r_p$.
	\item If $p,q,r \in \beta \NN_0$ and $a_n \in \Zeta^{n+q}_p$ for $r$-almost all $n$, then $\lim_{n\to r} a_n \in \Zeta^{r+q}_p$. 
	\end{enumerate}
\end{lemma} 
\begin{proof}\color{White}
	\begin{enumerate}[wide]
	\item Follows from the fact that $k^r \cdot k^p = k^{r+p}$. 
	\item The set $\Zeta_p^q$ is the preimage of the (closed) equivalence class of $k^q$ via the continuous map $a \mapsto a\cdot k^p$, hence it is closed. If $a \eqab a'$ and $a \in \Zeta^q_p$ then 
	\(
	a' \cdot k^p \eqab a \cdot k^p \eqab k^q,
	\)
	so $a' \in \Zeta^q_p$.
	\item If $a \in \Zeta^r_q$ and $b \in \Zeta^q_p$ then
	\(
		a \cdot b \cdot k^p \eqab a \cdot k^q \eqab k^r,
	\)
	so $a \cdot b \in \Zeta^r_p$.
	\item Put $a := \lim_{n\to r} a_n$. Since $a_n \cdot k^p \eqab k^{n+q}$ for $r$-almost all $n$ and the map $b \mapsto b \cdot k^p$ is continuous, it follows that 
	\[ a \cdot k^p = \lim_{n \to r} a_n \cdot k^p \eqab \lim_{n \to r} k^n \cdot k^p = k^{r+q}.\]
	 and hence $a \in \Zeta^{r+q}_p$.	
	\qedhere
	\end{enumerate}
\end{proof}

{\color{NavyBlue}
The addition formula \eqref{eq:99:01}, expressing $v^{a \cdot b}$ in term of $v^a$ and $v^b$, becomes considerably simpler under the additional assumption that $a \in \Zeta^r_q$ and $b$ takes a special form. For the sake of simplicity, we only record the ``diagonal'' case, corresponding to $p=q=r$ in Proposition \ref{lem:cGp-basic}, but the interested Reader will easily derive a more general statement. 
}

\begin{lemma}\label{lem:v-addition}\color{BurntOrange}
	Let $p \in \KN$ and $a,b \in \cG^p_p$. Suppose further that $a \in \Zeta^p_p$ and that $b \in k^p \cdot \cG^q_p$ for some $q \sim p$ with $p+q = p$. Then $v^{a \cdot b} \equiv v^a + v^b \bmod{\ZZ^{\cD}}$. 
\end{lemma}
\begin{proof}
	Let $b = k^p \cdot c$ where $c \in \cG^q_p$. Recall that $k^p \in \cH^{p+q}_q = \cH^p_q$, so in particular $v^{k^p} = 0$. 
	By Proposition \ref{lem:cGp-basic},
	\begin{align}
		v^b = v^{k^p \cdot c} &\equiv \lim_{m \to k^p} \fp{ m^D v^c } \pmod{ \ZZ^{\cD} }.
	\end{align}
	By another application of Proposition \ref{lem:cGp-basic}, 
	\begin{align*}
		v^{a \cdot b} 
		&\equiv v^a + \lim_{m \to a} \fp{m^{\s} v^b}
		\equiv v^a + \lim_{m \to a}\fp{ m^{\s} \lim_{l \to k^p} \fp{l^{\s} v^c}} \\
		& \equiv v^a + \lim_{m \to a \cdot k^p} \fp{m^{\s} v^c} 
		\equiv v^a + \lim_{m \to k^p} \fp{m^{\s} v^c} \\
		& \equiv v^a + v^b		
		\pmod{ \ZZ^{\cD} }. \qedhere
	\end{align*}
	
\end{proof}

{\color{NavyBlue}
With Lemma \ref{lem:v-addition} in hand, we are ready to show how elements of $\cG^q_p$ can be combined to produce an element of $\cH^q_p$.
}

\begin{corollary}\color{Mahogany}\label{cor:v-add}
Let $p \in \KN$ and $x^p \in (0,1)^\cD$. The following conditions are equivalent:
\begin{enumerate}
\item\label{it:69:A} $\cG^q_p \cap \Zeta^q_p \cap \Xi \neq \emptyset$ for some $q \sim p$ with $x^q \in (0,1)^{\cD}$.
\item\label{it:69:B} $\tcH^q_p \cap \Zeta^q_p \cap \Xi \neq \emptyset$ for all $q \sim p$ with $x^q \in (0,1)^{\cD}$.
\end{enumerate}
\end{corollary}
\begin{proof}\color{White}
It is clear that the condition \eqref{it:69:B} implies the condition \eqref{it:69:A}, so there is only one implication to prove. Suppose that \eqref{it:69:A} holds for some $q \sim p$. Without loss of generality, we may assume that $q = p$. (Otherwise, we can find $r \in \KN$ such that $p = r+q$ and $a \in \cG^q_p \cap \Zeta^q_p \cap \Xi$; then $k^r \cdot a \in \cG^p_p \cap \Zeta^p_p \cap \Xi$.) Also, it will suffice to show that $\tcH^p_p \cap \Zeta^p_p \cap \Xi \neq \emptyset$. (Otherwise, we can find $r \in \KN$ such that $q = r + p$ and $a \in \tcH^p_p \cap \Zeta^p_p \cap \Xi$; then $k^r \cdot a \in \tcH^q_p \cap \Zeta^q_p \cap \Xi$.)

Let us now pick $u \sim p$ which is idempotent and $r$ such that $r+p = u$. 
For $a \in \cG^p_p \cap \Zeta^p_p \cap \Xi$, let $\bar{a} := k^{p+r} \cdot a$; note that $\bar a \in \cG^p_p \cap \Zeta^p_p \cap \Xi$ and that $\bar{a} \in k^{p} \cdot \cG_p^u$. 
Hence, it follows from Corollary \ref{cor:v-add} that for any $a, b \in \cG^p_p \cap \Zeta^p_p \cap \Xi$ it holds that 
\begin{equation}\label{eq:781:1}
 v^{\bar a \cdot \bar b} \equiv v^{\bar a} + v^{\bar b} \bmod{\ZZ^{\cD}}. 
\end{equation}
Let us consider the set $A \subset \RR^{\cD}/\ZZ^{\cD}$ given by
\[
	A := \set{ v^{\bar a} \bmod{\ZZ^{\cD}} }{ a \in \cG^p_p \cap \Zeta^p_p \cap \Xi }.
\]
It follows from \eqref{eq:781:1} that $A$ is a semigroup, and in particular it contains elements arbitrarily close to $0$. Since $\cG^p_p \cap \Zeta^p_p \cap \Xi$ is closed, there exists $c \in \cG^p_p \cap \Zeta^p_p \cap \Xi$ such that $v^c \in \ZZ^{\cD}$. Since $x^p$ is an interior point of $(0,1)^{\cD}$, it follows that $v^c = 0$, meaning that $c \in \tcH^p_p$ and so $\tcH^p_p \cap \Zeta^p_p \cap \Xi \neq \emptyset$.
\end{proof}

{
We close this section by showing that, using Theorem \ref{thm:recurrence-strong} as a black box, we can produce ultrafilters in $\tcH^q_p \cap \Zeta^q_p$ (at the cost of a slight increase in the complexity of $\cD$). Recall that complexity was defined in Section \ref{ssec:setup-basic} and that we equip the family of eventually zero sequences in $\NN_0 \cup \{\infty\}$ with reverse lexicographical order. We will apply the following lemma with $\vec c = (\infty,c_1,c_2,\dots)$ where $c_i = \complexity(\cE)_i$; note that  if $\cD \not \subseteq \NN$ then 
$\complexity(\cE) < \vec c < \complexity(\cD)$. Following Remark \ref{rmk:dep-on-D}, in the following result we make the dependence of $\tcH^q_p$ on $\cD$ explicit, and hence write $\tcH^q_p[\cD]$, etc.
}

\begin{proposition}\label{lem:bootstrap}\color{Mahogany}
	Let $\vec c = (\infty,c_1,c_2,\dots) \in \bra{\NN_0 \cup \{\infty\}}^\infty$. 
	Suppose that Theorem \ref{thm:recurrence-strong} holds for all $\cD$ such that $\complexity(\cD) \leq \vec c$. Then, for all $\cD$ such that $\complexity(\cD) \leq \vec c$, the intersection $\tcH^q_p[\cD] \cap \Xi \cap \Zeta^q_p$ is nonempty for all $p \sim q \in \KN$ with $x^p,x^q \in (0,1)^\cD$.
\end{proposition}
\begin{proof}\color{White}
For a monomial $\alpha x^d$ with $\alpha \in \RR/\ZZ$ and $d \in \NN$, let 
\[
	\Alpha(\alpha,d) := \set{ a \in \cG^q_p}
	{ \lim_{m \to a} m^d \lim_{n \to p} k^{dn} \alpha = \lim_{n\to q} k^{dn} \alpha }.
\]
Note that $a \in \Zeta^q_p$ if and only if $a \in \Alpha(\alpha,d)$ for all $\alpha \in \RR/\ZZ$ and $d \in \NN$. 

Pick any $\alpha$ and $d$. Suppose first that there exists $i \in \cD \cap \NN$ such that $x^0_i \equiv \alpha \bmod{\ZZ}$ and $d_i = d$. Let $\pi \colon \RR^{\cD} \to \RR/\ZZ$ be the projection map $x \mapsto x_i \bmod \ZZ$. Then for any $a \in \tcH^q_p[\cD]$,
\begin{align*}
	 \lim_{m \to a}  m^d \lim_{n \to p} k^{dn} \alpha
	&= \lim_{m \to a} \lim_{n \to p} \pi\bra{ T_m (x^n)} 
	\\&= \lim_{m \to a} \pi\bra{ T_m|_p (x^p)} = \pi(x^q) 
	= \lim_{n \to q} k^{dn} \alpha,
\end{align*}
and consequently $\tcH^q_p[\cD] \subset \Alpha(\alpha,d)$.

Secondly, suppose that $\lim_{n\to p} k^{dn} \a = 0$. Let $r \in \beta \NN$ be such that $q = r+p$. Then 
\[ \lim_{n\to q} k^{dn} \a = \lim_{m\to r} k^{dm} \lim_{n\to p} k^{dn} \a =  0,\]
and consequently $\tcH^q_p[\cD] \subset \Alpha(\alpha,d)$. We reach the same conclusion if $\lim_{n\to q} k^{dn} \a = 0$.

Thirdly, suppose that $\lim_{n\to p} k^{dn} \a \neq 0$, $\lim_{n\to q} k^{dn} \a \neq 0$ and there exists no index $i$ such that $x^0_i \equiv \alpha \bmod{\ZZ}$ and $d_i = d$. Consider a larger index set $\hat\cD = \cD \cup \{j\}$ where $j \in \NN \setminus \cD$ and $d_j = d$. Let $\hat x^0 \in \RR^{\hat \cD} \simeq \RR^{\cD} \times \RR$ be the lift of $x^0$ with $\hat x^0_j = \alpha$. The same reasoning as before shows that 
\[ \tcH^q_p[\hat\cD] \subset \Alpha(\alpha,d) \cap \tcH^q_p[\cD].\]

Iterating this construction we conclude that for any finite sequences $\alpha^{(1)}, \dots, \alpha^{(N)} \in \RR/\ZZ$ and $d^{(1)},\dots,d^{(N)} \in \NN$ there exists an index set $\hat\cD^{(N)}$ with $\complexity(\hat\cD^{(N)}) < \vec{c}$ such that
\begin{equation*}\label{eq:396:1}
 \tcH^q_p[\hat\cD^{(N)}] \subset \bigcap_{i=1}^N \Alpha(\alpha^{(i)},d^{(i)}) \cap \tcH^q_p[\cD]
\end{equation*}
Hence, it follows from Theorem \ref{thm:recurrence-strong} applied to the system $\hat\cD^{(N)}$ that 
\[
\bigcap_{i=1}^N \Alpha(\alpha^{(i)},d^{(i)}) \cap \tcH^q_p[\cD] \cap \Xi \neq \emptyset.
\]
Each of the sets $\Alpha(\alpha,d)$ is closed, so it follows from compactness of $\beta \NN_0$ that
\[
\Zeta^q_p \cap \tcH^q_p[\cD] \cap \Xi = \bigcap_{\alpha,d} \Alpha(\alpha,d) \cap \tcH^q_p[\cD] \cap \Xi \neq \emptyset,
\]
where the intersection runs over all $\alpha \in \RR/\ZZ$ and $d \in \NN$.
\end{proof} 

\subsection{Inductive step}\label{ssec:recurrent-induction}

In this section we collect the ingredients which will be needed for the inductive step. For this reason we again need to keep track of the index set $\cD$ (see Remark \ref{rmk:dep-on-D}). In particular, we write $\cG^q_p[\cD]$ rather than $\cG^q_p$, etc. Recall that $\cE$ is given by \eqref{eq:def-of-D-low}. The proof of the following proposition occupies the remainder of this section.

\begin{proposition}\color{Mahogany}\label{prop:induction} Let $p \sim q\in \KN$ let $r \in \beta \NN$ be such that $r+p = q$. 
	Let $a_n \in \tcH_p^{n+p}[\cE]$ for $r$-almost all $n$. Then $\displaystyle\lim_{n \to r} a_n \in \Sigma_p^q[\cD]$.
\end{proposition}

Fix $p,q,r$ as above. We would like to associate conditions \ref{it:3:A}, \ref{it:3:B} and \ref{it:3:C} to a statement concerning the maps $T_l|_p$ which can be expressed in terms of semialgebraic geometry. This motivates us to define for $\e > 0$:
\begin{align}
\label{eq:def-of-cS1}
\cS_1(\e) &:= \set{ (T,v) \in \Aff(\cD) \times \cV_{\mathrm{top}}}{ \norm{T(x^p) - (x^q+v)}  < \e}\\
\label{eq:def-of-cS2}
\cS_2 &:= \set{ (T,v) \in \Aff(\cD) \times \cV_{\mathrm{top}} }{ T(\Vp_p) = \Vp_q + v \text{ and } T \text{ is invertible} }
\end{align} 
It follows directly from expanding the relevant definitions that for an ultrafilter $a \in \beta \NN$ and a sequence $v_l \in \cV_{\mathrm{top}}$ ($l \in \NN$), the conditions \ref{it:3:A} and \ref{it:3:B} hold if and only if
\[ \forall^a_l \ \bra{T_l|_p, v_l} \in \bigcap_{\e > 0} \cS_1(\e) \cap \cS_2. \]
\begin{lemma}\label{lem:ind-1}
	The sets $\cS_1(\e)$ ($\e > 0$) and $\cS_2$ are semialgebraic.
\end{lemma}
\begin{proof}
	Follows immediately from the fact that balls are defined by semialgebraic formulae and  Lemma \ref{lem:stab-is-alg}.
\end{proof}
The situation with condition \ref{it:3:C} is marginally more complicated. For $a \in \beta \NN$, the existence of a sequence $w_{m,l} \in \cV_{\mathrm{top}}$ ($m,l \in \NN$) satisfying \ref{it:3:C} is equivalent to the following condition:
\begin{equation}\label{eq:it-3-C-variant}
 \forall_m\ \forall^a_l \ \forall^p_n \ \bA_m(T_l|_p(x^n)) \equiv  \bA_m|_q \bmod{\cV_{\mathrm{top}}e_0^{\mathrm{T}}},
\end{equation}
where $\cV_{\mathrm{top}}e_0^{\mathrm{T}} = \set{ w e_0^{\mathrm{T}}}{ w \in \cV_{\mathrm{top}} }$ and $we_0^{\mathrm{T}}$ denotes the constant (affine) map $v \mapsto w$ on $\RR^{\cD}$. For each $m \in \NN$ and $n \in \NN_0$ let $\cT_{m}^n(\e)$ be the set of pairs $(T,v) \in \cS_1(\e)$ with $T$ given by $T(x) = A x - b$ ($x \in \RR^\cD$), such that \
\[ \bA_m\bra{T(x^n)-v} \equiv \bA_m|_q \bmod{\cV_{\mathrm{top}}e_0^{\mathrm{T}}}\quad \text{ and } \quad \norm{A}_{\mathrm{F}} < 1/\norm{x^n-x^p}\]
where $\norm{A}_{\mathrm{F}} = \sqrt{ \operatorname{Tr} A^{\mathrm{T}} A} \geq  \norm{A}$ denotes the Frobenius norm. Put
\begin{align}
\label{eq:def-of-cS3}
\cS_3(\e) := \bigcap_{m=1}^{\ceil{1/\e}} \lim_{n \to p} \cT^n_m(\e).
\end{align} 
\begin{lemma}\label{lem:ind-2}
	The set $\cS_3(\e)$ ($\e > 0$) is semialgebraic.
\end{lemma}
\begin{proof}
If $(T,v) \in \cS_1(\e)$ and $\norm{T} < 1/\norm{x^n-x^p}$ then 
\[
	\norm{T(x^n) - v - x^q} \leq \norm{T} \cdot \norm{x^n - x^p} + \norm{T(x^p) - v - x^q} \leq 1 + \e < 2.
\]
It follows from Proposition \ref{prop:Tk-basic} and Lemma \ref{lem:gp=pw-poly} that the restriction of $\bA_m$ to $\Ball(x^q,2)$ is piecewise polynomial. Hence, each of the sets $\cT^n_m(\e)$ is semialgebraic ($m\in \NN$, $n \in \NN_0$). It follows from Proposition \ref{prop:semialg-limit} (cf.\ Example \ref{ex:semialg-param}) that for each $m \in \NN$ the set $\lim_{n \to p} \cT^n_m(\e)$ is semialgebraic. It follows that $\cS_3(\e)$ is a finite intersection of semialgebraic sets, and hence a semialgebraic set.
\end{proof}
Combining definitions \eqref{eq:def-of-cS1}, \eqref{eq:def-of-cS2} and \eqref{eq:def-of-cS3}, for $\e > 0$ we introduce  
\begin{equation}\label{eq:def-of-cT}
	\cR(\e) := \set{ T \in \Aff(\cD) }{ (T,v) \in \cS_1(\e) \cap \cS_2 \cap \cS_3(\e) \text{ for some } v \in \cV_{\mathrm{top}}}.
\end{equation}
\begin{lemma}\label{lem:ind-3}
	The set $\cR(\e)$ ($\e > 0$) is semialgebraic.
\end{lemma}
\begin{proof}
	Follows from Tarski--Seidenberg theorem.
\end{proof}

We are now ready to make the connection between the semialgebraic sets $\cR(\e)$ and the set $\Sigma^q_p[\cD]$ more precise.

\begin{lemma}\label{lem:ind-4}
	Let $a \in \beta \NN$. Then $a \in \Sigma^q_p[\cD]$ if and only if for each $\e > 0$ it holds that
	\begin{equation}\label{eq:378:1}
		\forall^a_l \ T_{l}|_{p} \in \cR(\e).
	\end{equation}	 

\end{lemma}
\begin{proof}
If $a \in \Sigma^q_p[\cD]$ then \eqref{eq:378:1} follows directly from the construction of the sets $\cR(\e)$ and accompanying discussion.

Suppose now that \eqref{eq:378:1} holds. It follows from the definition of $\cR(\e)$ that there exist $v_{\e,l} \in \cV_{\mathrm{top}}$ ($l \in \NN$, $\e > 0$) such that for each $\e > 0$ it holds that
\begin{align}
	\label{eq:49:11}
	\forall^a_l &\ \norm{T_l|_p(x^p) - (x^q +v_{\e,l}) } < \e \\ 
	\label{eq:49:12}
	\forall^a_l &\ T_l|_p(\Vp_p) = \Vp_q + v_{\e,l} \\ 
	\label{eq:49:13}
	\forall_m \ \forall^a_l \ \forall^p_n &\ \bA_m\bra{ T_{l}|_p(x^n) - v_{\e,l} } \equiv \bA_m|_q \bmod{\cV_{\mathrm{top}} e_0^{\mathrm{T}}}.
\end{align}
Let $w := \lim_{l \to a} T_l|_p(x^p) - x^q$. We conclude from \eqref{eq:49:11} that
\begin{align*}
	\lim_{\e \to 0} \lim_{l \to a} v_{\e,l} = w.
\end{align*}
In particular, we may assume without loss of generality that the sequence $v_{\e,l}$ is bounded.

Let $E_2 \subset \NN$ be the set of all integers $l \in \NN$ such that there exists $v \in \cV_{\mathrm{top}}$ with
\(
	T_l|_p(\Vp_p) = \Vp_q + v.
\) Note that \eqref{eq:49:12} guarantees that $E_2 \in a$. For each $l \in E_2$, let $v_l^{\perp} \in \cV_{\mathrm{top}}$ be a the unique vector with $v_l^{\perp} \in \cU_q^\perp$ and 
\[
	T_l|_p(\Vp_p) = \Vp_q + v_{l}^{\perp}.
\]
It follows from \eqref{eq:49:12} that $v_{\e,l} \in v_l^{\perp} + \cU_q$ ($\e > 0$).
For $l \in E_2$, let $v_l$ be the element of $v_l^{\perp} + \cU_q$ such that $\norm{w - v_l}$ is smallest possible (that is, the orthogonal projection of $w$). Then for each $\e > 0$ we have
\[
	\forall^a_l \ \norm{T_l|_p(x^p) - (x^q +v_{l}) } \leq \norm{T_l|_p(x^p) - (x^q +v_{\e,l}) } < \e.
\]
Letting $\e \to 0$ we conclude that
\[
	\lim_{l \to a} T_l|_p(x^p) = x^q + \lim_{l \to a}v_{l}.
\]
Thus, the conditions \ref{it:3:A} and \ref{it:3:B} are satisfied for the ultrafiler $a$ and the sequence $v_l$ ($l \in \NN$). 

The situation for \ref{it:3:C} is much simpler. Recall that $\bA_m|_q \bmod{\cV_{\mathrm{top}} e_0^{\mathrm{T}}}$ invariant under translations in $\cV_{\mathrm{top}}$. Hence, condition \ref{it:3:C} (in the equivalent form \eqref{eq:it-3-C-variant}) follows immediately from \eqref{eq:49:13}. 
We have shown that $a \in \tcG^q_p$. It remains to recall that $\tcG^q_p = \cG^q_p$ by Proposition \ref{prop:rigidity}. 
\end{proof}

\begin{proof}[Proof of Proposition \ref{prop:induction}]
By Proposition \ref{prop:Tk-basic}, the coefficients of the matrix $\bA_l|_p$ are polynomials functions in $l$, $x^{p}$ and $T_l|_{p}(x^{p})$. More precisely, $\bA_l|_p \bmod{\cV_{\mathrm{top}}e_0^{\mathrm{T}}}$ is a polynomial in $l$, $x^p_{\cE}$ and $T_l|_{p}(x^{p})_\cE$, where $x_\cE = (x_\mu)_{\mu \in \cE}$.  Moreover, each of the sets $\cS_1(\e), \cS_2, \cS_3(\e)$ ($\e > 0$) is preserved under the operation $(T,v) \mapsto (T+ue_0^{\mathrm{T}},v+u)$ for $u \in \cV_{\mathrm{top}}$. Hence, $\cR(\e)$ ($\e > 0$) is preserved under the operation $T \mapsto T+ue_0^{\mathrm{T}}$ for $u \in \cV_{\mathrm{top}}$. In other words, membership in $\cR(\e)$ ($\e > 0$) depends only on equivalence class modulo ${\cV_{\mathrm{top}}e_0^{\mathrm{T}}}$. Hence, for each $\e > 0$ there exists a semialgebraic set $R(\e)$ such that for each $l \in \NN$,
\begin{align}\label{eq:49:00}
 T_l|_p &\in \cR(\e) &\text{ if and only if }&& \bra{T_l|_{p}(x^{p})_\cE,l} &\in R(\e).
\end{align}

Temporarily fix $\e > 0$. Then $k^r \in \tcH^q_p[\cD] \subset \cG^q_p[\cD]$ by Lemma \ref{lem:k^r-in-H}, and hence it follows from \eqref{eq:49:00} and Lemma \ref{lem:ind-4} that
 \[
 \forall^r_n \ \bra{ T_k^n|_p(x^p)_{\cE}, k^n} = \bra{x^{n+p}_{\cE},k^n} \in R(\e).
 \]
Let $V^u_v[\cE]$ and $V^*_v[\cE]$ ($u,v \in \beta \NN$) be the varieties in $\RR^{\cE}$ defined in full analogy with $V^{u}_v$ and $V^*_{v}$, and put
\[ V := V^p_{r}[\cE] = V^*_{q}[\cE] = \alglim_{n\to r}\bra{x^{n+p}_{\cE}} \subset \RR^{\cE}.\]
It follows from Proposition \ref{prop:sep-semialg-IP} that there exists a relatively open set $U \subset \Vp_q$ and a continuous function $f \colon U \to \RR$ such that 
\begin{equation}\label{eq:49:01}
 \set{ (x,y) \in \RR^{\cE} \times \RR}{ x \in U, \ y \geq f(x) } \subset R(\e),
 \end{equation}
and for $r$-almost all $n$ we have \( x^{n+p}_{\cE} \in U.\) Moreover, since for $r$-almost all $n$ we have $a_n \in \tcH^{n+p}_p[\cE]$, it follows that
\[ \lim_{l \to a_n} T_l|_p(x^p)_{\cE} = x^{n+p}_{\cE} \qquad \text{ and } \qquad \forall^{a_n}_l\ T_l|_p(V_p^*[\cE]) = V_{n+p}^*[\cE] = V_q^*[\cE] = V.\]
In particular, $T_l|_p(x^p)_{\cE} \in V$ for $a_n$-almost all $l$, and consequently 
\[
	\forall^{a_n}_l\ T_l|_p(x^p)_{\cE} \in U.
\]
Since $a_n$ is non-principal, we also have $l \geq f(x^{n+p}_{\cE})+1$ for $a_n$-almost all $l$. Hence, bearing in mind that $f$ is continuous, it follows from \eqref{eq:49:01} that 
\begin{equation}\label{eq:49:02}
	\forall^r_n\ \forall^{a_n}_l\ \bra{ T_l|_p(x^p)_{\cE}, l} \in R(\e).
 \end{equation}
Hence, by \eqref{eq:49:00} we have
\begin{equation}\label{eq:49:03}
	\forall^r_n\ \forall^{a_n}_l\ T_l|_p \in \cR(\e).
 \end{equation}
Put $a := \lim_{n \to r} a_n$. Passing to the limit $n \to r$ in \eqref{eq:49:03} yields
\begin{equation}\label{eq:49:04}
	\forall^{a}_l\ T_l|_p \in \cR(\e).
 \end{equation}
Since $\e > 0$, it follows from Lemma \ref{lem:ind-4} that $a \in \cG^q_p[\cD]$.
\end{proof}

\subsection{Combining the ingredients}\label{ssec:rec-combined}
{\color{NavyBlue}
We are now ready to combine the ingredients discussed above to prove Theorem \ref{thm:recurrence-strong}. As alluded to before, we consider the abelian case separately.
}

\begin{proof}[Proof of Theorem \ref{thm:recurrence-strong}, case $\cD \subset \NN$]\color{LightGray}
	The maps $A_k \colon \RR^{\cD} \to \ST(\cD)$ are constant, and $A_k(x) = \Delta_k$ for each $x \in \RR^{\cD}$. This allows us to repeat the argument in the proof of Proposition \ref{lem:idempotent} almost verbatim to show that $\Vp_p$ and $\Vp_q$ are affine spaces defined over $\QQ$. (For instance, it follows from Proposition \ref{prop:Vp=G(z)} that $\Vp_p$ is an affine space, and it follows from Proposition \ref{lem:U-is-Q} that $\Vp_p$ is defined over $\QQ$.) Let $\cW$ be the vector space such that $\Vp_p = x^p + \cW$. 
	
	There exists an infinite set $E \subset \NN$ such that $\Vp_p$ is preserved under the maps $T_k^n|_p^{-1} \circ T_k^m|_p$ for all $m,n \in E$. Hence, $\cW$ is preserved by $\Delta_k^{h}$ for all $h \in E-E$, and it follows from linear algebra that $\cW$ is spanned by eigenvectors of $\Lambda$. Consequently, $\cW$ is preserved by $\Delta_l$ for all $l \in \NN$ and $\Vp_q = x^q + \cW$.
	
	At several points in the argument we encounter technical issues if some of the coordinates $x^p_i$ are rational. To deal with them, it will be convenient to define $\fp{t}^1 := 1-\fp{-t} \in (0,1]$ for $t \in \RR$, and for the sake of uniformity also put $\fp{t}^0 = \fp{t}$. Then $\fp{t}^0 = \fp{t}^1$ for $t \in \RR \setminus \ZZ$, while $\fp{t}^{\epsilon} = \epsilon$ for $t \in \{0,1\}$. For $i \in \cD$, let $\epsilon(i) = 0$ if $x^n_i \geq x^p_i$ for $p$-almost all $n$ and $\epsilon(i) = 1$ if $x^n_i < x^p_i$ for $p$-almost all $n$. Then
	\[
		T_m|_p(x^p)_i = \fp{m^{d_i} x^p_i}^{\epsilon(i)} \quad (m \in \NN,\ i \in \cD).
	\]
	
Let $r \in \beta \NN$ be such that $r+p = q$. We will show that any ultrafilter $a$ of the form $k^r + e$ with $e \in \EN$ satisfies conditions \ref{it:1:A}--\ref{it:1:C}. Since $\EN \cap \Xi \neq \emptyset$, from here we can conclude that $\cH^q_p \cap \Xi \neq \emptyset$ which is more than what is needed. Recall that for each polynomial map $f \colon \ZZ \to \RR/\ZZ$ we have $\lim_{l \to e} f(l) = f(0)$.

\begin{enumerate}[wide]
\item[\ref{it:1:A}:]
For each $i \in \cD$ we can compute that
\begin{align*}
	\lim_{l \to a} T_l|_p(x^p)_i 
	&= \lim_{n \to r} \lim_{l \to e} \fp{ (k^n + l)^{d_i} x^p_i}^{\epsilon(i)}
	\\&\overset{!}{=}  \lim_{n \to r} \fp{ k^{d_i n} x^p_i}^{\epsilon(i)} 
	= \lim_{n \to r} T_k^n|_p(x^p)_i = x^q_i;
\end{align*}
note that the labelled with the exclamation mark holds for slightly different reasons when $x^p_i$ is rational and irrational. 

\item[\ref{it:1:B}:] Because $\Vp_p$ is defined over $\QQ$ and $\cW$ splits into eigenspaces of $\Lambda$, the sequence $T_l|_p(\Vp_p) = T_l|_p(x^p) + \cW$ ($l \in \NN$) is finitely-valued. Hence, \ref{it:1:B} follows from \ref{it:1:A}. (See also Proposition \ref{prop:v_perp-discrete-general}.)

\item[\ref{it:1:C}:] Let $m \in \NN$ and note that for any $l \in \NN$, both of the maps $T_{ml}|_p$ and $T_m|_q \circ T_l|_p$ take the form $\Delta_{ml} - b$ with $b \in \ZZ^{\cD}$. Hence, it will suffice to verify that 
\[ 
	\forall^a_l \  \norm{  T_{ml}|_p(x) - \lim_{l \to a} T_{m}|_q \circ T_l|_p(x) } < 1
\]
for at least one point $x \in [0,1]^{\cD}$. In fact, we will prove a more precise statement,
\[
	\lim_{l \to a} T_{ml}|_p(x^p) = \lim_{l \to a} T_{m}|_q \circ T_l|_p(x^p) = T_{m}|_q(x^q).
\] 
This follows from a computation very similar to the one we encountered for item \ref{it:1:A}. Let $i \in \cD$; then 
\begin{align*}
	\lim_{l \to a} T_{ml}|_p(x^p)_i 
	&= \lim_{n \to r} \lim_{l \to e} \fp{ m^{d_i}(k^n + l)^{d_i} x^p_i}^{\epsilon(i)}
	\overset{!}{=}  \lim_{n \to r} \fp{ m^{d_i} k^{d_i n} x^p_i}^{\epsilon(i)}
	\\ &=   \lim_{n \to r} T_{m k^n}|_p(x^p)_i 
	= \lim_{n \to r} T_{m}|_q \circ T_k^n|_p(x^p)_i \qedhere
	= T_m|_q(x^q)_i.
\end{align*}
\end{enumerate} 
\end{proof}

We are now ready to combine the ingredients introduced so far to deal with the case when $\cD \not \subset \NN$. The argument is little more than a list of references to previously proved facts.

\begin{proof}[Proof of Theorem \ref{thm:recurrence-strong}, general case]
	We may assume that Theorem \ref{thm:recurrence-strong} has already been proved for all $\cD'$ with $\complexity(\cD') < \complexity(\cD)$ and that $\cD \not \subset \NN$. 	 By Proposition  \ref{lem:bootstrap}, the intersection $\tcH^{u}_p[\cE] \cap \Zeta^{u}_p \cap \Xi$ is nonempty  for any $u \sim p \in \KN$ with $x^{u} \in (0,1)^\cD$. Pick $r$ such that $r+p = q$. Then, for $r$-almost all $n$, we can find an ultrafilter $a_n \in \tcH^{n+p}_p[\cE] \cap \Zeta^{n+p}_p \cap \Xi$. Put $a := \lim_{n \to r} a_n$. Then, by Proposition \ref{prop:induction} and Lemma \ref{lem:Zeta-basic}, $a \in \tcH^{q}_p[\cE] \cap \Zeta^{q}_p \cap \Xi$. It follows from Corollary \ref{cor:v-add} that $\tcH^q_p[\cD] \cap \Zeta^q_p \cap \Xi \neq \emptyset$. In particular, $\tcH^q_p[\cD] \cap \Xi \neq \emptyset$, as needed.
\end{proof}  \section{Proof of the Main Theorem}\label{sec:main}
{
We are now ready to finish the proof of Theorem \ref{thm:A}. The argument is reminiscent of the derivation of Theorem \ref{thm:main-strong-torus} in the abelian case. Let us recall that main result of the previous section. It will be convenient to phrase it in a slightly more verbose way. (Recall that the relevant definitions related to ultrafilters can be found in Sec.{} \ref{ssec:prelims-ultrafilters}, and the $\times k$ maps $T_k$ were introduced in Sec.{} \ref{ssec:setup-Tk}.)
}

\begin{customthm}{\ref{thm:recurrence}${}^\prime$}\label{thm:recurrence-alpha}
	Let $p, q \in \beta \NN$ be minimal ultrafilters belonging to the same minimal left ideal and assume that $x^p, x^q \in (0,1)^{\cD}$. Then there exists an ultrafilter $a \in \beta \NN$ such that $d^*(E) > 0$ for all $E \in a$ and $\displaystyle \lim_{l \to a} T_l|_p(x^p) = x^q$ and $\displaystyle \forall^a_l \ T_l|_p(\Vp_p) = \Vp_q$.
\end{customthm}

Together with Proposition \ref{prop:sep-semialg-IP}, Theorem \ref{thm:recurrence-alpha} leads to the following result, which can be construed as an analogue of Theorem \ref{thm:main-strong-torus}. We point out that the argument proceeds along similar lines as the one in Section \ref{ssec:torus-general}, with Theorem \ref{thm:recurrence-alpha} playing the role of Corollary \ref{cor:main-ur-1}.

\begin{theorem}\label{thm:main-strong}
	Let $p \in \beta \NN$ be a minimal idempotent and let $S \subset [0,1)^\cD$ be a semialgebraic set. Suppose that $x^n \in S$ for $p$-almost all $n$ and that $x^p \in (0,1)^\cD$. Let 
\begin{equation*}
	L := \set{ l \in \NN }{ \forall^p_n \ T_l(x^n) \in S }.
\end{equation*}
Then the set $L$ has positive Banach density.
\end{theorem}

\begin{proof}
It follows directly from the relevant definitions that for any $l \in \NN$ we have
\begin{equation}\label{eq:19:00}
	 \forall^p_{m} \ T_l(x^m) = A_l|_p(x^m-x^p) + T_l|_p(x^p).
\end{equation}
By Proposition \ref{prop:Tk-basic}\eqref{it:Tk-basic:poly}, there exists a polynomial $B$ such that $A_l|_p = B(l,x^p,T_l|_p(x^p))$ for all $l \in \NN$. 
Consider the set 
\begin{equation}\label{eq:19:02}
	R := \set{(x,t) \in \RR^{\cD} \times \RR }{\forall^p_{m} \ x + B(t,x^p,x)(x^m - x^p) \in S }.
\end{equation}
Then $R$ is semialgebraic by Proposition \ref{prop:semialg-limit}. The definitions are set up so that 
\begin{equation*}
 L = \set{ l \in \NN}{ \bra{T_l|_p(x^p), l} \in R }.
\end{equation*}

The fact that $p$ is idempotent implies that
\begin{equation}\label{eq:19:01}
	\forall^p_{n} \ \forall^p_{m} \
	x^{n + m} = T_k^n(x^m) = A_k^n|_p(x^m - x^p) + x^{n+p} \in S,
\end{equation}
and hence we also have
\begin{equation}\label{eq:19:11}
	\forall^{p}_{n} \ \bra{x^{n+p}, k^n} \in R.
\end{equation}
By Proposition \ref{prop:sep-semialg-IP}, there exists a relatively open subset $U \subset \Vp_p$ and a continuous map $f \colon U \to \RR$ such that $x_n \in U$ for $p$-almost all $n$ and
\[
	\set{(x,y) \in \RR^{\cD} \times \RR}{x \in U,\ y \geq f(x)} \subset R. 
\]
Pick $n$ with $x_n \in U$ and $\Vp_{n+p} = \Vp_p$. By Theorem \ref{thm:recurrence-alpha}, there exists an ultrafilter $a \in \Xi$ such that $\lim_{l \to a} T_l|_p(x^p) = x^{n+p}$ and $\forall^a_l \ T_l|_p(\Vp_p) = \Vp_p$. Then for $a$-almost all $n$ we have $T_l|_p(x^p) \in \Vp_p$, and since $U$ is relatively open also $T_l|_p(x^p) \in U$. Accordingly, for $a$-almost all $l$ we have $l > f(x^{n+p}) + 1 > f(T_l|_p(x^p))$. Consequently, $\bra{T_l|_p(x^p), l} \in R$ for $a$-almost all $l$. Thus, $L \in a$ and in particular $d^*(L) > 0$.
\end{proof}

{
Finally, we derive Theorem \ref{thm:A} from Theorem \ref{thm:main-strong}. The argument is based on Leibman's Theorem \ref{thm:Leibman} and is analogous to how the abelian case of Theorem \ref{thm:A} is derived from Theorem \ref{thm:main-strong-torus}. On a more technical side, this is the place where we deal with complications corresponding to points on the boundary of the cube $[0,1]^{\cD}$, which we have until now mostly avoided.

We also point out that this is the only place where we use Proposition \ref{prop:gp-lim-is-gp}. Without it, we could have derived a version of Theorem \ref{thm:A} with a weaker notion of largeness, namely positive Banach density in place of $\IP^*_+$. This weaker version is still sufficient to prove Theorem \ref{thm:B}.
}

\begin{proof}[Proof of Theorem \ref{thm:A}]By Leibman's Theorem \ref{thm:Leibman}, for suitably chosen index set $\cD \subset \cB$ satisfying \eqref{eq:def-of-D} there exists $\alpha \in \RR^{\cD}$ and a piecewise polynomial set $S \subset [0,1)^\cD$ such that
	\begin{equation}\label{eq:976:1}
		E = \set{ m \in \NN}{ \fp{v^\alpha(m)} \in S}.
	\end{equation}
	(Recall that the grading $(d_\mu)_{\mu \in \cB}$ was fixed in Section \ref{ssec:setup-basic}.) Put $x^0 := \fp{v^\alpha(1)}$ so it follows from Proposition \ref{prop:Tk-basic}\eqref{it:Tk-basic:times-k} that 
	\begin{equation}\label{eq:976:2}
		\fp{v^\alpha(m)} = T_m(x^0) \qquad (m \in \NN).
	\end{equation}
	Let $q \in \beta \NN$ be a minimal idempotent. Since $k^n \in E$ for all $n \in \NN$, in particular
	\begin{equation}\label{eq:976:3}
		\forall^q_n \ x^n = T_k^n(x^0) = \fp{v^{\alpha}(k^n)} \in S.
	\end{equation}
	
	Suppose first that $x^q \in (0,1)^{\cD}$. Then by Theorem \ref{thm:main-strong} the set $L \subset \NN$ given by
	\begin{equation}\label{eq:976:5}
		L := \set{ l \in \NN}{ \forall^q_n \ T_l(x^n)  \in S}
	\end{equation}
	has positive Banach density. Moreover, $L$ is a generalised polynomial set by Proposition \ref{prop:gp-lim-is-gp} and Proposition \ref{prop:Tk-basic}\eqref{it:Tk-basic:gen-poly}. Hence, it follows from Bergelson--Leibman Theorem \ref{thm:Bergelson-Leibman} that $L$ is $\IP^*_+$. It remains to observe that if $l \in L$ then $lk^n \in E$ for $q$-almost all $n$, and in particular for infinitely many $n \in \NN_0$.
	
	Consider now the general case when $x^q$ may lie on the boundary of $[0,1]^{\cD}$. In order to reduce the problem to the previous case, we will replace the index set $\cD$ and the vector $\alpha \in \RR^{\cD \cap \NN}$ with another index set $\cC$ and vector $\beta \in \RR^{\cC \cap \NN}$ which still allow us to represent the set $E$ and which have the property that 
	\begin{equation}\label{eq:976:4}
\lim_{n \to q} \bra{ \fp{v^\beta_{\nu}(k^n) }}_{\nu \in \cC} \in (0,1)^{\cC}.
	\end{equation}
Let us say that the pair $(\cC,\beta)$ where $\cC \subset \cB$ satisfies the analogue of the downward closure property \eqref{eq:def-of-D} and $\beta \in \RR^{\cC \cap \NN}$ is \emph{more expressive} than the pair $(\cD,\alpha)$ if there exists a generalised polynomial map $h \colon \RR^{\cC} \to \RR^{\cD}$ such that 
	\begin{equation}\label{eq:976:6}
		\bra{ \fp{v^\alpha_{\mu}(m) }}_{\mu \in \cD} = h\bra{ \fp{v^\beta_{\nu}(m) }}_{\nu \in \cC} 
\qquad (m \in \NN). 
	\end{equation}
	If \eqref{eq:976:6} holds then the set $R = h^{-1}(S)$ is semialgebraic and in analogy with \eqref{eq:976:1} we have
	\begin{equation}\label{eq:976:1b}
		E = \set{ m \in \NN}{ \bra{ \fp{v^\beta_{\nu}(m) }}_{\nu \in \cC}  \in R}.
	\end{equation}	
	
Hence, we can freely replace $(\cD,\alpha)$ with any pair $(\cC,\beta)$ that is more expressive. As the first application, we show that we may assume that for each index $\mu \in \cD$ and for each $i \in \cD \cap \NN$, $i$ appears at most once in the expansion of $\mu$. In fact, we will ensure marginally more, namely that there exists a (partial) order $\sqsubseteq$ on $\cD \cap \NN$ such that any $\mu \in \cD$ is \emph{compatible with $\sqsubseteq$} in the following sense: if $\mu$ is written in the form $\mu = i \fpp{\lambda_1} \fpp{\lambda_2} \dots \fpp{\lambda_r}$ and $j \in \cD \cap \NN$ appears in one of $\lambda_1,\dots,\lambda_r$ then $i \sqsupset j$.
 Enumerate $\cD \cap \NN = \{i_1,\dots,i_N\}$ and let $t \in \NN$ be the largest number of times that any $i \in \NN$ appears in an index $\mu \in \cD$. Let $\cC \cap \NN$ consist of $t$ copies of $\cD \cap \NN$, that is, 
\[ 
\cC \cap \NN = \{i_1,\dots,i_N,i_1',\dots,i_N',\dots, \dots, i_1^{(t-1)}, \dots i_N^{(t-1)}\},
\]
where $d_{i_j^{(r)}} = d_{i_j}$ and $\beta_{i_j^{(r)}} = \alpha_{i_j}$ for all $1 \leq j \leq N$ and $0 \leq r < t$. Let $\sqsubseteq$ be the partial order on $\cC$ given by $i_j^{(r)} \sqsubseteq i_{j'}^{(r')}$ if and only if $r \leq r'$ ($0 \leq r,r, < t$ and $1 \leq j,j' \leq N$). For $\mu \in \cB$ whose representation contains only integers from $\cC \cap \NN$ let $\bar \mu$ denote the result of replacing every instance of $i_j^{(r)}$ with $i_j$. (Hence, for instance, if $\mu = i_1 \fpp{i_2' \fpp{i_3}\fpp{i_1''} }$ then $\bar{ \mu } = i_1 \fpp{i_2 \fpp{i_3}\fpp{i_1} }$.) Finally, let $\cC$ consist of those $\mu \in \cB$ which are compatible with $\sqsubseteq$ in the sense described above. It is routine to check that thus defined pair $(\cC,\beta)$ is more expressive than $(\cD, \alpha)$ (one can take $h$ to be the natural projection map $\RR^{\cC} \to \RR^{\cD}$, specified by $e_{\mu} \mapsto e_{\bar\mu}$ for each $\mu$) and that $\cC$ satisfies the analogue of \eqref{eq:def-of-D}. It remains to replace $(\cD,\alpha)$ with $(\cC,\beta)$.
	
In order to ensure \eqref{eq:976:4}, we perform the same construction as above with $t = 2$, except we set $\beta_{i_j'} = \alpha_{i_j} - \beta_{i_j}$, where $ \beta_{i_j}$ remain to be determined ($1 \leq j \leq N$) and let $\cC$ consist of all $\mu \in \cB$ with $\bar \mu \in \cD$. One can show by structural induction on $\mu \in \cD$ that $\fp{v^{\alpha}_\mu(m)}$ can be represented as a generalised polynomial expression in $\fp{v^{\beta}_\nu(m)}$ ($\nu \in \cC$). (For instance, $\fp{ v^{\alpha}_{i}(m) } = \fp{ \fp{v^{\beta}_{i}(m) } + \fp{v^{\beta}_{i'}(m) } }$ for $i \in \cC \cap \NN$.) Another structural induction argument shows that there exists a choice of $\beta$ such that \eqref{eq:976:4} holds. (In fact, one can pick $\beta_{i} = b_{i}/Q$, where $Q$ is a large prime and $b_{i}$ are chosen uniformly at random from $\{0,1,\dots,Q-1\}$; as $Q \to \infty$, the probability that \eqref{eq:976:4} holds tends to $1$ as $Q \to \infty$.)
\end{proof}

\begin{remark}
	The above argument shows that Theorem \ref{thm:A} remains true if the assumption that $k^n \in E$ for all $n \in \NN_0$ is replaced with the weaker assumption that the set of $n$ such that $k^n \in E$ is central. However, we are not aware of examples of sets $E$ satisfying the latter but not the former.
\end{remark}

\appendix 
\addcontentsline{toc}{chapter}{APPENDIX}

\section{Limits of semialgebraic sets}\label{ap:sg-lim}

We include here the material concerning the limits of semialgebraic sets which complements the discussion in Section \ref{ssec:prelims-semialgebraic}. In particular, we prove Proposition \ref{prop:semialg-limit}.

For a sequence of sets $S_n \subset \RR^d$, let us say that $\lim_{n\to\infty} S_n = S$ if for each $x \in \RR^d$ there exists an integer $n_0$ such that for all $n \geq n_0$ we have the equivalence $x \in S \iff x \in S_n$. In other words, $\lim_{n\to\infty} S_n = S$ if and only if $S = \lim_{n \to p} S_n$ for all non-principal ultrafilters $p \in \beta \NN \setminus \NN$.

\begin{example}\label{ex:set-limit-alg-bad}
\begin{enumerate}[wide]
\item For any open set $U \subset \RR^d$ there exists a sequence of open semialgebraic sets $S_n$ ($n \in \NN$) such that $\lim_{n\to\infty} S_n = U$. 
\item For any open and convex set $U \subset \RR^d$ there exists a sequence of open basic semialgebraic sets $S_n$ ($n \in \NN$) described only by inequalities of degree $1$ such that $\lim_{n\to\infty} S_n = U$. 
\end{enumerate}
\end{example}
\begin{proof}
\begin{enumerate}[wide]
\item For $n \in \NN$, let $S_n^{(0)}$ be the union of all closed cubes of the form $x + [0,2^{-n}]^d$ with $x \in 2^{-n}\ZZ^d$ and $\norm{x}_2 \leq n$, contained in $S$. Put also $S_n = \inter S_n^{(0)}$. Then $U = \bigcup_{n=1}^\infty S_n$ and $S_n \subset S_{n+1}$ for all $n$. Hence, $\lim_{n \to \infty} S_n = U$. It remains to notice that each of the sets $S_n$ is semialgebraic, which is an easy consequence of the observations that the family of semialgebraic sets is closed under union, and that cubes are semialgebraic.
  
\item Note that the interior of any polytope in $\RR^d$ is the intersection of a finite number of open half-spaces, and hence a basic semialgebraic set given by linear inequalities. Let $S_n^{(0)}$ ($n \in \NN$) be a sequence of convex polytopes with  $U = \bigcup_{n=1}^\infty S_n^{(0)}$ and $S_{n}^{(0)} \subset S_{n+1}^{(0)}$ for all $n$, and put $S_n = \inter S_n^{(0)}$. Then $U = \lim_{n\to \infty} S_n$. \qedhere
\end{enumerate}
\end{proof}

Recall that the statement of Proposition \ref{prop:semialg-limit} and the definition of complexity of a semialgebraic set appear in Section \ref{ssec:prelims-semialgebraic}.

\begin{proof}[Proof of Proposition \ref{prop:semialg-limit}]
Fix $p \in \beta \NN$. We will show that for any algebraic variety $V \subset \RR^d$ the following holds:
\begin{equation}\label{eq:claim:43}
	\textit{If $S_n \subset V$ $(n \in \NN)$ and $\sup_{n \in \NN} \complexity(S_n) < \infty$, then $\lim_{n \to p} S_n$ is semialgebraic.}
\end{equation}
If $V = \emptyset$ then \eqref{eq:claim:43} is vacuously true and we are ultimately interested in the case where $V = \RR^d$. 
Proceeding by induction on $V$, we may assume that \eqref{eq:claim:43} holds for any proper subvariety $V' \subsetneq V$ (recall that this mode of reasoning is valid since any strictly descending sequence of algebraic varieties has finite length). 

Let $S_n \subset V$ ($n \in \NN$) be a sequence of semialgebraic sets with bounded complexity, let
$ \displaystyle C := \sup_{n \in \NN} \complexity(S_n) < \infty$ and $\displaystyle S := \lim_{n \to p} S_n.$ Decomposing $S_n$ into basic components and using Lemma \ref{lem:set-limit-basic}, we may assume without loss of generality that that $S_n$ is a basic algebraic set for all $n \in \NN$. By the definition of complexity, there exist sets $F_n,G_n \subset \RR[\mathbf{x}_1,\dots,\mathbf{x}_d]$ such that 
	\begin{align*}
	S_n &= \fS(F_n,G_n) \text{ and } \sum_{f \in F_n} \deg(f) + \sum_{g \in G_n} \deg(g) \leq C 
	\text{ for all } n \in \NN.
	\end{align*}	
	Let $V_n := \fV(F_n)$; replacing $V_n$ with $V_n \cap V$ we may assume that $V_n \subset V$ for all $n \in \NN$. 
	Let $\cP \subset \RR[\mathbf{x}_1,\dots,\mathbf{x}_d]$ denote the vector space consisting of all polynomials with degree $\leq C$. Let $\cI := \cP \cap \fI(V)$ and let $\cR < \cP$ be a complement of $\cI$, meaning that $\cI \cap \cR = \{0\}$ and $\cI + \cR = \cP$. Pick also a norm $\norm{\cdot}$ on $\cR$. Note that any $f \in \cP \setminus \cI$ has a unique decomposition as
	\begin{equation}\label{eq:499:1}
		f = \lambda \bar f + h, \text{ where $\lambda \in \RR_{>0}$, $\bar f \in \cR$, $h \in \cI$ and $\norm{\bar f} = 1$.}
	\end{equation}
	The argument now splits into two cases, depending on how often the inclusion $V_n \subset V$ is strict.
	
Suppose first that for $p$-almost all $n$ we have strict containment $V_n \subsetneq V$, whence we can pick $f_n \in F_n \setminus \cI$. Let $f = \lambda_n \bar f_n + h_n$ be the decomposition of $f_n$ as in \eqref{eq:499:1} ($n \in \NN$). Since the unit sphere in $\cR$ is compact and disjoint from $\cI$, we may define 
\begin{align*}
\bar f := \lim_{n \to p } \bar f_n  \text{ and } U := \fV(\{ \bar f\}) \cap V \subsetneq V.
\end{align*}
If $x \in S$ then $x \in V$ and $x \in S_n$ for $p$-almost all $n$, so
\[
	\bar f(x) = \lim_{n \to p} \bar f_n(x) = \lim_{n \to p} \lambda_n^{-1}\bra{f_n(x) - h_n(x)} = 0.
\] 
It follows that $S \subset U \subsetneq V$ and $S$ is semialgebraic by the inductive assumption. 

Suppose next that $V_n = V$ for $p$-almost all $n$, so we can assume that $F_n = F$ for all $n \in \NN$, where $F$ is a finite family of polynomials with $\fV(F) = V$. Repeating entries if necessary, we may assume that all of the sets $G_n$ $(n \in \NN$) have the same size $s \leq C$, whence we can enumerate $G_n = \{g_n^1, g_n^2, \dots, g_n^s\}$. Let $g^j_n = \lambda^j_n \bar g^j_n + h_n^j$ be the decompositions of $g^j_n$ as in \eqref{eq:499:1} ($n \in \NN, 1 \leq j \leq s$). Using compactness of the unit sphere in $\cR$ again, we can define 
\[ \displaystyle \bar g{}^j := \lim_{n\to p} \bar g{}^j_n \ (1 \leq j \leq s) \text{ and } G := \{\bar g^1, \bar g^2,\dots,\bar g^s\}.\]
Consider the corresponding semialgebraic set
\[
	R := \fS(F,G) = \set{x \in V}{ \bar g^j(x) > 0 \text{ for all } 1 \leq j \leq s}
\]
as well as the ``boundary'' set
\[
	U := \fV\bra{F \cup \left\{ {\textstyle\prod_{j=1}^s \bar g^j } \right\}} = \set{x \in V}{ \bar g^j(x) = 0 \text{ for some } 1 \leq j \leq s}.
\]
(Note that $U$ contains the topological boundary of $R$ as a subset of $V$.)

For any $1 \leq j \leq s$ and $x \in V$, if $g^j_n(x) > 0$ for $p$-almost all $n$ then also $\bar g^j(x) \geq 0$; conversely, if $\bar g^j(x) > 0$ then $g^j_n(x) > 0$ for $p$-almost all $n$. It follows that $R \subset S \subset R \cup U$. Since $U \subsetneq V$, it follows from the inductive assumption that $S \cap U$ is semialgebraic. Consequently, $S = R \cup (S \cap U)$ is semialgebraic.
\end{proof}

\section{\checkmark\ Limits of generalised polynomials}\label{ssec:lemmas-gp-lim}\mbox{}
{
We now consider the question of closure of the set of generalised polynomials under pointwise limits, completing the discussion from Section \ref{ssec:prelims-gen-poly}. To begin with, consider the following motivating example concerning ordinary polynomials.
}

\begin{example}\label{ex:poly-limits}
\begin{enumerate}[wide]
\item If $h \colon \RR \to \RR$ is continuous then there exists a sequence of polynomials $f_i \colon \RR \to \RR$ ($i \in \NN$) such that $\lim_{i \to \infty} f_i(x) = h(x)$ for each $x \in \RR$. In particular, any sequence $a \colon \ZZ \to \RR$ is the pointwise limit of polynomial sequences. This is a simple consequence of the Stone--Weierstrass theorem.
\item If $f_i \colon \RR \to \RR$ ($i \in \NN$) is a sequence of polynomials such that $\sup_{i} \deg f_i < \infty$ and the pointwise limit $h(x) := \lim_{i \to \infty} f_i(x)$ exists for infinitely many $x \in \RR$ then the same limit exists for all $x \in \RR$ and $h$ is a polynomial map $\RR \to \RR$. This is a simple consequence of Lagrange interpolation.
\end{enumerate}
\end{example}

{
In analogy, one could hope that the limit of any sequence of generalised polynomials with bounded complexity is a generalised polynomial (for suitably defined notion of complexity). Unfortunately, the following example shows that even limits of extremely simple generalised polynomials need not be generalised polynomials. Recall that the Iverson bracket convention was introduced in Section \ref{ssec:prelims-notation}.
}

\begin{example}\label{ex:gp-limit-not-gp}
\begin{enumerate}[wide]
\item	Consider $g_1 \colon \ZZ \to \RR$ given by 
	\begin{equation*}
	g_1(n) := 1+\lim_{i \to \infty} \ip{n/i} = \braif{n \geq 0}.
	\end{equation*}
	Then $g_1$ is not a generalised polynomial on $\ZZ$.
\item	Consider $g_2 \colon \NN^2 \to \RR$ given by
	\begin{equation*}
	g_2(n,m) :=  1+\lim_{i \to \infty} \ip{(n-m)/i} = \braif{n \geq m}.
	\end{equation*}
 	Then $g_2$ is not a generalised polynomial on $\NN^2$.
\end{enumerate}
\end{example}

{
Surprisingly, the situation is radically different for generalised polynomials whose domain is $\NN$, as stated in Proposition \ref{prop:gp-lim-is-gp}.
The remainder of this subsection is devoted to the proof of the aforementioned proposition. We begin with a generalisation of Lemma \ref{lem:gp-a<g<b} for unbounded generalised polynomials on $\NN$.
}

\begin{lemma}\label{prop:gp-g>0}
	Let $g \colon \NN \to \RR$ be a generalised polynomial. Then the map $\NN \to \RR$ given by $n \mapsto \braif{g(n) \geq 0}$ is a generalised polynomial.
\end{lemma}
\begin{proof}
	Expand $g(n)$ as a polynomial in $n$ with bounded generalised polynomial coefficients:
	\begin{equation}\label{eq:825:1}
	\textstyle
		g(n) = \sum_{i=0}^d h_i(n) n^i.
	\end{equation}
	Existence of such expansion can be proved using structural induction on $g$ (cf.\ \cite[Sec.\ 10]{Leibman-2012}). 
	
	We proceed by induction on $d$. If $d = 0$ then $g = h_0$ is bounded and the claim follows from Lemma \ref{lem:gp-a<g<b}. Suppose now that $d > 0$ and let $0 < C < \infty$ be given by
	\begin{equation}\label{eq:825:2}
		C := \sup_{n \in \NN} \abs{ \frac{g(n) - h_d(n)n^d}{n^d} } 
		= \sup_{n \in \NN} \abs{h_{d-1}(n) + h_{d-2}(n)/n + \dots + h_0(n)/n^{d-1}}.
	\end{equation}
	Let $q \colon \NN \to \{0,1\} \subset \RR$ be the generalised polynomial given by
	\begin{equation}\label{eq:825:3}
		q(n) := \braif{ \abs{ h_d(n) } < C/n } = \braif{ -C < n h_d(n)  < C }.
	\end{equation}
	If $n \in \NN$ and $q(n) = 0$ then the term $h_d(n) n^d$ exceeds in absolute value the sum of all the other terms in the expansion of $g$ in \eqref{eq:825:1}, whence $\braif{g(n) \geq 0 } = \braif{h_d(n) \geq 0}$. Define further 
	\begin{align}\label{eq:825:4}
	h_{d-1}'(n) &= h_{d-1}(n)+2C\bra{\fp{ \frac{n h_d(n)}{2C} + \frac{1}{2}}-\frac{1}{2}},\\
\label{eq:825:5}\textstyle 	g'(n) &= h_{d-1}'(n) n^{d-1} + \sum_{i=0}^{d-2} h_i(n) n^i.
	\end{align} 
	By direct inspection, $h_{d-1}' \colon \NN \to \RR$ is a bounded generalised polynomial. If $n \in \NN$ and $q(n) = 1$ then $n h_d(n) + h_{d-1}(n) = h_{d-1}'(n)$ and consequently $g(n) = g'(n)$. 
	
	Combining the above observations we conclude for each $n \in \NN$ that
\begin{equation}\label{eq:825:6}
	\braif{g(n) \geq 0} = (1-q(n)) \braif{h_d(n) \geq 0} + q(n) \braif{g'(n) \geq 0}.
\end{equation}
By the inductive assumption, the maps $n \mapsto \braif{h_d(n) \geq 0}$ and $n \mapsto \braif{g'(n) \geq 0}$ are generalised polynomials on $\NN$, hence so is $n \mapsto \braif{g(n) \geq 0}$.
\end{proof} 

We next prove a variant of Proposition \ref{prop:gp-lim-is-gp} for generalised polynomials of particularly simple shape.

\begin{lemma}\label{lem:gp-lin-comb>0}
	Let $p \in \beta \NN$, let $x_i \in \RR^d$ $(i \in \NN)$ and let $g \colon \NN \to \RR^d$ be generalised polynomial. Then the map $\NN \to \RR$ given by 
	\(n \mapsto \braif{ \forall^p_i \ \bsp{g(n),x_i} \geq 0}\)
	is a generalised polynomial.
\end{lemma}
\begin{proof}
	Rescaling the vectors $x_i$ if necessary, we may assume without loss of generality that $\norm{x_i} \in \{0,1\}$ for each $i \in \NN$. 
	Put $x := \lim_{i \to p} x_i$ and $\Delta x_i := x_i - x$ ($i \in \NN$). 
If $n \in \NN$ is such that $\bsp{g(n),x} \neq 0$ then 
	\begin{equation*}
	\braif{ \forall^p_i \ \bsp{g(n),x_i} \geq 0} = \braif{ \bsp{g(n),x} \geq 0}.
	\end{equation*}
	Consider next the case when $\bsp{g(n),x} = 0$. Let $\cW = \RR x$ be the vector space spanned by $x$ and let $\pi \colon \RR^{d} \to \cW^{\perp}$ be the orthogonal projection. Further, let $x_i'$ be the sequence obtained from $\pi(\Delta x_i)$ by normalisation, that is, $x_i' = \pi(\Delta x_i)/\norm{\pi(\Delta x_i)}$ if $\Delta x_i \not \in \cW$ and $x_i' = 0$ if $\Delta x_i \in \cW$. Put also $x' = \lim_{i \to p} x_i'$. If $n \in \NN$ and $\bsp{g(n),x} = 0$ then
	\begin{equation*}
	\braif{ \forall^p_i \ \bsp{g(n),x_i} \geq 0} = \braif{\forall^p_i \ \bsp{g(n),\Delta x_i} \geq 0} = \braif{\forall^p_i \ \bsp{g(n), x_i'} \geq 0}.
	\end{equation*}
Combining the two cases considered above, for any $n \in \NN$ we conclude that
\begin{equation*}
	\braif{ \forall^p_i \ \bsp{g(n),x_i} \geq 0} =
	h(n) + \braif{ g(n) \perp \cW} \cdot \braif{ \forall^p_i \ \bsp{g(n),x_i'} \geq 0},
\end{equation*}
where $h(n) :=  \braif{ \bsp{g(n),x} > 0}$ is a generalised polynomial on $\NN$ by Lemma \ref{prop:gp-g>0}. Note that $x_i' \perp x$ for all $i \in \NN$, and that if $x = 0$ then the construction becomes trivial: $\cW = \{0\}$, $\pi = \mathrm{id}$, $x_i' = x_i$ ($i \in \NN$).

We next iterate the above construction. Suppose that for some $t \in \NN_0$ we have constructed a vector space $\cW_t$, a normalised sequence $x_i^{(t)} \in \cW_t^{\perp}$ ($i \in \NN$) with $x^{(t)} := \lim_{i \to p} x_i^{(t)}$ and a generalised polynomial $h_t\colon \NN \to \RR$ such that 
\begin{align}\label{eq:53:20}
	\braif{ \forall^p_i \ \bsp{g(n),x_i} \geq 0} =
	h_t(n) + \braif{ g(n) \perp \cW_t} \cdot \braif{ \forall^p_i \ \bsp{g(n),x_i^{(t)}} \geq 0}.
\end{align}
The initial steps of the construction are given by $\cW_0 = \{0\}$, $x^{(0)}_i = x_i$ for $i \in \NN$ and $h_0 = 0$, and $\cW_1 = \cW$, $x^{(1)}_i = x_i'$ for $i \in \NN$ and $h_0 = h$. 
For general $t$, let $\cW_{t+1} := \cW_{t} + \RR x^{(t)}$, let $\pi_t \colon \RR^{d} \to \cW_{t+1}^\perp$ be the orthogonal projection and let $x^{(t+1)}_i$ be the normalisation of $\pi\bra{\Delta x^{(t)}}$ for $i \in \NN$. Then
\begin{align*}
	\braif{ \forall^p_i \ \bsp{g(n),x_i} \geq 0} & = h_{t+1}(n) + \braif{ g(n) \perp \cW_{t+1}} \cdot \braif{\forall^p_i \ \bsp{g(n),x_i^{(t+1)}} \geq 0}, \\ \text{ where } 
	h_{t+1}(n) &:= h_t(n) + \braif{ g(n) \perp \cW_t} \cdot  \braif{ \bsp{h(n),x^{(t)}} > 0}.
\end{align*}
Hence, we obtain the analogue of \eqref{eq:53:20} with $t+1$ in place of $t$. 

The construction guarantees that $x^{(t)} \in \cW_t^{\perp} \cap \cW_{t+1}$ and $\cW_t \subset \cW_{t+1}$ for all $t \in \NN_0$. In particular, there are at most $d$ values of $t$ such that $x^{(t)} \neq 0$. Pick $s \in \NN$ such that $x^{(s)} = 0$. It follows from the definition of $x^{(s)}$ as a limit that $x^{(s)}_i = 0$ for $p$-almost all $i$. Substituting this into \eqref{eq:53:20} we obtain
\begin{align}\label{eq:53:20-a}
	\braif{ \forall^p_i \ \bsp{g(n),x_i} \geq 0} =
	h_s(n) + \braif{ g(n) \perp \cW_s },
\end{align}
which finishes the argument since the expression on the right hand side of \eqref{eq:53:20-a} is a generalised polynomial on $\NN$.
\end{proof}

{
	To state the next lemma, note that the polynomial ring $\RR[\mathbf{x}_1,\dots,\mathbf{x}_d]$ is a real vector space and hence it makes sense to speak of generalised polynomial maps $\NN \to \RR[\mathbf{x}_1,\dots,\mathbf{x}_d]$; these are the maps of the form 
	\(
		n \mapsto \sum_{j=1}^s g_j(n) f_j,
	\)
	where $g_j \colon \NN \to \RR$ are generalised polynomials and $f_j \in \RR[\mathbf{x}_1,\dots,\mathbf{x}_d]$ are polynomials. We might call a map of this form a generalised polynomial family of polynomials. Our next lemma can be viewed as an analogue of the elementary fact that the Taylor expansion of a polynomial map at any point is a polynomial map whose coefficients depend on the base point in a polynomial manner.
}

\begin{lemma}\label{lem:gp-derivative}
	Let $p \in \beta \NN$, let $x_i \in \RR^d$ $(i \in \NN)$ be a bounded sequence and let $h \colon \RR^d \times \NN \to \RR$ be a generalised polynomial. Put $x := \lim_{i \to p} x_i$ and $\Delta x_i := x_i -x$. Then there exists a generalised polynomial map $g \colon \NN \to \RR[\mathbf{x}_1,\dots,\mathbf{x}_d]$ (dependent on $p$, $(x_i)_{i=1}^\infty$ and $h$) such that
\begin{equation}\label{eq:854:0}
	\forall_n \ \forall^p_i \ h(x_i,n) = g(n)\bra{ \Delta x_i}.
\end{equation}	
\end{lemma}
\begin{proof}

	We proceed by structural induction on $h$. If $h$ is a polynomial then $ h(x_i,n) = h(x + \Delta x_i, n)$ is a polynomial in $\Delta x_i$, $x$ and $n$, which immediately implies the claim. Moreover, if the claim holds for $h'$ and $h''$ then it also holds for $h' + h''$ and $h' \cdot h''$ (one may take $g = g' + g''$ and $g = g' \cdot g''$ respectively).
	
	Suppose next that $h = \fp{ h' }$ and that the claim has already been proved for $h'$. Then by the inductive assumption, there exists a generalised polynomial family of polynomials $g' \colon \NN \to \RR[\mathbf{x}_1,\dots,\mathbf{x}_d]$ such that 
\begin{equation*}
	 \forall_n \ \forall^p_i \ h(x_i,n) = \fp{ h'(x_i,n) } = \fp{ g'(n)(\Delta x_i)}.
\end{equation*}
 We may write $g'(n)(\Delta x_i)$ as a linear combination of monomials in $\Delta x_i$ with generalised polynomials in $n$ as coefficients,
\begin{align*}
	g'(n)(\Delta x_i) &= \sum_{\alpha \in \NN_0^d} f_{\alpha}(n) \Delta x^{\alpha} 
	& \text{where}&& 
	y^\alpha & = \prod_{j=1}^d y_j^{\alpha_j},\ (y \in \RR^d, \alpha \in \NN_0^d).
\end{align*}
Let $g_0'(n) := f_{(0,\dots,0)}(n)$ denote the constant part of $g'$ and let $g_1'(n) := g'(n) - g_0'(n)$ denote the part vanishing at $(0,\dots,0)$ ($n \in \NN$). Using the fact that the fractional part map $\RR \to [0,1)$, $t \mapsto \fp{t}$ is continuous away from $\ZZ$ and has a jump discontinuity at each point of $\ZZ$, we may compute that
\begin{align*}
	\forall_n \ \forall^p_i \ h(x_i,t) &= \fp{ g'(n)(\Delta x_i)} 
	= \fp{ g'_0(n) +  g_1'(n)(\Delta x_i)} 
	\\ &= \fp{ g'_0(n)} +  g_1'(n)(\Delta x_i) + 
		\braif{ g'_0(n) \in \ZZ} \cdot \braif{\forall^p_j \ g_1'(n)(\Delta x_j) < 0 }.
\end{align*}
 It follows from Lemma \ref{lem:gp-lin-comb>0} that the map 
 \begin{equation*}
	n \mapsto \braif{\forall^p_j \ g_1'(n)(\Delta x_j) < 0 }
\end{equation*}	
is a generalised polynomial on $\NN$. 
Hence, the claim also holds for $h$. Thus, by structural induction, the claim holds for all generalised polynomials.
\end{proof}

With Lemmas \ref{lem:gp-lin-comb>0} and \ref{lem:gp-derivative} in hand, we are ready to finish the proof of the main result of this section. 

\begin{proof}[Proof of Proposition \ref{prop:gp-lim-is-gp}]
	We proceed by structural induction on $h$. If $h$ is a polynomial then $g$ is also a polynomial, and in particular the claim holds. Moreover, if the claim holds for $h'$ and $h''$ then it also holds for $h = h' + h''$ or $h' \cdot h''$ ($g = g' + g''$ or $g = g' \cdot g''$ respectively). It remains to consider the case when $h = \fp{h'}$ and the claim holds for $h'$. 
	
	Put $x := \lim_{i \to p} x_i$ and $\Delta x_i := x_i - x$. Let $f \colon \NN \to \RR[\bb x_1, \dots, \bb x_n]$ be a generalised polynomial such that	
	\begin{equation*}
	\forall_n \ \forall^p_i \ h'(x_i,n) = f(n)(\Delta x_i),
	\end{equation*}
	whose existence is guaranteed by Lemma \ref{lem:gp-derivative}. Arguing among similar lines as in the proof of Lemma \ref{lem:gp-derivative}, we decompose $f = f_0 + f_1$ where $f_0(n)$ is constant and $f_1(n)(0) = 0$ for all $n \in \NN$. Using piecewise continuity of the fractional part map, for each $n \in \NN$ we compute that
	\begin{align*}
		g(n) &= \lim_{i \to p} h(x_i,n) = \lim_{i \to p} \fp{ h'(x_i,n) }		
		= \lim_{i \to p} \fp{ f_0(n) + f_1(n)(\Delta x_i) }
		\\ &=  \fp{ f_0(n)} 
			+ \braif{ f_0(n) \in \ZZ} \cdot \braif{\forall^p_i \ f_1(n)(\Delta x_i) < 0}.
	\end{align*} 
	The final expression is a generalised polynomial on $\NN$ by Lemma \ref{lem:gp-lin-comb>0}, so $g$ is a generalised polynomial.
\end{proof}

\bibliographystyle{alphaabbr}
\bibliography{bibliography}

\begin{thebibliography}{BCR98}

\bibitem[AS03]{AlloucheShallit-book}
J.-P. Allouche and J.~Shallit.
\newblock {\em Automatic sequences}.
\newblock Cambridge University Press, Cambridge, 2003.
\newblock Theory, applications, generalizations.

\bibitem[BCR98]{BochnakCosteRoy-book}
J.~Bochnak, M.~Coste, and M.-F. Roy.
\newblock {\em Real algebraic geometry}, volume~36 of {\em Ergebnisse der
  Mathematik und ihrer Grenzgebiete (3) [Results in Mathematics and Related
  Areas (3)]}.
\newblock Springer-Verlag, Berlin, 1998.
\newblock Translated from the 1987 French original, Revised by the authors.

\bibitem[Ber03]{Bergelson-2003}
V.~Bergelson.
\newblock Minimal idempotents and ergodic {R}amsey theory.
\newblock In {\em Topics in dynamics and ergodic theory}, volume 310 of {\em
  London Math. Soc. Lecture Note Ser.}, pages 8--39. Cambridge Univ. Press,
  Cambridge, 2003.

\bibitem[Ber10]{Bergelson-2010}
V.~Bergelson.
\newblock Ultrafilters, {IP} sets, dynamics, and combinatorial number theory.
\newblock In {\em Ultrafilters across mathematics}, volume 530 of {\em Contemp.
  Math.}, pages 23--47. Amer. Math. Soc., Providence, RI, 2010.

\bibitem[BK18]{ByszewskiKonieczny-2018}
J.~Byszewski and J.~Konieczny.
\newblock Sparse generalised polynomials.
\newblock {\em Trans. Amer. Math. Soc.}, 370(11):8081--8109, 2018.

\bibitem[BK19]{ByszewskiKonieczny-2019}
J.~Byszewski and J.~Konieczny.
\newblock Automatic sequences and generalised polynomials.
\newblock {\em Canadian Journal of Mathematics}, 2019.
\newblock To appear.

\bibitem[BL07]{BergelsonLeibman-2007}
V.~Bergelson and A.~Leibman.
\newblock Distribution of values of bounded generalized polynomials.
\newblock {\em Acta Math.}, 198(2):155--230, 2007.

\bibitem[CG90]{CorwinGreenleaf-book}
L.~J. Corwin and F.~P. Greenleaf.
\newblock {\em Representations of nilpotent {L}ie groups and their
  applications. {P}art {I}}, volume~18 of {\em Cambridge Studies in Advanced
  Mathematics}.
\newblock Cambridge University Press, Cambridge, 1990.
\newblock Basic theory and examples.

\bibitem[Fur81]{Furstenberg-book}
H.~Furstenberg.
\newblock {\em Recurrence in ergodic theory and combinatorial number theory}.
\newblock Princeton University Press, Princeton, N.J., 1981.
\newblock M. B. Porter Lectures.

\bibitem[HS12]{HindmanStrauss-book}
N.~Hindman and D.~Strauss.
\newblock {\em Algebra in the {S}tone-\v {C}ech compactification}.
\newblock de Gruyter Textbook. Walter de Gruyter \& Co., Berlin, second
  edition, 2012.

\bibitem[Lei12]{Leibman-2012}
A.~Leibman.
\newblock A canonical form and the distribution of values of generalized
  polynomials.
\newblock {\em Israel J. Math.}, 188:131--176, 2012.

\bibitem[TY05]{TauvelYu-book}
P.~Tauvel and R.~W.~T. Yu.
\newblock {\em Lie algebras and algebraic groups}.
\newblock Springer Monographs in Mathematics. Springer-Verlag, Berlin, 2005.

\end{thebibliography}

\end{document}